\documentclass[12pt]{article}
\usepackage[utf8]{inputenc}
\usepackage{fullpage}
\usepackage{xcolor}
\usepackage{authblk}

\usepackage{amsmath,amssymb,amsthm}
\usepackage[normalem]{ulem}
\usepackage{mathtools}

\usepackage[new]{old-arrows}
\usepackage[all]{xy}
\usepackage{hyperref}
\usepackage{tikz}
\usepackage{tikz-cd}
\usepackage{subcaption}

\newtheorem{thm}{Theorem}
\newtheorem*{thm2}{Theorem}

\newtheorem{prop}{Proposition}[section]
\newtheorem{lem}[prop]{Lemma}

\newtheorem{cor}[prop]{Corollary}

\theoremstyle{definition}
\newtheorem{defn}[prop]{Definition}
\newtheorem{remark}[prop]{Remark}
\newtheorem{example}[prop]{Example}
\newtheorem{question}[prop]{Question}
\usepackage{nomencl}
\makenomenclature
\renewcommand{\nompreamble}{The list below describes several symbols that will be used throughout the article.}

\newcommand{\fin}{\ensuremath{\mathrm{fin}}}

\newcommand{\R}{\ensuremath{\mathbb{R}}}

\newcommand{\cP}{\ensuremath{\mathcal{P}}}

\newcommand{\cPfin}{\ensuremath{\mathcal{P}}^\fin}
\newcommand{\cPq}[2]{\ensuremath{\mathcal{P}}_{#1,#2}}

\newcommand{\supp}{\ensuremath{\mathrm{supp}}}
\newcommand{\Sp}{\ensuremath{\mathbb{S}}}

\newcommand{\cpl}{\ensuremath{\mathrm{Cpl}}}

\newcommand{\dWq}{\ensuremath{d_{\mathrm{W},q}^X}}

\newcommand{\dWp}{\ensuremath{d_{\mathrm{W},p}^X}}

\newcommand{\dWqq}[1]{\ensuremath{d_{\mathrm{W},#1}^X}}

\newcommand{\di}{\ensuremath{d_{\mathrm{I}}}}

\newcommand{\dB}{\ensuremath{d_{\mathrm{B}}}}

\newcommand{\vr}[2]{\mathrm{VR}(#1;#2)}

\newcommand{\vrf}[1]{\mathrm{VR}(#1,\smb)}

\newcommand{\vrp}[2]{\mathrm{VR}_{p}(#1;#2)}

\newcommand{\vrpf}[1]{\mathrm{VR}_{p}(#1;\smb)}

\newcommand{\vrpp}[3]{\mathrm{VR}_{#1}(#2;#3)}

\newcommand{\vrppf}[2]{\mathrm{VR}_{#1}(#2;\smb)}

\newcommand{\vrpfin}[2]{\mathrm{VR}_{p}^\fin(#1;#2)}

\newcommand{\vrpffin}[1]{\mathrm{VR}_{p}^\fin(#1;\smb)}

\newcommand{\vrppfin}[3]{\mathrm{VR}_{#1}^\fin(#2;#3)}

\newcommand{\vrppffin}[2]{\mathrm{VR}_{#1}^\fin(#2;\smb)}

\newcommand{\cech}[2]{\mathrm{\check{C}}(#1;#2)}

\newcommand{\cechf}[1]{\mathrm{\check{C}}(#1;\smb)}

\newcommand{\cechp}[2]{\mathrm{\check{C}}_{p}(#1;#2)}

\newcommand{\cechpf}[1]{\mathrm{\check{C}}_{p}(#1;\smb)}

\newcommand{\cechpp}[3]{\mathrm{\check{C}}_{#1}(#2;#3)}

\newcommand{\cechppf}[2]{\mathrm{\check{C}}_{#1}(#2;\smb)}

\newcommand{\acechp}[3]{\mathrm{\check{C}}_{p}(#1,#2;#3)} % For ambient Cech C_p(X,Y;r)

\newcommand{\acechpffin}[2]{\mathrm{\check{C}}_{p}^\fin(#1, #2;\smb)}

\newcommand{\acechpfin}[3]{\mathrm{\check{C}}_{p}^\fin(#1, #2;#3)}

\newcommand{\acechpf}[2]{\mathrm{\check{C}}_{p}(#1, #2;\smb)}

\newcommand{\cechpfin}[2]{\mathrm{\check{C}}_{p}^\fin(#1;#2)}

\newcommand{\cechpffin}[1]{\mathrm{\check{C}}_{p}^\fin(#1;\smb)}

\newcommand{\cechppfin}[3]{\mathrm{\check{C}}_{#1}^\fin(#2;#3)}

\newcommand{\cechppffin}[2]{\mathrm{\check{C}}_{#1}^\fin(#2;\smb)}

\newcommand{\norm}[1]{\left\lVert#1\right\rVert}

\newcommand{\dgm}{\ensuremath{\mathrm{dgm}}}

\newcommand{\dgmvr}{\dgm^\ensuremath{\mathrm{VR}}}

\newcommand{\dgmcech}{\dgm^\ensuremath{\mathrm{\check{C}}}}

\newcommand{\diam}{\mathrm{diam}}

\newcommand{\rad}{\mathrm{rad}}

\newcommand{\dGH}{d_\mathrm{GH}}

\newcommand{\dHT}{d_\mathrm{HT}}

\newcommand{\smb}{\bullet}

\newcommand{\mi}{\mathfrak{i}}

\newcommand{\iqp}{\mi_{q, p}}

\newcommand{\filt}[3]{[#1,#2;#3]}

\newcommand{\filtf}[2]{[#1,#2;\smb]}

\title{The Persistent Topology of Optimal Transport Based Metric Thickenings}

\date{\today}

\begin{document}
\author[1]{Henry Adams}
\author[2]{Facundo M\'emoli}
\author[3]{Michael Moy}
\author[4]{Qingsong Wang}

\affil[1]{Department of Mathematics, 
Colorado State University\\
\texttt{henry.adams@colostate.edu}}

\affil[2]{Department of Mathematics and Department of Computer Science and Engineering,
The Ohio State University\\

\texttt{facundo.memoli@gmail.com}}

\affil[3]{Department of Mathematics, 
Colorado State University\\
\texttt{michael.moy@colostate.edu}}

\affil[4]{Department of Mathematics,
The Ohio State University\\ 
\texttt{wang.8973@osu.edu}}

\maketitle

\begin{abstract}
A metric thickening of a given metric space $X$ is any metric space admitting an isometric embedding of $X$.
Thickenings have found use in applications of topology to data analysis, where one may approximate the shape of a dataset via the persistent homology of an increasing sequence of spaces.
We introduce two new families of metric thickenings, the $p$-Vietoris--Rips and $p$-\v{C}ech metric thickenings for all $1\le p\le \infty$, which include all probability measures on $X$ whose $p$-diameter or $p$-radius is bounded from above, equipped with an optimal transport metric.
The $p$-diameter (resp.\ $p$-radius) of a measure is a certain $\ell_p$ relaxation of the usual notion of diameter (resp.\ radius) of a subset of a metric space.
These families recover the previously studied Vietoris--Rips and \v{C}ech metric thickenings when $p=\infty$.
As our main contribution, we prove a stability theorem for the persistent homology of $p$-Vietoris--Rips and $p$-\v{C}ech metric thickenings, which is novel even in the case $p=\infty$.
In the specific case $p=2$, we prove a Hausmann-type theorem for thickenings of manifolds, and we derive the complete list of homotopy types of the $2$-Vietoris--Rips thickenings of the $n$-sphere as the scale increases.
\end{abstract}

\maketitle

%\newpage
\tableofcontents

%% table of symbols

\nomenclature[01]{$(X,d_X)$}{A metric space.}

\nomenclature[02]{$\Sp^n$}{Sphere endowed with the usual geodesic metric.}

\nomenclature[03]{$Z_n$}{The metric space with $n$ points and all interpoint distances equal to 1.}

\nomenclature[04]{$\diam(A)$}{The diameter of a subset $A$ of a metric space.}

\nomenclature[05]{$\rad(A)$}{The radius of a subset $A$ of a metric space.}

\nomenclature[06]{$\dGH$}{The Gromov--Haussdorff distance.}

\nomenclature[07]{$\cP_X$}{Set of all Radon probability measures on the metric space $X$.}

\nomenclature[08]{$\cPfin_X$}{Set of all finitely supported Radon probability measures on the metric space $X$.}

\nomenclature[09]{$\cPq{q}{X}$}{Set of all Radon probability measures with finite moments of order $q$.}

\nomenclature[10]{$\filtf{\cP_X}{\mi^X}$}{The sublevelset filtration associated to invariant $\mi$.}

\nomenclature[11]{$\filtf{\cPfin_X}{\mi^X}$}{The sublevelset filtration with finite support associated to invariant $\mi$.}

\nomenclature[12]{$\cpl(\alpha,\beta)$}{The set of all couplings between probability measures $\alpha$ and $\beta.$}

\nomenclature[13]{$\dWq$}{The $q$-Wasserstein distance on the metric space $X$.}

\nomenclature[14]{$\diam_p(\alpha)$}{The $p$-diameter of the probability measure $\alpha$.}

\nomenclature[15]{$\rad_p(\alpha)$}{The $p$-radius of the probability measure $\alpha$.}

\nomenclature[16]{$\vrf{X}$}{The standard Vietoris-Rips simplicial complex filtration.}

\nomenclature[17]{$\vrpf{X}$}{The $p$-Vietoris-Rips filtration.}

\nomenclature[18]{$\vrpffin{X}$}{The $p$-Vietoris-Rips filtration with finite support.}

\nomenclature[19]{$\cechf{X}$}{The standard \v{C}ech simplicial complex filtration.}

\nomenclature[20]{$\cechpf{X}$}{The $p$-\v{C}ech metric thickening filtration.}

\nomenclature[21]{$\cechpffin{X}$}{The $p$-\v{C}ech metric thickening filtration with finite support.}

\nomenclature[22]{$\dB$}{Bottleneck distance between persistence diagrams.}

\nomenclature[23]{$\di^\mathcal{C}$}{Interleaving distance between functors $(\R,\leq)\to \mathcal{C}$.}

\nomenclature[24]{$\dgm$}{Persistence diagram for a persistence module.}

\nomenclature[25]{$\dgmvr_{k}$}{Degree $k$-persistence diagram arising from the Vietoris--Rips simplicial complex filtration.}

\nomenclature[26]{$\dgmvr_{k,p}$}{Degree $k$-persistence diagram arising from the $p$-Vietoris--Rips metric thickening.}

\nomenclature[27]{$\dgmcech_{k}$}{Degree $k$-persistence diagram arising from the \v{C}ech simplicial complex filtration.}

\nomenclature[28]{$\dgmcech_{k,p}$}{Degree $k$-persistence diagram arising from the $p$-\v{C}ech metric thickening.}

\nomenclature[29]{$\varepsilon$ and $\delta$}{distortion variables}

\nomenclature[30]{$r, s, t$}{Scale parameters.}

\printnomenclature
\section{Introduction}
\label{sec:intro}

Geometric simplicial complexes, such as Vietoris--Rips or \v{C}ech complexes, are one of the cornerstones of topological data analysis.
One can approximate the shape of a dataset $X$ by building a growing sequence of Vietoris--Rips complexes with $X$ as the underlying set, and then computing persistent homology.
The shape of the data, as measured by persistence, is reflective of important patterns within~\cite{Carlsson2009}.

The popularity of Vietoris--Rips complexes relies on at least three facts.
First, Vietoris--Rips filtrations and their persistent homology signatures are computable~\cite{bauer2021ripser}.
Second, Vietoris--Rips persistent homology is stable~\cite{chazal2009gromov,ChazalDeSilvaOudot2014}, meaning that the topological data analysis pipeline is robust to certain types of noise.
Third, Vietoris--Rips complexes are topologically faithful at low scale parameters: one can use them to recover the homotopy types~\cite{Latschev2001} or homology groups~\cite{ChazalOudot2008} of an unknown underlying space, when given only a finite noisy sampling.

At higher scale parameters, we mostly do not know how Vietoris--Rips complexes behave.
This is despite the fact that one of the key insights of persistent homology is to allow the scale parameter to vary from small to large, tracking the lifetimes of features as the scale increases.
Our practice is ahead of our theory in this regard: data science practitioners are building Vietoris--Rips complexes with scale parameters larger than those for which the reconstruction results of~\cite{Latschev2001,ChazalOudot2008} apply.

Stability implies that as more and more data points are sampled from some ``true'' underlying space $M$, the Vietoris--Rips persistent homology of the dataset $X$ converges to the Vietoris--Rips persistent homology of $M$.
The simplest possible case is when the dataset $X$ is sampled from a manifold $M$, and so we cannot fully understand the Vietoris--Rips persistent homology of data without also understanding the Vietoris--Rips persistent homology of manifolds.
However, not much is known about Vietoris--Rips complexes of manifolds, except at small scales~\cite{Hausmann1995}.
Even the Vietoris--Rips persistent homology of the $n$-sphere $\Sp^n$ is almost entirely unknown.

Two potential obstacles for understanding the homotopy types of Vietoris--Rips complexes of a manifold $M$ are:
\begin{enumerate}
	\item the natural inclusion $M \hookrightarrow \vr{M}{r}$ is not continuous, and
	\item we do not yet have a full Morse theory for Vietoris--Rips complexes of manifolds.
\end{enumerate}

There are by now several strategies for handling these two obstacles.
One strategy is to remain in the setting of Vietoris--Rips simplicial complexes.
Obstacle (1) is then unavoidable.
Regarding obstacle (2), Bestvina--Brady Morse theory has only been successfully applied at low scale parameters, allowing Zaremsky~\cite{zaremsky2019} to prove that $\vr{\Sp^n}{r}$ recovers the homotopy type of $\Sp^n$ for $r$ small enough, but not to derive new homotopy types that appear as $r$ increases.
Simplicial techniques have been considered for a long time, but even successes such as an understanding of the homotopy types of the Vietoris--Rips complexes of the circle at all scales~\cite{AA-VRS1} are not accompanied by a broader Morse theory (though some techniques feel Morse-theoretic).

A second strategy is very recent.
\cite{lim2020vietoris,okutan2019persistence} show that the Vietoris--Rips simplicial complex filtration is equivalent to thickenings of the Kuratowski embedding into $L^\infty(M)$ or any other \emph{injective} metric space: in particular, the two filtrations have the same persistent homology.
This overcomes obstacle (1): the inclusion of a metric space into a thickening of its Kuratowski embedding is continuous, and indeed an isometry onto its image.
This connection has created new opportunities, such as the Morse theoretic techniques employed by Katz~\cite{katz1983filling,katz1989diameter,katz9filling,katz1991neighborhoods}.
This Morse theory allows one to prove the first new homotopy types that occur for Vietoris--Rips simplicial complexes of the circle and 2-sphere, but have not yet inspired progress for larger scales, or for spheres above dimension two.

A third strategy is to consider Vietoris--Rips metric thickenings, which rely on optimal transport and Wasserstein-distances~\cite{AAF}.
We refer to these spaces as the $\infty$-metric thickenings, for reasons that will become clear in the following paragraph.
Such thickenings were invented in order to enable Morse-theoretic proofs of the homotopy types of Vietoris--Rips type spaces.
The first new homotopy type of the $\infty$-Vietoris--Rips metric thickening of the $n$-sphere is known~\cite{AAF}, but only for a single (non-persistent) scale parameter.
It was previously only conjectured that the $\infty$-Vietoris--Rips metric thickenings have the same persistent homology as the more classical Vietoris--Rips simplicial complexes~\cite[Conjecture~6.12]{AAF}; one of our contributions is to answer this conjecture in the affirmative.
\cite{MirthThesis} considers a Morse theory in Wasserstein space, which is inspired in part by applications to $\infty$-Vietoris--Rips metric thickenings, but which does not apply as-is to these thickenings as the $\infty$-diameter functional is not ``$\lambda$-convex''~\cite{santambrogio2017euclidean}.

We introduce a generalization: the $p$-Vietoris--Rips metric thickening for any $1\le p\le \infty$.
Let $X$ be an arbitrary metric space.
For $1\leq p \leq \infty$, the $p$-Vietoris--Rips metric thickening at scale parameter $r> 0$ contains all probability measures on $X$ whose $p$-diameter is less than $r$.
The $p$-Vietoris--Rips metric thickening will be equipped with the topology induced from the weak topology on $\cP_X$.
When $X$ is bounded, the weak topology is generated by an optimal transport based metric; see Corollary~\ref{cor:dwq_generate_weak_topology}.
For $p$ finite, the $p$-diameter of a probability measure $\alpha$ on the metric space $X$ is defined as
\[\diam_p(\alpha):=\left(\iint_{X\times X} d_X^p(x,x')\,\alpha(dx)\,\alpha(dx')\right)^{1/p},\]
and $\diam_\infty(\alpha)$ is defined to be the diameter of the support of $\alpha$.

The $p$-Vietoris--Rips metric thickenings at scale $r$ form a metric bifiltration of $X$ that is covariant in $r$ and contravariant in $p$.
Indeed, we have an inclusion map $\vrpp{p}{X}{r} \hookrightarrow \vrpp{p'}{X}{r'}$ for $r\le r'$ and $p \ge p'$; see Figure~\ref{fig:vrpBifiltration}.
\[
\xymatrix{
    \vrpp{p'}{X}{r} \ar@{^{(}->}[r] & \vrpp{p'}{X}{r'} \\
    \vrp{X}{r} \ar@{^{(}->}[r]\ar@{^{(}->}[u] &  \vrp{X}{r'}\ar@{^{(}->}[u]
	}
\]

\begin{figure}
\centering
\includegraphics[width=0.9\textwidth]{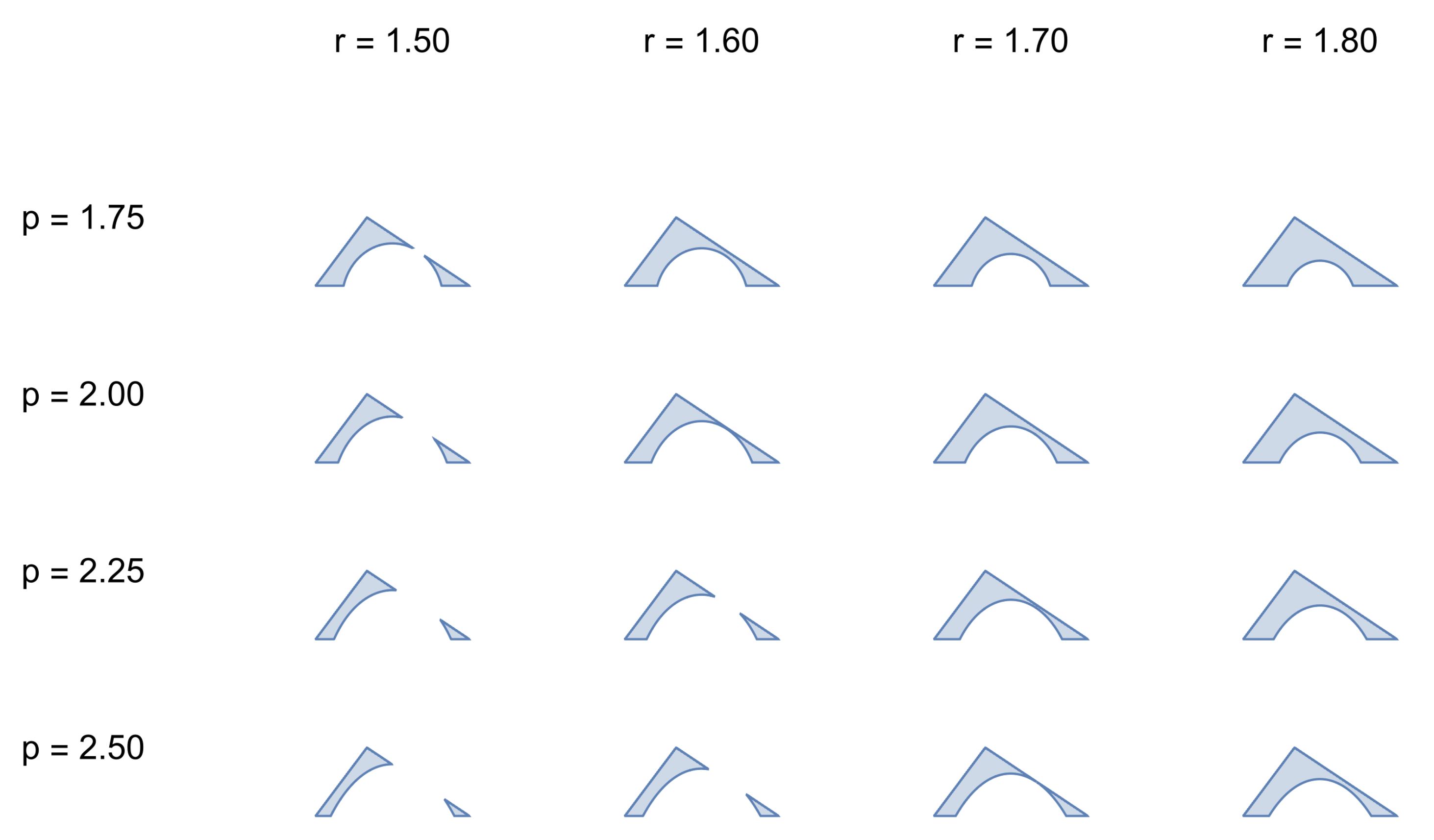}
\caption{The $p$-Vietoris--Rips bifiltration $\vrp{X}{r}$ for $X$ a metric space of three points in $\R^2$, with $\cP_X$ visualized as the convex hull of $X$ in $\R^2$.
Note $\vrpp{p}{X}{r}\subseteq\vrpp{p'}{X}{r'}$ for $r\le r'$ and $p \ge p'$.}
\label{fig:vrpBifiltration}
\end{figure}

\begin{figure}
\centering
\includegraphics[width=0.9\textwidth]{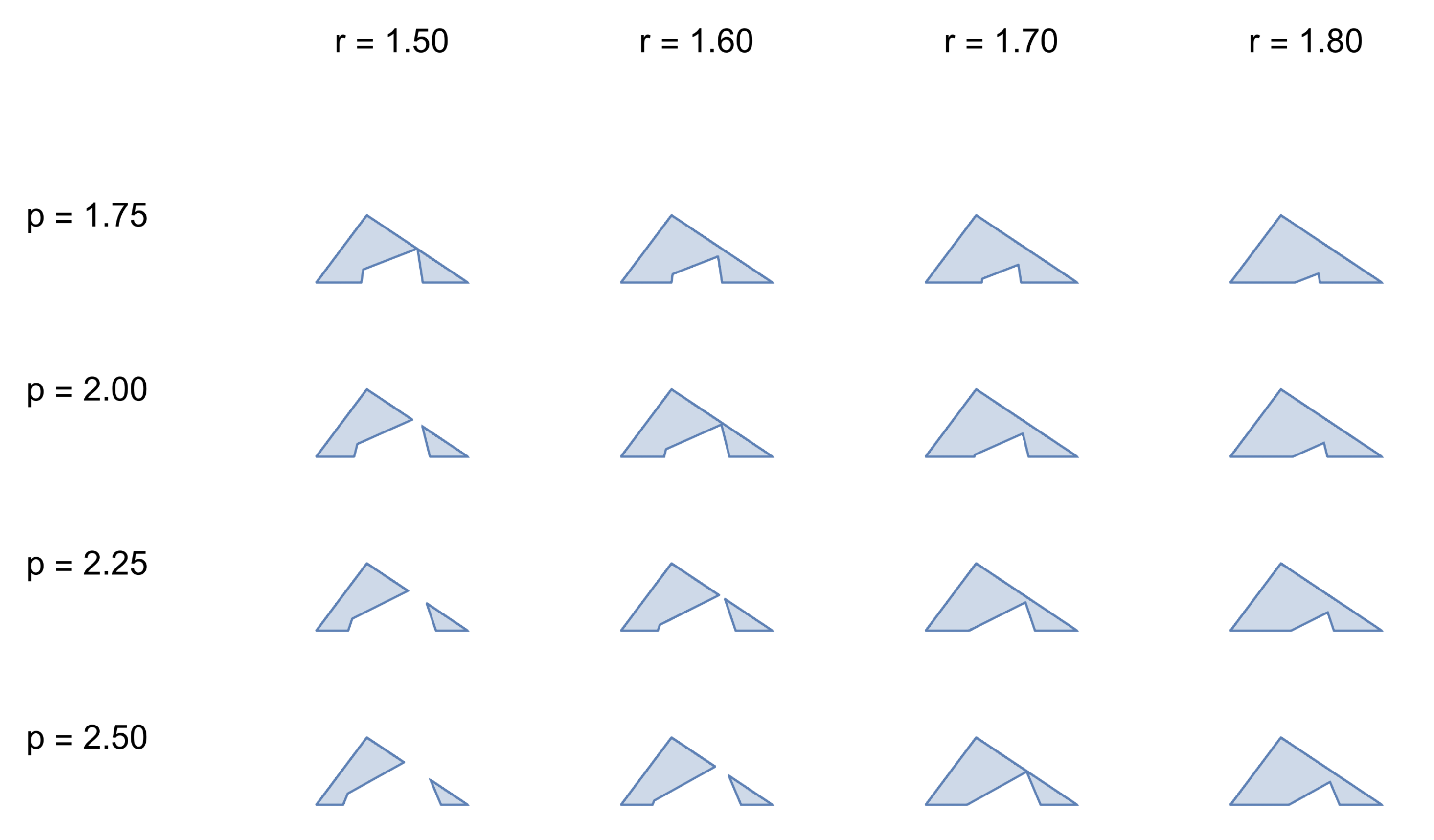}
\caption{
The $p$-\u{C}ech bifiltration $\cechp{X}{r}$ for $X$ a metric space of three points in $\R^2$, with $\cP_X$ visualized as the convex hull of $X$ in $\R^2$.
Note $\cechpp{p}{X}{r}\subseteq\cechpp{p'}{X}{r'}$ for $r\le r'$ and $p \ge p'$.}
\label{fig:cechpBifiltration}
\end{figure}

As one of our main contributions, we prove that the $p$-Vietoris--Rips metric thickening is stable.
This means that if two totally bounded metric spaces $X$ and $Y$ are close in the Gromov--Hausdorff distance, then their filtrations $\vrpf{X}$ and $\vrpf{Y}$ are close in the homotopy interleaving distance.
This was previously unknown even in the case $p=\infty$ (see~\cite[Conjecture~6.14]{AAF}); we prove stability for all $1\le p\le \infty$.
As a consequence, it follows that the (undecorated) persistent homology diagrams for the Vietoris--Rips simplicial complexes $\vr{X}{r}$ and for the $p=\infty$ Vietoris--Rips metric thickenings $\vrppf{\infty}{X}$ are identical.
In other words, the persistent homology barcodes for $\vrf{X}$ and $\vrppf{\infty}{X}$ are identical up to replacing closed interval endpoints with open endpoints, or vice-versa.
This answers~\cite[Conjecture~6.12]{AAF} in the affirmative.
Another consequence of stability is that the $p$-metric thickenings give the same persistence diagrams whether one considers all Radon probability measures, or instead the restricted setting of only measures with finite support.

The proof of stability for metric thickenings is more intricate than the proof of stability for the corresponding simplicial complexes (see for instance~\cite{chazal2009gromov,ChazalDeSilvaOudot2014,memoli2017distance}).
Whereas simplicial complexes can be compared via simplicial maps, the direct analogues of simplicial maps on metric thickenings are not necessarily continuous.
Thus, new techniques are required to construct continuous maps between metric thickenings.
Our technique, relying on partitions of unity, allows us to continuously approximate measures on one metric space $X$ by measures on another metric space $Y$ (where this approximation depends on the Gromov--Hausdorff distance between $X$ and $Y$), thereby allowing us to construct the desired interleavings.

We furthermore introduce the $p$-\v{C}ech metric thickenings, again for $1\le p \le \infty$, and prove analogues of all of the above results.
The $p$-\v{C}ech metric thickening at scale parameter $r> 0$ contains all probability measures supported on $X$ whose $p$-radius is less than $r$; see Figure~\ref{fig:cechpBifiltration}.

We also deduce the complete spectrum of homotopy types of $2$-Vietoris--Rips and $2$-\v{C}ech metric thickenings of the $n$-sphere, equipped with the Euclidean metric $\ell_2$:
$\vrpp{2}{(\Sp^n,\ell_2)}{r}$ first attains the homotopy type of $\Sp^n$ for scales $r\le \sqrt{2}$, and then for all scales $r>\sqrt{2}$ the space is contractible.
By contrast, the Vietoris--Rips simplicial complexes or the $\infty$-Vietoris--Rips metric thickenings of the $n$-sphere (with either the Euclidean or the geodesic metric) are only known for a bounded range of scales, including only a single change in homotopy type (\cite[Section~5]{AAF}),
even though infinitely many changes in homotopy type are conjectured (c.f.~\cite[Question~8.1]{ABF2}).
See however,~\cite[Corollary 7.18]{lim2020vietoris} for results for round spheres with the $\ell_\infty$ metric.

One of our main motivations for introducing the $p$-Vietoris--Rips and $p$-\v{C}ech metric thickenings is to enable effective Morse theories on these types of spaces.
The $p$-variance for \v{C}ech metric thickenings is a minimum of linear functionals, and therefore fits in the framework of Morse theory for min-type functions~\cite{baryshnikov2014min,gershkovich1997morse,bryzgalova1978maximum,matov1982topological}.
On the Vietoris--Rips side, we remark that gradient flows of functionals on Wasserstein space can be defined when the functional is ``$\lambda$-convex''~\cite{santambrogio2017euclidean,MirthThesis}.
Though the $\infty$-diameter is not $\lambda$-convex, we hope that the $p$-diameter functional for $p<\infty$ may be $\lambda$-convex in certain settings.

\paragraph*{Organization.}
In Section~\ref{sec:background} we describe background material and set notation.
We define the $p$-Vietoris--Rips and $p$-\v{C}ech metric thickenings in Section~\ref{sec:p-relaxation}, and consider their basic properties in Section~\ref{sec:basic-properties}.
In Section~\ref{sec:stability} we prove stability.
We consider Hausmann-type theorems in Section~\ref{sec:hausmann}, and deduce the $2$-Vietoris--Rips metric thickenings of Euclidean spheres in Section~\ref{sec:spheres}.
In Section~\ref{app:spread} we bound the length of intervals in $p$-Vietoris--Rips and $p$-\v{C}ech metric thickenings using a generalization of the spread of a metric space, called the $p$-spread.

We conclude the paper by providing some discussion in Section~\ref{sec:conclusion}.

In Appendix~\ref{app:Metrization_weak_topology} we explain how the $q$-Wasserstein distance metrizes the weak topology for $1\le q<\infty$.
We describe connections to min-type Morse theories in Appendix~\ref{app:Morse}.
In Appendix~\ref{app:finite-Cech} we show that $p$-\v{C}ech thickenings of finite metric spaces are homotopy equivalent to simplicial complexes.
We prove the persistent homology diagrams of the $p$-Vietoris--Rips and $p$-\v{C}ech metric thickenings of a family of discrete metric spaces in Appendix~\ref{app:pd_delta_1_space}, and we describe the $0$-dimensional persistent homology of the $p$-Vietoris--Rips and $p$-\v{C}ech metric thickenings of an arbitrary metric space in Appendix~\ref{app:PH0}.
We consider crushings in Appendix~\ref{app:crushings}.
In Appendix~\ref{app:ambient}, we show that the main properties we prove for the (intrinsic) $p$-\v{C}ech metric thickening also hold for the ambient $p$-\v{C}ech metric thickening.

\section{Background}
\label{sec:background}

This section introduces background material and notation.

\paragraph{Metric spaces and the Gromov--Hausdorff distance.}
Let $(X,d_X)$ be a metric space.
For any $x\in X$, we let $B(x;r):=\{y\in X~|~d_X(x,y)<r\}$ denote the open ball centered at $x$ of radius $r$.

Given a metric space $(X,d_X)$, the \emph{diameter} of a non-empty subset $A\subset X$ is the number $\diam(A):=\sup_{a,a'\in A}d_X(a,a')$, whereas its \emph{radius} is the number $\rad(A):=\inf_{x\in X}\sup_{a\in A}d_X(x,a)$.
Note that in general we have \[\frac{\diam(A)}{2}\leq \rad(A)\leq \diam(A).\]

\begin{defn}[Uniform discrete metric space]
For any natural number $n$, we use $Z_n$ to denote the metric space consisting of $n$ points with all interpoint distances equal to 1.
\end{defn}

\begin{defn}[$\varepsilon$-net]
Let $X$ be a metric space.
A subset $U\subset X$ is called an \emph{$\varepsilon$-net} of $X$ if for any point $x\in X$, there is a point $u\in U$ with $d_X(x, u)< \varepsilon$.
\end{defn}

Let $X$ and $Y$ be metric spaces.
The \emph{distortion} of an arbitrary function $\varphi\colon X\to Y$ is \[\mathrm{dis}(\varphi):=\sup_{x,x'\in X}|d_X(x,x')-d_Y(\varphi(x),\varphi(x'))|.\]
The \emph{codistortion} of a pair of arbitrary functions $\varphi\colon X\to Y$, $\psi\colon Y\to X$ is
\[\mathrm{codis}(\varphi,\psi):=\sup_{x\in X,y\in Y}|d_X(x,\psi(y))-d_Y(\varphi(x),y)|.\]
We will use the following expression for the Gromov--Hausdorff distance between $X$ and $Y$~\cite{banach}: 
\begin{equation}\label{eq:def-dgh}\dGH(X,Y) = \frac{1}{2}\inf_{\varphi,\psi}\max\big(\mathrm{dis}(\varphi),\mathrm{dis}(\psi),\mathrm{codis}(\varphi,\psi)\big).
\end{equation}

\begin{remark}\label{remark:codis}
Note that if $\mathrm{codis}(\varphi,\psi)<\eta$, then for all $(x,y)\in X\times Y$,
\[|d_X(x,\psi(y))-d_Y(\varphi(x),y)|<\eta.\]
In particular, by letting $y=\varphi(x)$ above, this means that
\[d_X(x,\psi\circ\varphi(x))<\eta\quad\text{for all } x\in X.\]
\end{remark}

See~\cite[Chapter 7]{bbi} for more details about the Gromov--Hausdorff distance.

\paragraph{Simplicial complexes.}
Two of the most commonly used methods for producing filtrations in applied topology are the Vietoris--Rips and \v{C}ech simplicial complexes, defined as follows.
Let $X$ be a metric space, and let $r\ge0$.
The \emph{Vietoris--Rips simplicial complex $\vr{X}{r}$} has $X$ as its vertex set, and contains a finite subset $[x_0,\ldots,x_k]$ as a simplex when $\diam([x_0,\ldots,x_k]):=\max_{i,j}d_X(x_i,x_j)<r$.
The \emph{\v{C}ech simplicial complex $\cech{X}{r}$} has $X$ as its vertex set, and contains a finite subset $[x_0,\ldots,x_k]$ as a simplex when $\cap_{i=0}^k B(x_i;r)\neq\emptyset$.

\paragraph{Probability measures and Wasserstein distances.}
Our main reference for measure theory elements is~\cite{bogachev2018weak} (albeit we use slightly different notation).
Given a metric space $(X,d_X)$, by $\cP_X$ we denote the set of all \emph{Radon} probability measures on $X$.
We equip $\cP_X$ with the weak topology: a sequence $\alpha_1, \alpha_2, \alpha_3, \ldots \in \cP_X$ is said to converge weakly to $\alpha\in \cP_X$ if for all bounded, continuous functions $\varphi\colon X\to \R$, we have $\lim_{n \to \infty} \int_X \varphi(x)\,\alpha_n(dx) = \int_X \varphi(x)\,\alpha(dx)$.
The \emph{support} $\supp(\alpha)$ of a probability measure $\alpha\in\cP_X$ is the largest closed set $C$ such that every open set which has non-empty intersection with $C$ has positive measure.
If $\supp(\alpha)$ consists of a finite set of points, then $\alpha$ is called \emph{finitely supported} and can be written as $\alpha = \sum_{i\in I} a_i\cdot\delta_{x_i}$, where $I$ is finite, $a_i \geq 0$ for all $i$, $\sum_{i \in I} a_i = 1$, and each $\delta_{x_i}$ is a Dirac delta measure at $x_i$.
Let $\cPfin_X$ denote the set of all finitely supported Radon probability measures on $X$.

Given another metric space $Y$ and a measurable map $f\colon X\to Y$, the \emph{pushforward} map $f_\sharp \colon \cP_X\to \cP_Y$ induced by $f$ is defined by $f_\sharp (\alpha)(B)=\alpha(f^{-1}(B))$ for every Borel set ${B\subset Y}$.
In the case of finitely supported probability measures, we have the explicit formula $f_\sharp \left( \sum_{i \in I} a_i\cdot\delta_{x_i} \right) = \sum_{i \in I} a_i\cdot \delta_{f(x_i)}$, so $f_\sharp$ restricts to a function $\cPfin_X \to \cPfin_Y$.
If $f$ is a continuous map, then $f_{\sharp}$ is a continuous map between $\cP_X$ and $\cP_Y$ in the weak topology; see Chapter~5 in~\cite{bogachev2018weak}.
In the finitely supported case, the restriction of $f_\sharp$ is a continuous map from $\cPfin_X$ to $\cPfin_Y$.

Given $\alpha,\beta\in\cP_X$, a \emph{coupling} between them is any probability measure $\mu$ on $X\times X$ with marginals $\alpha$ and $\beta$: $(\pi_1)_\sharp\mu=\alpha$ and $(\pi_2)_\sharp \mu= \beta$, where $\pi_i\colon X\times X\to X$ is the projection map defined by $\pi_i(x_1,x_2)=x_i$ for $i=1,2$.
By $\cpl(\alpha,\beta)$ we denote the set of all couplings between $\alpha$ and $\beta$.
Notice that $\cpl(\alpha,\beta)$ is always non-empty as the product measure $\alpha\otimes \beta$ is in $\cpl(\alpha,\beta)$.

Given $q\in[1,\infty)$, let $\cPq{q}{X}$ be the subset of $\cP_X$ consisting of Radon probability measures with finite moments of order $q$, that is, measures $\alpha$ with $\int_X d_X^q(x,x_0)\, \alpha(dx) < \infty$ for some, and thus any, $x_0\in X$.
Note that when $X$ is a bounded metric space, we have $\cPq{q}{X} = \cP_X$.
We can equip $\cPq{q}{X}$ with the $q$-\emph{Wasserstein distance} (or \emph{Kantorovich} $q$-metric).
In this setting, the $q$-Wasserstein distance is given by
\begin{equation}\label{eq:dW}
    \dWq(\alpha,\beta):= \inf_{\mu\in\cpl(\alpha,\beta)} \left(\iint_{X\times X}d_X^q(x,x')\,\mu(dx\times dx')\right)^{\frac{1}{q}}
\end{equation}
and it can be shown that $\dWq$ defines a metric on $\cPq{q}{X}$; see Chapter~3.3 in~\cite{bogachev2018weak}.
If $q = \infty$, then again for any metric space $(X,d_X)$ we define the \emph{$\infty$-Wasserstein distance}
\begin{equation}\label{eq:dW_infty}
	\dWqq{\infty}(\alpha,\beta):=
	\inf_{\mu\in\cpl(\alpha,\beta)} \sup_{(x,x')\in\supp(\mu)}d_X(x,x')
\end{equation}
for $\alpha, \beta \in \cPq{\infty}{X}$, the set of measures with bounded support.
The $\infty$-Wasserstein distance $\dWqq{\infty}$ is clearly symmetric, and the triangle inequality comes from a similar gluing trick as in Chapter~3.3 of~\cite{bogachev2018weak}.
Thus, $\dWqq{\infty}$ is a pseudo-metric that is lower bounded by a metric, for example $\dWqq{1}$; therefore $\dWqq{\infty}$ is also a metric.
For general $q \le q'$, it follows from H\"{o}lder's inequality that $\dWq \le \dWqq{q'}$.

It is known that on a bounded metric space, for any $q \in [1, \infty)$, the $q$-Wasserstein metric generates the weak topology.
For a summary of the result, see Appendix~\ref{app:Metrization_weak_topology}.
On the other hand, the $\infty$-Wasserstein metric $\dWqq{\infty}$ generates a finer topology than the weak topology in general.

Let $f,g \colon X\to \R$ be measurable functions over the metric space $X$, and let $p\in[1,\infty)$.
We will frequently use the Minkowski inequality, which states that
\[\left(\int_X |f(x) + g(x)|^p\,\alpha(dx)\right)^{\frac{1}{p}}\leq \left(\int_X |f|^p(x)\,\alpha(dx)\right)^{\frac{1}{p}} + \left(\int_X |g|^p(x)\,\alpha(dx)\right)^{\frac{1}{p}}.\]

The following proposition shows that we may construct linear homotopies in spaces of probability measures, analogous to linear homotopies in Euclidean spaces.
Linear homotopies will play an important role in Section~\ref{sec:stability}.

\begin{prop}\label{prop:linear_homotopies}
Let $Z$ be a metric space (or more generally a first-countable space), let $X$ be a metric space, and let $f,g \colon Z \to \cP_X$ be continuous.
Then the linear homotopy $H \colon Z \times [0,1] \to \cP_X$ given by $H(z,t) = (1-t)\,f(z) + t\,g(z)$ is continuous.
\end{prop}

\begin{proof}
It suffices to show sequential continuity, because $Z \times I$ is metrizable (or more generally first-countable)
and the weak topology on the set $\cP_X$ is also metrizable; see~\cite[Theorem~3.1.4]{bogachev2018weak} and also Appendix~\ref{app:Metrization_weak_topology}.
If $\{ (z_n, t_n) \}_n$ converges to $(z,t)$ in $Z \times I$, then to show weak convergence of the image in $\cP_X$, let $\gamma \colon X \to \mathbb{R}$ be any bounded, continuous function.
We have
\begin{align*}
\lim_{n \to \infty} \int_X \gamma(x) H(z_n,t_n)(dx) & = \lim_{n \to \infty} \left( (1-t_n)\int_X \gamma(x) f(z_n)(dx) + t_n \int_X \gamma(x) g(z_n)(dx) \right)\\
& = (1-t)\int_X \gamma(x) f(z)(dx) + t \int_X \gamma(x) g(z)(dx)\\
& = \int_X \gamma(x) H(z,t)(dx)
\end{align*}
where the second equality uses the fact that $f(z_n)$ and $g(z_n)$ are weakly convergent, since $f$ and $g$ are continuous.
Therefore $H(z_n,t_n)$ converges weakly to $H(z,t)$, so $H$ is continuous.
\end{proof}

\paragraph{Fr\'{e}chet means.}

For each $p\in[1,\infty]$ let the \emph{$p$-Fr\'{e}chet function of $\alpha\in \cP_X$}, namely $F_{\alpha,p}\colon X\to \R\cup\{\infty\}$, be defined by
\[F_{\alpha,p}(x):=
\begin{dcases}
\left(\int_X d_X^p(z,x)\,\alpha(dz)\right)^{\frac{1}{p}} & p<\infty \\
\\
\sup_{z\in \supp(\alpha)} d_X(x,z)                       & p=\infty.
\end{dcases}
\]

Note that $F_{\alpha,p}(x)=\dWp(\delta_x,\alpha)$.
A point $x\in X$ that minimizes $F_{\alpha,p}(x)$ is called a \emph{Fr\'{e}chet mean} of $\alpha$; in general Fr\'{e}chet means need not be unique.
See~\cite{Karcher1977} for some of the basic properties of Fr\'{e}chet means.

\paragraph{Metric thickenings with $p=\infty$.}

Let $X$ be a bounded metric space.
The Vietoris--Rips and \v{C}ech metric thickenings were introduced in~\cite{AAF} with the notation $\mathrm{VR}^m(X;r)$ and $\mathrm{\check{C}}^m(X;r)$, where the superscript $m$ denoted ``metric.''
We instead denote these spaces by $\vrppfin{\infty}{X}{r}$ and $\cechppfin{\infty}{X}{r}$, since one of the main contributions of this paper will be to introduce the generalizations $\vrp{X}{r}$ and $\cechp{X}{r}$ (and their finitely supported variants $\vrpfin{X}{r}$ and $\cechpfin{X}{r}$) for any $1\le p\le \infty$.

The \emph{Vietoris--Rips metric thickening $\vrppfin{\infty}{X}{r}$} is the space of all finitely supported probability measures of the form $\sum_{i=0}^k a_i \delta_{x_i}$, such that $\diam(\{x_0,\ldots,x_k\})<r$, equipped with the $q$-Wasserstein metric for some $1\le q< \infty$.
The choice of $q\in [1,\infty)$ does not affect the homeomorphism type by Corollary~\ref{cor:dwq_generate_weak_topology}.
The \emph{\v{C}ech metric thickening $\cechppfin{\infty}{X}{r}$} is the space of all finite probability measures of the form $\sum_{i=0}^k a_i \delta_{x_i}$, such that $\cap_{i=0}^k B(x_i;r)\neq\emptyset$, equipped with the $q$-Wasserstein metric for some $1\le q< \infty$.

\paragraph{Comparisons.}
We give a brief survey of the various advantages and disadvantages of using simplicial complexes and $p=\infty$ metric thickenings.

The Vietoris--Rips and \v{C}ech simplicial complexes were developed first, and they enjoy the benefits of simplicial and combinatorical techniques.
For this reason, these complexes (and related complexes) have been used in (co)homology theories for metric spaces, and in discrete versions of homotopy theories~\cite{Vietoris27,roe1993coarse,berestovskii2007uniform,dranishnikov2009cohomological,cencelj2012combinatorial,brodskiy2013rips,plaut2013discrete,BarceloCapraroWhite,brazas2014thick,conant2014discrete,memoli2019persistent,rieser2020semiuniform}.
The persistent homology of Vietoris--Rips complexes on top of finite metric spaces can be efficiently computed~\cite{bauer2021ripser}.
Even when built on top of infinite metric spaces, much is known about the theory of Vietoris--Rips simplicial complexes using Hausmann's theorem~\cite{Hausmann1995}, Latschev's theorem~\cite{Latschev2001}, the stability of Vietoris--Rips persistent homology~\cite{ChazalDeSilvaOudot2014,chazal2009gromov,cohen2007stability}, Bestvina--Brady Morse theory~\cite{zaremsky2019}, and the Vietoris--Rips complexes of the circle~\cite{AA-VRS1}.

Despite this rich array of simplicial techniques, there are some key disadvantages to Vietoris--Rips simplicial complexes.
If $X$ is not a discrete metric space, then the inclusion from $X$ into $\vr{X}{r}$ for any $r\ge0$ is not continuous since the vertex set of a simplicial complex is equipped with the discrete topology.
Another disadvantage is that even though we start with a metric space $X$, the Vietoris--Rips simplicial complex $\vr{X}{r}$ may not be metrizable.
Indeed, a simplicial complex is metrizable if and only if it is locally finite~\cite[Proposition~4.2.16(2)]{sakai2013geometric}, i.e.\ if and only if each vertex is contained in only a finite number of simplices.
So when $X$ is infinite, $\vr{X}{r}$ is often not metrizable, i.e.\ its topology cannot be induced by any metric.
Though Vietoris--Rips complexes accept metric spaces as input, they do not remain in this same category, and may produce as output topological spaces that cannot be equipped with any metric structure.

By contrast, the metric thickening $\vrppfin{\infty}{X}{r}$ is always a metric space that admits an isometric embedding $X\hookrightarrow \vrppfin{\infty}{X}{r}$~\cite{AAF}.
Furthermore, metric thickenings allow for nicer proofs of Hausmann's theorem.
For $M$ a Riemannian manifold and $r>0$ sufficiently small depending on the curvature of $M$, Hausmann produces a map $T\colon \vr{M}{r}\to M$ from the simplicial complex to the manifold that is not canonical in the sense that it depends on a total order of all points in the manifold.
Also, the inclusion map $M\hookrightarrow\vr{M}{r}$ is not continuous, and therefore cannot be a homotopy inverse for $T$.
Nevertheless, Hausmann is able to prove $T$ is a homotopy equivalence without constructing an explicit inverse.
By contrast, in the context of metric thickenings, one can produce a canonical map $\vrppfin{\infty}{M}{r}\to M$ by mapping a measure to its Fr\'{e}chet mean (whenever $r$ is small enough so that measures of diameter less than $r$ have unique Fr\'{e}chet means).
The (now continuous) inclusion $M\hookrightarrow\vrppfin{\infty}{M}{r}$ can be shown to be a homotopy inverse via linear homotopies~\cite[Theorem~4.2]{AAF}.

One of the main contributions of this paper is showing how these various spaces relate to each other, especially when it comes to persistent homology.
This allows one to work either simplicially, geometrically, or with measures --- whichever perspective is most convenient for the task at hand.

\paragraph{Homology and Persistent homology.}
For each integer $k\geq 0$, let $H_k$ denote the singular homology functor from the category $\mathrm{Top}$ of topological spaces to the category $\mathrm{Vec}$ of vector spaces and linear transformations.
We use coefficients in a fixed field, so that homology groups are furthermore vector spaces.
For background on persistent homology, we refer the reader to~\cite{EdelsbrunnerHarer,edelsbrunner2000topological,zomorodian2005computing}.

In applications of topology, such as topological data analysis~\cite{Carlsson2009}, one often models a dataset not as a single space $X$, but instead as an increasing sequence of spaces.
We refer to an increasing sequence of spaces, i.e.\ a functor from the poset $(\R, \leq)$ to $\mathrm{Top}$, as a \emph{filtration}.
If $X$ is a metric space, then a common filtration is the Vietoris-Rips simplicial complex filtration $\vrf{X}$.
In this paper we will introduce relaxed versions, the $p$-Vietoris--Rips metric thickening filtrations $\vrpf{X}$ (see Section~\ref{sec:p-relaxation}).

By applying homology (with coefficients in a field) to a filtration, we obtain a functor from the poset $(\R, \leq)$ to $\mathrm{Vec}$, i.e.\ a \emph{persistence module}.
We will use symbols like $V$, $W$ and so on to denote persistence modules.
Following~\cite{chazal2016structure}, a persistence module is \emph{Q-tame} if for any $s<t$, the structure map $V(s)\to V(t)$ is of finite rank.
In~\cite{chazal2016structure}, it is shown that a Q-tame persistence module $V$ can be converted into a \emph{persistence diagram}\footnote{Specifically, we consider \emph{undecorated} persistence diagrams, as described in~\cite{chazal2016structure}.}, $\dgm(V)$, which is a multiset in the extended open half-plane consisting of pairs $p = (b, d),\,-\infty \leq b < d \leq + \infty.$
The persistence diagrams can be compared via the bottleneck distance $\dB$, which is defined as follows.
Given persistence diagrams $D_1$ and $D_2$, a subset $M\subset D_1\times D_2$ is said to be a \emph{partial matching} between $D_1$ and $D_2$ if the following are satisfied:
\begin{itemize}
	\item For every point $p$ in $D_1$, there is at most one point $q$ in $D_2$ such that $(p, q)\in M$.
(If there is no such $q$, we then say $p$ is unmatched.)
	\item For every point $q$ in $D_2$, there is at most one point $p$ in $D_1$ such that $(p, q)\in M$.
(If there is no such $p$, we then say $q$ is unmatched.)
\end{itemize}
The \emph{bottleneck distance $\dB$} between two persistence diagrams $D_1$ and $D_2$ is
\[
	\dB(D_1, D_2) := \inf_M \max \left\{ \sup_{(p, q)\in M}\lVert p-q \rVert_\infty, \quad \sup_{s\in D_1 \sqcup D_2 \text{ unmatched}} \left|\frac{s_b-s_d}{2}\right| \right\},
\]
where $s$ is an element in $D_1$ or $D_2$, $s_b$ is the birth time of $s$, $s_d$ is the death time of $s$, and $M$ varies among all possible partial matchings.

\paragraph{Interleavings}

Let $\mathcal{C}$ be a category.
We call any functor from the poset $(\R, \leq)$ to $\mathcal{C}$ \emph{an $\R$-space}.
Such a functor gives a \emph{structure map} $X(s)\to X(t)$ for any $s\le t$.
For $X$ an $\R$-space, the \emph{$\delta$-shift} of $X$ is the functor $X^\delta\colon \R\to C$ with $X^\delta(t)= X(t+\delta)$ for all $t\in \R$.
We have a natural transformation $\mathrm{id}_X^\delta$ from $X$ to $X^\delta$ which maps $X(t)$ to $X^\delta(t)$ using the structure maps from $X$.
For two $\R$-spaces $X$ and $Y$, we say they are \emph{$\delta$-interleaved} if there are natural transformations $F \colon X\to Y^\delta$ and $G\colon Y\to X^\delta$ such that $F\circ G = \mathrm{id}_X^{2\delta}$ and $G\circ F = \mathrm{id}_Y^{2\delta}$.
Now we define a pseudo-distance between $\R$-spaces, which we call the \emph{interleaving distance $\di^\mathcal{C}$}, as follows:
\[\di^\mathcal{C}(X, Y) := \inf\{\delta~|~X\text{ and }Y\text{ are }\delta\text{-interleaved}\}.\]

We note that $\di^{\mathrm{Vec}}$ is the interleaving distance for persistence modules.
The isometry theorem of~\cite{lesnick2015theory,chazal2016structure} states that $\di^{\mathrm{Vec}}(V,W) = \dB(\dgm(V),\dgm(W))$ for Q-tame persistence modules $V$ and $W$.

We will use the following lemma in our proofs in Section~\ref{sec:stability} on stability.

\begin{lem}\label{lem:qtame_approximation}
If a persistence module $P$ can be approximated arbitrarily well in the interleaving distance by $Q$-tame persistence modules, then $P$ is also $Q$-tame.
\end{lem}

\begin{proof}
For any $s < t$, let $\varepsilon = t - s$, so there is a Q-tame persistence module $P_\varepsilon$ such that $\di(P, P_\varepsilon) < \frac{\varepsilon}{3}$.
Then using the maps of an $\frac{\varepsilon}{3}$-interleaving between $P$ and $P_\varepsilon$,
the structure map $P(s)\to P(t)$ can be factored as follows.

\setlength\mathsurround{0pt}%<<<<<<<<<<<<<<<<<<<<<<<<
\[
\begin{tikzcd}
P(s) \arrow[rd] \arrow[rrr] & & & P(t)\\
& P_\varepsilon\left(s + \frac{\varepsilon}{3}\right) \arrow[r] & P_\varepsilon\left(s + \frac{2\varepsilon}{3}\right) \arrow[ru] & 
\end{tikzcd}
\]
As $P_\varepsilon$ is a Q-tame module, the rank of $P_\varepsilon\left(s + \frac{\varepsilon}{3}\right) \to P_\varepsilon(s + \frac{2\varepsilon}{3})$ is finite, and hence so is the rank of $P(s)\to P(t)$.
Since $s$ and $t$ are arbitrary, we get the Q-tameness of $P$.
\end{proof}

\paragraph{The homotopy type distance}

We next recall the definition of the homotopy type distance $\dHT$ distance from~\cite{frosini2017persistent}.
The following is a small generalization of~\cite[Definition 2.2]{frosini2017persistent} in that we do not require the maps $\varphi_X$ and $\varphi_Y$ to be continuous.

\begin{defn}
\label{defn:delta-homotopy}
Let $(X, \varphi_X)$ and $(Y, \varphi_Y)$ be two topological spaces with real-valued functions on them.
For any $\delta \ge 0$, a \emph{$\delta$-map} between $(X, \varphi_X)$ and $(Y, \varphi_Y)$ is a continuous map $\Phi\colon X\to Y$ such that $\varphi_Y\circ \Phi(x)\leq \varphi_X(x) + \delta$ for any $x\in X$.
For any two $\delta$-maps $\Phi_0\colon X\to Y$ and $\Phi_1\colon X\to Y$, a \emph{$\delta$-homotopy} between $\Phi_0$ and $\Phi_1$ with respect to the pair $(\varphi_X, \varphi_Y)$ is a continuous map $H\colon X\times [0, 1] \to Y$ such that
\begin{enumerate}
	\item $\Phi_0 \equiv H(\smb, 0)$;
	\item $\Phi_1 \equiv H(\smb, 1)$;
	\item $H(\smb, t)$ is a $\delta$-map with respect to the pair $(\varphi_X, \varphi_Y)$ for every $t\in [0, 1]$.
\end{enumerate}
\end{defn}

\begin{defn}[{\cite[Definitions~2.5 and~2.6]{frosini2017persistent}}]
\label{defn:dHT}
For every $\delta\geq 0$ and for any two pairs $(X, \varphi_X)$ and $(Y, \varphi_Y)$, we say $(X, \varphi_X)$ and $(Y, \varphi_Y)$ are \emph{$\delta$-homotopy equivalent} if there exist $\delta$-maps $\Phi\colon X\to Y$ and $\Psi\colon Y\to X$, with respect to $(\varphi_X, \varphi_Y)$ and $(\varphi_Y, \varphi_X)$, respectively, such that
\begin{itemize}
	\item the map $\Psi\circ \Phi\colon X\to X$ is $2\delta$-homotopic to $\mathrm{id}_X$ with respect to $(\varphi_X, \varphi_X)$, and
	\item the map $\Phi\circ \Psi\colon Y\to Y$ is $2\delta$-homotopic to $\mathrm{id}_Y$ with respect to $(\varphi_Y, \varphi_Y)$.
\end{itemize}
The \emph{$\dHT$-distance} between $(X, \varphi_X)$ and $(Y, \varphi_Y)$ is
\[\dHT((X, \varphi_X), (Y, \varphi_Y)):= \inf\{\delta\geq 0\mid \text{$(X, \varphi_X)$ and $(Y, \varphi_Y)$ are $\delta$-homotopy equivalent}\}.\]
If $(X, \varphi_X)$ and $(Y, \varphi_Y)$ are not $\delta$-homotopy equivalent for any $\delta$, then we declare
\[\dHT((X, \varphi_X), (Y, \varphi_Y)) = \infty.\]
\end{defn}

\begin{prop}[{\cite[Proposition~2.10]{frosini2017persistent}}]\label{prop:dht_pseudometric}
The $\dHT$ distance is an extended pseudo-metric on the set of pairs of topological spaces and real-valued functions.
\end{prop}

A pair $(X, \varphi_X)$ (where $X$ is a topological space and $\varphi_X\colon X\to\R$ is a real-valued function) induces an $\R$-space $\filtf{X}{\varphi_X}$ given by the sublevel set filtration, and defined as follows:
\[ \filt{X}{\varphi_X}{r} := \varphi_X^{-1}((-\infty, r))\mbox{ for }r\in \R\]
and
\[ \filtf{X}{\varphi_X} := \left\{\varphi_X^{-1}\big((-\infty,r)\big) \subseteq \varphi_X^{-1}\big((-\infty,r')\big)\right\}_{r\leq r'}.\]
The following theorem shows the interleaving distance of the persistent homology of the sublevel set filtrations is bounded by the $\dHT$ distance of the respective pairs.
The theorem is a slight generalization of~\cite[Lemma~3.1]{frosini2017persistent} in which we do not require the continuity of $\varphi_X$ and $\varphi_Y$; we omit its (identical) proof.

\begin{lem}[{\cite[Lemma 3.1]{frosini2017persistent}}]
\label{lem:dht_bound_interleaving}
Let $(X, \varphi_X)$ and $(Y, \varphi_Y)$ be two pairs.
Then for any integer $k\geq 0$, we have
\[\di^{\mathrm{Vec}}(H_k\circ \filtf{X}{\varphi_X}, H_k\circ \filtf{Y}{\varphi_Y})\leq \dHT((X, \varphi_X), (Y, \varphi_Y)).\]
\end{lem}

\section{The $p$-relaxation of metric thickenings}\label{sec:p-relaxation}

We now describe the construction of the $p$-relaxations of metric thickenings.

\subsection{The relaxed diameter and radius functionals}

Throughout this section, $(X,d_X)$ will denote a bounded metric space.
Recall that $\cP_X$ is the set of all Radon probability measures on $X$, equipped with the weak topology.
We now introduce for each $p\in[1,\infty]$ the \emph{$p$-diameter} of a measure $\alpha$ in $\cP_X$, where the $\infty$-diameter of $\alpha$ is precisely the diameter of its support $\supp{(\alpha)}$ (as a subset of $X$).
Consider for each $p\in[1,\infty]$ the $p$-\emph{diameter} map $\diam_p\colon \cP_X\to \R_{\geq 0}$
given by
\[\diam_p(\alpha) =
\begin{dcases}\left(\iint_{X\times X}d_X^p(x,x')\,\alpha(dx)\,\alpha(dx')\right)^{\frac{1}{p}} & \mbox{when $p<\infty$} \\
& \\
\diam(\supp(\alpha)) & \mbox{for $p=\infty$.}
\end{dcases}
\]
We remark that the $p$-diameter has been studied and critiqued as a measure of diversity in population biology~\cite{rao1982diversity,pavoine2005measuring}, and it has also been considered in relation to the Gromov--Wasserstein distance~\cite{memoli2011gromov}.
Similarly, define the $p$-\emph{radius} map $\rad_p\colon \cP_X\to \R_{\geq 0}$
via
\[\rad_p(\alpha) =
\begin{dcases}\inf_{x\in X}\left(\int_{X}d_X^p(x,x')\,\alpha(dx')\right)^{\frac{1}{p}} & \mbox{when $p<\infty$} \\
& \\
\rad(\supp(\alpha))
& \mbox{for $p=\infty$.}
\end{dcases}
\]
Note that the $p$-radius of a measure is simply the $p$-th root of its $p$-variance.

We observe that
\[\diam_p(\alpha) = \left(\int_X F^p_{\alpha, p}(x)\,\alpha(dx)\right)^\frac{1}{p} = \left(\int_X \big(\dWp(\alpha, \delta_x)\big)^p\,\alpha(dx)\right)^\frac{1}{p}\]
and
\[\rad_p(\alpha) = \inf_{x\in X}F_{\alpha,p}(x) = \inf_{x\in X}\dWp(\delta_x,\alpha).\]

\begin{prop}
\label{prop:basic-diamp}
The functions $\diam_p,\rad_p\colon \cP_X\to \R$ satisfy
\[\rad_p(\alpha)\leq \diam_p(\alpha)\leq 2\,\rad_p(\alpha).\]
\end{prop}

\begin{proof}
To see that $\rad_p(\alpha)\leq \diam_p(\alpha)$, note that we have
\begin{align*}
\diam_p(\alpha) &= \Big(\iint_{X\times X} \big(d_X(x,x')\big)^p\alpha(dx)\,\alpha(dx')\Big)^{\frac{1}{p}}   \\
& =\Big(\int_{X} \left(\int_Xd^p_X(x, x')\,\alpha(dx)\,\right)\,\alpha(dx')\Big)^{\frac{1}{p}} \\
& \geq \left(\inf_{x'\in X}\left(\int_{X}d_X^p(x,x')\,\alpha(dx)\right)\right)^{\frac{1}{p}} \\
&= \inf_{x'\in X}\left(\int_{X}d_X^p(x,x')\,\alpha(dx)\right)^{\frac{1}{p}} \\
&= \rad_p(\alpha).
\end{align*}

To see that $\diam_p(\alpha)\leq 2\,\rad_p(\alpha)$, for any $\varepsilon>0$, there is a point $z\in X$ with
\[\left(\int_Xd_X^p(x, z)\alpha(dx)\right)^\frac{1}{p} \le \rad_p(\alpha)+\varepsilon.\]
Then we have
\begin{align*}
\diam_p(\alpha) & = \Big(\iint_{X\times X} \big(d_X(x,x')\big)^p\alpha(dx)\,\alpha(dx')\Big)^{\frac{1}{p}}    \\
& \leq \Big(\iint_{X\times X} \big(d_X(x,z) + d_X(z, x')\big)^p\alpha(dx)\,\alpha(dx')\Big)^{\frac{1}{p}}  \\
& \leq \Big(\iint_{X\times X} \big(d_X(x,z)\big)^p\alpha(dx)\,\alpha(dx')\Big)^{\frac{1}{p}} + \Big(\iint_{X\times X} \big(d_X(z,x')\big)^p\alpha(dx)\,\alpha(dx')\Big)^{\frac{1}{p}} \\
& \le 2\,\rad_p(\alpha)+2\varepsilon.
\end{align*}
Since $\varepsilon>0$ was arbitrary, this shows $\diam_p(\alpha)\leq 2\,\rad_p(\alpha)$.
\end{proof}

\begin{remark}
The tightness of the bound $\rad_p(\alpha)\leq \diam_p(\alpha)$ can be seen from the calculation for the uniform measure on $Z_n$, the metric space on a set of size $n$ with all interpoint distances equal to $1$.
To see the (asymptotic) tightness of $\diam_p(\alpha)\leq 2\,\rad_p(\alpha)$, we consider the metric space $Z_n\cup \{O\}$, where the newly introduced ``center" $O$ has distance $1/2$ to every other point.
Then for the measure $\alpha := \sum_{z\in Z_n} \frac{1}{n}\delta_z$, we have $\rad_p(\alpha) = 1/2$ and $\diam_p(\alpha) = (\frac{n-1}{n})^{\frac{1}{p}}$.
We obtain asymptotic tightness by letting $n$ go to infinity.
\end{remark}

\subsection{The relaxed Vietoris--Rips and \v{C}ech metric thickenings}

\begin{defn}[$p$-Vietoris--Rips filtration]\label{defn:vrp}
For each $r>0$ and $p\in[1,\infty]$, let the \emph{$p$-Vietoris--Rips metric thickening at scale $r$} be
\[\vrp{X}{r}:=\{\alpha\in\cP_X~|~\diam_p(\alpha)<r\}.\]

We regard $\vrp{X}{r}$ as a topological space by endowing it with the subspace topology from $\cP_X$ (i.e., the weak topology).
By convention, when $r \leq 0$, we will let ${\vrp{X}{r} = \varnothing}$.
By \[\vrpf{X}:=\left\{\vrp{X}{r}\stackrel{\nu_{r,r'}^X}{\longhookrightarrow} \vrp{X}{r'}\right\}_{r \leq r'}\]
we will denote the filtration thus induced.

We use $\vrpfin{X}{r}$ and $\vrpffin{X}$ to denote the finitely-supported variants, obtained by replacing $\cP_X$ in the above definition with $\cPfin_X$.
\end{defn}

Note that in the specific case $p=\infty$, the $\infty$-diameter of a measure is simply the diameter of its support.
For a bounded metric space $X$, the weak topology is generated by the $1$-Wasserstein metric, so the definition of $\vrpfin{X}{r}$ generalizes~\cite[Definition~3.1]{AAF}, which is the specific case $p=\infty$.

\begin{remark}[$\vrpf{X}$ as a softening of $\mathrm{VR}_{\infty}(X,\smb)$]
We should point out that the definition of $\vrpf{X}$ can be extended to the whole range $p\in[0,\infty]$ as follows.
One first notices that $\diam_p(\alpha)$ can still be defined as above for $p\in(0,1)$, and by $\diam_0(\alpha) = 0$ when $p=0$.
Furthermore, if $\alpha = \sum_{i\in I}a_i\,\delta_{x_i}$, then
\[\lim_{p\downarrow 0}\diam_p(\alpha) = \prod_{i,j\in I} \big(d_X(x_i,x_j)\big)^{a_i\cdot a_j},\] 
which equals $0$ since the product contains terms with $i=j$.

At any rate, by the standard generalized means inequality~\cite{bullen2013means}, we have $\diam_{p'}(\alpha)\leq \diam_p(\alpha)$ for any $p'\leq p$ in the range $[0,\infty]$.
So for fixed $r>0$, not only does $\mathrm{VR}_p(X;r)$ become larger and larger as $p$ decreases, but also for $p=0$, $\mathrm{VR}_0(X;r)$ contains \emph{all} Radon probability measures on $X$ and thus has trivial reduced homology.
    
\end{remark}

\begin{defn}[$p$-\v{C}ech filtration]
\label{defn:cechp}
For each $r>0$ and $p\in[1,\infty]$, let the \emph{$p$-\v{C}ech metric thickening at scale $r$} be
\[\cechp{X}{r}:=\{\alpha\in\cP_X~|~\rad_p(\alpha)<r\}.\]

We regard $\cechp{r}{X}$ as a topological space by endowing it with the subspace topology from $\cP_X$ (i.e., the weak topology).
By convention, when $r \leq 0$, we will let ${\cechp{X}{r} = \varnothing}$.

By
\[\cechpf{X}:=\left\{\cechp{X}{r}\stackrel{\nu_{r,r'}^X}{\longhookrightarrow} \cechp{X}{r'}\right\}_{r \leq r'}\] we will denote the filtration thus induced.\footnote{We will use $\nu_{r,r'}^X$ to denote the inclusion maps in both \v{C}ech and Vietoris--Rips thickenings; this will not lead to confusion in this paper.}

\begin{remark}
Let $\alpha$ be a measure in $\cP_X$.
As $\rad_p(\alpha) = \inf_{x\in X}\dWp(\delta_x,\alpha)$, we have $\rad_p(\alpha) <r$ if and only if there is some $x\in X$ such that $\dWp(\delta_x,\alpha) < r$.
Therefore, $\cechp{X}{r}$ is exactly the union over all $x\in X$ of the balls $B(\delta_x; r)$, with respect to the $p$-Wasserstein metric, 
centered at points $\delta_x$ in the isometric image of $X$ in $\cP_X$.
\end{remark}

We use $\cechpfin{X}{r}$ and $\cechpffin{X}$ to denote the finitely-supported variants, obtained by replacing $\cP_X$ in the above definition with $\cPfin_X$.
\end{defn}

Note that in the specific case $p=\infty$, the $\infty$-radius of a measure is simply the radius of its support.
For a bounded metric space $X$, the weak topology is generated by the $1$-Wasserstein metric, so the definition of $\cechpfin{X}{r}$ generalizes~\cite{AAF} which considers the specific case $p=\infty$.

Though the above definitions are given with the $<$ convention, we remark that they instead could have been given with the $\le$ convention, namely $\diam_p(\alpha)\le r$ or $\rad_p(\alpha)\le r$.
In this paper, we restrict attention to the $<$ convention, even though many of the statements we give are also true with the $\le$ convention.

The next proposition shows that $p$-Vietoris--Rips and $p$-\v{C}ech metric thickenings are nested in the same way that Vietoris--Rips and \v{C}ech simplicial complexes are.

\begin{prop}
\label{prop:nesting}
Let $X$ be a bounded metric space.
Then, for any $r> 0$, we have
\[\vrp{X}{r}\subseteq \cechp{X}{r}\subseteq\vrp{X}{2r}.\]
\end{prop}

\begin{proof}
This follows immediately from Proposition~\ref{prop:basic-diamp}, which implies that for any $\alpha\in\cP_X$,
\[\diam_p(\alpha) \ge \rad_p(\alpha) \ge \tfrac{1}{2}\,\diam_p(\alpha).\]
\end{proof}

If $X$ and $Z$ are metric spaces with $X\subseteq Z$, and if the metric on $Z$ is an extension of that on $X$, then following Gromov~\cite[Section~1B]{Gromov93} we say that $Z$ is a \emph{$r$-metric thickening of $X$} if for all $z\in Z$, there is some $x\in X$ with $d_Z(x,z)\le r$.

\begin{prop}\label{prop:r-metric_thickenings}
Let $X$ be a bounded metric space.
When equipped with the $q$-Wasserstein metric for $1\le q\le p$, we have that $\cechp{X}{r}$ and $\vrp{X}{r}$ are each $r$-metric thickenings of $X$ for all $r > 0$.
\end{prop}

\begin{proof}
We use the isometric embedding $X \to \cechp{X}{r}$ given by $x \mapsto \delta_{x}$.
Let $\alpha\in\cechp{X}{r}$.
Hence there exists some $x\in X$ with $r > \dWp(\delta_x,\alpha) \ge \dWq(\delta_x,\alpha)$, which shows that $\cechp{X}{r}$ equipped with the $q$-Wasserstein metric is an $r$-metric thickening of $X$.

The Vietoris--Rips case follows immediately since $\vrp{X}{r}\subseteq \cechp{X}{r}$ by Proposition~\ref{prop:nesting}.
\end{proof}

The following remark follows the perspective introduced in the recent paper~\cite{lim2020vietoris}, which shows that the filling radius is related to the persistent homology of the Vietoris--Rips simplicial complex filtration $\vrf{X}$.
\begin{remark}\label{rmk:filling_radius}
Let $X$ be a closed connected $n$-dimensional manifold so that the \emph{fundamental class} of $X$ is well-defined.
Let $Z$ be an $r$-metric thickening of $X$ for some $r>0$.
It is shown in~\cite[Page~8, Another Corollary]{gromov1983filling} that the map
\[
H_n(X)\hookrightarrow H_n(Z)
\]
induced from the inclusion $\iota:X\rightarrow Z$ is an injection whenever $r$ is less than a scalar geometric invariant $\mathrm{FillRad}(X)$ called the \emph{filling radius} of $X$.
Hence Proposition~\ref{prop:r-metric_thickenings} implies that for all $p\in [1, \infty]$, the persistence diagram in dimension $n$ of either the filtration $\vrpf{X}$ or $\cechpf{X}$ contains an interval that with left endpoint equal to $0$ and length at least $\mathrm{FillRad}(X)$.
As proved in~\cite{katz1983filling}, the filling radius of the $n$-sphere (with its geodesic metric) is $\mathrm{FillRad}(\Sp^n) = \frac{1}{2} \arccos\left(\frac{-1}{n+1}\right)$.
Therefore, when $X=\Sp^n$, the $n$-dimensional persistence diagram of either $\vrpf{\Sp^n}$ or $\cechpf{\Sp^n}$ contains an interval starting at zero which is no shorter than $\frac{1}{2} \arccos\left(\frac{-1}{n+1}\right)$.
\end{remark}

\begin{lem}
For all $r>0$ and all $p,p'\in[1,\infty]$ with $p\geq p'$, one has
\[\vrpp{p}{X}{r}\subseteq \vrpp{p'}{X}{r}
\quad\mbox{and}\quad
\cechpp{p}{X}{r} \subseteq \cechpp{p'}{X}{r}.\]
\end{lem}

\begin{proof}
This comes from applying the H\"{o}lder inequality on $\diam_p$ and $\rad_p$.
\end{proof}

\begin{example}\label{ex:Delta-n}
We recall that $Z_{n+1}$ is the metric space consisting of $n+1$ points with all interpoint distances equal to 1.
For any natural number $n$, the Vietoris--Rips or \v{C}ech simplicial complex filtrations of $Z_{n+1}$ do not produce any non-diagonal point in their persistence diagram except in homological dimension zero.
However, for any $1<p<\infty$, both the $p$-Vietoris--Rips and $p$-\v{C}ech filtrations of $Z_{n+1}$ will contain non-diagonal points in their persistence diagrams for homological degrees from $0$ to $n-1$.
More specifically, both the $p$-Vietoris--Rips and $p$-\v{C}ech filtrations are homotopy equivalent to a filtration of an $n$-simplex $\Delta_n$ by its $k$-skeleta, $\Delta_n^{(k)}$.
We prove the following result in Appendix~\ref{app:pd_delta_1_space}.
For $r$ in the interval $\left( \left(\frac{k}{k+1}\right)^{\frac{1}{p}}, \left(\frac{k+1}{k+2}\right)^{\frac{1}{p}}\right]$ with $0\leq k \leq n-1$, we have
\[ \vrp{Z_{n+1}}{r} \simeq \cechp{Z_{n+1}}{r} \simeq \Delta_{n}^{(k)}, \]
and when $r>\left(\frac{n}{n+1}\right)^{\frac{1}{p}}$, both $\vrp{Z_{n+1}}{{r}}$ and $\cechp{Z_{n+1}}{{r}}$ become the $n$-simplex $\Delta_{n}$ which is contractible.

From this we get the persistence diagrams
\[ \dgmvr_{k,p}(Z_{n+1}) = \dgmcech_{k,p}(Z_{n+1}) = \begin{cases}
\left(0, \left(\tfrac{1}{2}\right)^{\tfrac{1}{p}}\,\right)^{\otimes {n}}\oplus (0, \infty)  & \,\text{if $k=0$},             \\
\left(\left(\tfrac{k}{k+1}\right)^{\tfrac{1}{p}}, \left(\tfrac{k+1}{k+2}\right)^{\frac{1}{p}}\,\right)^{\otimes {{n}\choose{k+1}}} & \,\text{if $1\le k \leq n-1$}, \\
\emptyset & \,\text{if $k \ge n$}.
\end{cases}
\]
Note that from the definition of $\vrppf{\infty}{Z_{n+1}}$, we have $\dgmvr_{k,\infty}=\emptyset$ for $k\geq 1$, and furthermore that $\lim_{p\uparrow \infty}\dgmvr_{k,p}(Z_{n+1}) = \dgmvr_{k,\infty}(Z_{n+1})$ for each $k\geq 0$.
An analogous result holds for $\cechpf{Z_{n+1}}$.

\begin{figure}
\centering
\includegraphics[width=0.6\textwidth]{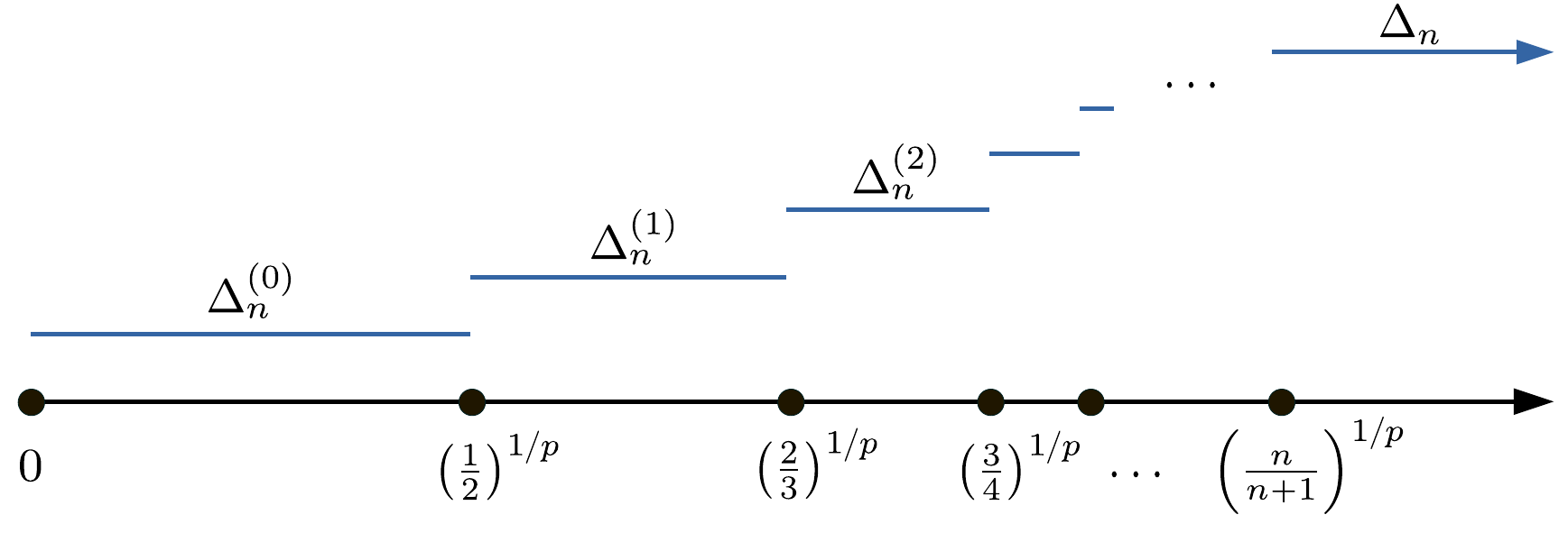}
\caption{Homotopy types of both $\vrp{Z_{n+1}}{\smb}$ and $\cechp{Z_{n+1}}{\smb}$.}
\label{fig:homotopyTypesDelta}
\end{figure}

\end{example}

\subsection{A more general setting: controlled invariants}
We now generalize the relaxed $p$-Vietoris-Rips and $p$-\v{C}ech filtrations.

\begin{defn}[Controlled invariants on $\cP_X$]\label{def:controlled-invariant}
Let $\mi$ be a functional that associates to each bounded metric space $X$ a function $\mi^X\colon \cP_X\to \mathbb{R}$.
We say, for some $L>0$, that $\mi$ is an \emph{$L$-controlled} invariant if the following conditions are satisfied:
\begin{enumerate}
	\item \emph{Stability under pushforward}: For any map $f$ between finite metric spaces $X$ and $Y$ and any $\alpha$ in $\cP_X$,
    \[\mi^Y(f_\sharp(\alpha))  \leq \mi^X(\alpha)+L\cdot \mathrm{dis}(f).\]
	\item \emph{Stability with respect to $\dWqq{\infty}$}: For any bounded metric space $X$ and any $\alpha, \beta$ in $\cP_X$,
	\[|\mi^X(\alpha) - \mi^X(\beta)| \leq 2L\cdot \dWqq{\infty}(\alpha, \beta).\]
\end{enumerate}
\end{defn}

In the next section, we will prove that both $\diam_p$ and $\rad_p$ induce controlled invariants.

For any controlled invariant $\mi$, we will use the notation $\filtf{\cP_X}{\mi^X}$ to denote the sublevel set filtration induced by the pair $(\cP_X, \mi^X)$.
That is,
\[  \filt{\cP_X}{\mi^X}{r}=(\mi^X)^{-1}\big((-\infty,r)\big)\mbox{ for }r\in \R\]
and
\[ \filtf{\cP_X}{\mi^X}=\left\{(\mi^X)^{-1}\big((-\infty,r)\big) \subseteq (\mi^X)^{-1}\big((-\infty,r')\big)\right\}_{r\leq r'}.\]
This construction generalizes both the definition of $\vrpf{X}$ and that of $\cechpf{X}$.

Similarly, any controlled invariant $\mi\colon \cPfin_X\to\R$ induces an analogous sublevel set filtration $\filtf{\cPfin_X}{\mi^X}$ of $\cPfin_X$.
This construction generalizes both the definition of $\vrpffin{X}=\filtf{\cP_X}{\diam_p^X}$ and that of $\cechpffin{X}=\filtf{\cP_X}{\rad_p^X}$.

We'll see later on in Corollary~\ref{cor:cP_cPfin_0-interleaved} that the filtrations $\filtf{\cP_X}{\mi^X}$ and $\filtf{\cPfin_X}{\mi^X}$ yield persistence modules which are $\epsilon$-interleaved for arbitrary small $\epsilon>0$.

%%%%%%%%%%%%%%%%%%%%%%%%%%%%

\section{Basic properties}\label{sec:basic-properties}

We prove some basic properties.
These include the convexity of the Wasserstein distance, the fact that $\diam_p$ and $\rad_p$ are $1$-controlled, and the fact that for $X$ finite, the $p=\infty$ metric thickenings $\vrpp{\infty}{X}{r}$ and $\cechpp{\infty}{X}{r}$ are homeomorphic to the simplicial complexes $\vr{X}{r}$ and $\cech{X}{r}$, respectively.
We begin with the property that nearby measures have nearby integrals.
Throughout this section, $X$ is a bounded metric space.

\begin{lem}\label{lem:dwq_stability_wrt_Lipschitz}
Let $\varphi$ be an $L$-Lipschitz function on $X$.
Then for any $\alpha, \beta$ in $\cP_X$ and any $q\in [1, \infty]$, we have
\[\left\vert \left(\int_X|\varphi(x)|^q\, \, \alpha(dx)\right)^{\frac{1}{q}} -\left(\int_X|\varphi(x)|^q\, \, \beta(dx)\right)^{\frac{1}{q}}  \right\vert\leq L \cdot \dWq(\alpha, \beta).\]
\end{lem}
\begin{proof}
Let $\mu \in \cpl(\alpha,\beta)$ be any coupling.
Then we have
\begin{align*}
& \left\vert\left( \int_X|\varphi(x)|^q\, \, \alpha(dx)\right)^{\frac{1}{q}}  -\left(\int_X|\varphi(x)|^q\, \, \beta(dx)\right)^{\frac{1}{q}}  \right\vert \\
=\ & \bigg\vert \left( \iint_{X\times X}|\varphi(x)|^q\, \, \mu(dx\times dx') \right)^{\frac{1}{q}} -\left( \iint_{X\times X}|\varphi(x')|^q\, \,  \mu(dx\times dx')  \right)^{\frac{1}{q}}\bigg\vert \\
\leq\ & \left(\iint_{X\times X}|\varphi(x) - \varphi(x')|^q\, \,  \mu(dx\times dx')\right)^{\frac{1}{q}}                  \\
\leq\ & L\left(\iint_{X\times X}d_X^q(x, x')\,  \mu(dx\times dx')\right)^{\frac{1}{q}}.
\end{align*}
We obtain the claim by infimizing the right hand side with respect to the coupling $\mu$.
\end{proof}

We will often use the following convexity result of the Wasserstein distance; for a more general result, see~\cite[Theorem 4.8]{villani2008optimal}.

\begin{lem}\label{lem:bound-distance-convex-comb}
If $c_1, \dots, c_n$ are non-negative real numbers satisfying $\sum_{k=1}^n c_k = 1$ and if $\alpha_1, \dots, \alpha_n,\alpha'_1, \dots, \alpha'_n \in \cP_X$, then for all $q \in [1, \infty)$, we have
\[
\dWq \Big( \sum_{k=1}^n c_k \alpha_k ,\, \sum_{k=1}^n c_k \alpha'_k \Big) \leq \left( \sum_{k=1}^n c_k \left(\dWq(\alpha_k, \alpha'_k) \right)^q \right)^\frac{1}{q}
\]
and
\[
\dWqq{\infty} \Big( \sum_{k=1}^n\,c_k \alpha_k ,\, \sum_{k=1}^n c_k\,\alpha'_k \Big) \leq \max_k\,\dWqq{\infty}(\alpha_k, \alpha'_k).\]
\end{lem}

\begin{proof}
Let $\varepsilon>0$ and $q \in [1,\infty)$.
For each $k$, suppose $\mu_k$ is a coupling between $\alpha_k$ and $\alpha'_k$ such that
\[ \iint_{X\times X}d^q_X(x,x')\,\mu_k(dx\times dx') < \left(\dWq(\alpha_k, \alpha'_k) + \varepsilon\right)^q.\]
Then it can be checked that $\sum_{k=1}^n c_k \mu_k$ is a coupling between $\sum_{k=1}^n c_k \alpha_k$ and $\sum_{k=1}^n c_k \alpha'_k$.
We have
\begin{align*}
\dWq \left( \sum_{k=1}^n\,c_k \alpha_k ,\, \sum_{k=1}^n c_k\,\alpha'_k \right) & \leq \left( \iint_{X\times X}d_X^q(x,x')\,\sum_{k=1}^n c_k \mu_k(dx\times dx') \right)^\frac{1}{q}        \\
& < \left( \sum_{k=1}^n c_k \left(\dWq(\alpha_k, \alpha'_k) + \varepsilon \right)^q \right)^\frac{1}{q}     \\
& \leq \left( \sum_{k=1}^n c_k \left(\dWq(\alpha_k, \alpha'_k) \right)^q \right)^\frac{1}{q} + \varepsilon.
\end{align*}
Since this holds for all $\varepsilon>0$, the claimed inequality holds.
The case for $q = \infty$ can be checked separately using the same matching.
\end{proof}

We define an invariant generalizing $\diam_p$, and then prove that it is $1$-controlled.

\begin{defn}[$\iqp^X$ invariant]\label{defn:iqp}
Let $X$ be a bounded metric space.
For any $p, q \in [1, \infty]$, we define the \emph{$\iqp$ invariant}, which associates to each bounded metric space $X$ a function $\iqp^X\colon \cP_X\to \R$, 
\[\iqp^X(\alpha)=\left(\int_X\big(\dWq(\alpha,\delta_x)\big)^p\alpha(dx)\right)^{1/p}.\]
Note that $\iqp^X$ recovers $\diam_p^X$ when $p=q$.
\end{defn}

\begin{remark}
In the contruction of $\iqp^X$, one can get other interesting filtration functions by replacing the Wasserstein distance with other distances between measures.
These include, for example, the L\'{e}vy--Prokhorov metric, the Fortet--Mourier metric, and variants of the Kantorovich--Rubinstein metric.
For definitions of these metrics, see~\cite[Section 3]{bogachev2018weak}.
\end{remark}

\begin{lem}\label{lem:iqp_distortion_via_pushforward}
Let $X$ and $Y$ be bounded metric spaces, and let $f \colon X\to Y$ be a map.
Then for any element $\alpha\in \cP_X$ and $p\in[1,\infty]$, we have
\[\iqp^Y(f_\sharp(\alpha))  \leq \iqp^X(\alpha)+\mathrm{dis}(f).\]
\end{lem}

\begin{proof}
We compute
\begin{align*}
\iqp^Y(f_\sharp(\alpha)) & =\left(\int_Y\big(d_{\mathrm{W}, q}^Y(f_\sharp(\alpha),\delta_y)\big)^p\,f_\sharp(\alpha)(dy)\right)^{1/p} \\
& =\left(\int_X\big(d_{\mathrm{W}, q}^Y(f_\sharp(\alpha),\delta_{f(x)})\big)^p\,\alpha(dx)\right)^{1/p}    \\
& =\left(\int_X\bigg(\int_Yd_Y^q(y,f(x))\,f_\sharp(\alpha)(dy)\bigg)^\frac{p}{q}\,\alpha(dx)\right)^{1/p}     \\
& =\left(\int_X\bigg(\int_Xd_Y^q(f(x'),f(x))\,\alpha(dx')\bigg)^\frac{p}{q}\,\alpha(dx)\right)^{1/p}          \\
& \leq\left(\int_X\bigg(\int_X\big(d_X(x',x)+ \mathrm{dis}(f) \big)^q\,\alpha(dx')\bigg)^\frac{p}{q}\,\alpha(dx)\right)^{1/p} \\
& \leq \iqp^X(\alpha)  + \mathrm{dis}(f).
\end{align*}
\end{proof}

\begin{lem}\label{lem:iqp_stability_wrt_dW}
Let $X$ be a bounded metric space.
Then for any $\alpha, \beta$ in $\cP_X$, we have
\[|\mi^X_{q,p}(\alpha) - \iqp^X(\beta)|\leq  \dWq(\alpha, \beta) + \dWp(\alpha, \beta).\]
In particular, $|\mi^X_{q,p}(\alpha) - \iqp^X(\beta)|\leq  2\,\dWqq{\infty}(\alpha, \beta)$.
\end{lem}

\begin{proof}
We first note that
\begin{align*}
\iqp^X(\alpha) & =\left(\int_X\big(\dWq(\alpha,\delta_x)\big)^p\alpha(dx)\right)^{1/p}                           \\
& \leq\left(\int_X\big(\dWq(\beta ,\delta_x) + \dWq(\alpha, \beta)\big)^p\alpha(dx)\right)^{1/p}    \\
& \leq \left(\int_X\big(\dWq(\beta ,\delta_x) \big)^p \alpha(dx)\right)^{1/p} + \dWq(\alpha, \beta).
\end{align*}
We also have
\begin{align*}
&\left\vert\left(\int_X\big(\dWq(\beta ,\delta_x) \big)^p \alpha(dx)\right)^{1/p} - \iqp(\beta)\right \vert \\
=& \bigg\vert\left(\int_X\big(\dWq(\beta ,\delta_x) \big)^p \alpha(dx)\right)^{1/p} - \left(\int_X\big(\dWq(\beta ,\delta_x) \big)^p \beta(dx)\right)^{1/p}\bigg \vert \\
\leq & \dWp(\alpha, \beta),
\end{align*}
where the last line follows from Lemma~\ref{lem:dwq_stability_wrt_Lipschitz} since $\dWq(\beta, \delta_x)$ is a $1$-Lipschitz function in $x$.
Therefore,
\[\iqp(\alpha) \leq  \iqp(\beta) + \dWq(\alpha, \beta) + \dWp(\alpha, \beta).\]
We then get the result by swapping the roles of $\alpha$ and $\beta$.
\end{proof}

\begin{remark}
This establishes the continuity of $\iqp$ for $q, p \in [1, \infty)$.
If $p$ or $q$ equals infinity, then $\iqp$ is not necessarily continuous in the weak topology since $\dWqq{\infty}$ does not necessarily induce the same topology on $\cP_X$ as $\dWp$ does for $p<\infty$.
Indeed, if $X=[0,1]$ is the unit interval, then $\frac{n-1}{n}\delta_0+\frac{1}{n}\delta_1\to\delta_0$ in $\cP_X$ as $n\to\infty$, even though $\diam_\infty(\frac{n-1}{n}\delta_0+\frac{1}{n}\delta_1)$ is equal to $1$ for all $n\ge 1$, and hence does not converge to $0=\diam(\delta_0)$.
\end{remark}

The above two lemmas imply that $\iqp$ (and hence $\diam_p=\mi_{p.p}$) is a $1$-controlled invariant.
We next consider the analogous properties for $\rad_p$.

\begin{lem}[$\rad_p$ under a map with bounded distortion] \label{lem:rad_p_dist}
Let $X$ and $Y$ be bounded metric spaces, and let $f\colon X\to Y$ be a map.
Then for any $\alpha \in \cP_X$ and $p\in[1,\infty]$, we have
\[\rad_p(f_\sharp(\alpha))  \leq  \rad_p(\alpha)+ \mathrm{dis}(f).\]
\end{lem}

\begin{proof}
We only give the proof for the case $p<\infty$; the case $p=\infty$ is similar.
\begin{align*}
\rad_p(f_\sharp(\alpha)) & = \inf_{y\in Y}\left(\int_{Y}d_Y^p(y,y')\,f_\sharp(\alpha)(dy')\right)^{\frac{1}{p}}                        \\
& \leq \inf_{y\in f(X)}\left(\int_{Y}d_Y^p(y,y')\,f_\sharp(\alpha)(dy')\right)^{\frac{1}{p}}                  \\
& =\inf_{x\in X}\left(\int_{Y}d_Y^p(f(x),y')\,f_\sharp(\alpha)(dy')\right)^{\frac{1}{p}}                      \\
& =\inf_{x\in X}\left(\int_{X}d_Y^p(f(x),f(x'))\,\alpha(dx')\right)^{\frac{1}{p}}                             \\
& \leq \inf_{x\in X}\left(\int_{X}\left(d_X(x, x') + \mathrm{dis}(f)\right)^{p}\,\alpha(dx')\right)^{\frac{1}{p}} \\
& \leq \inf_{x\in X}\left(\int_{X}d_X^p(x,x')\,\alpha(dx')\right)^{\frac{1}{p}} + \mathrm{dis}(f)                 \\
& =\rad_p(\alpha) + \mathrm{dis}(f).
\end{align*}
\end{proof}

\begin{lem}[Stability of $\rad_p$]\label{lem:stability_of_rad_p}
Let $X$ be a bounded metric space.
For any two probability measures $\alpha,\beta \in \cP_X$ and for every $p\in[1,\infty]$, we have
\[ |\rad_p(\alpha)-\rad_p(\beta)|\leq \dWp(\alpha,\beta).\]
\end{lem}

\begin{proof}
We compute
\begin{align*}
\rad_p(\alpha) & = \inf_{x\in X} \dWp(\alpha, \delta_x)   \\
& \leq \inf_{x\in X} \left(\dWp(\beta, \delta_x) + \dWp(\beta, \alpha)\right) \\
&= \inf_{x\in X} \dWp(\beta, \delta_x) + \dWp(\beta, \alpha)              \\
& = \rad_p(\beta) + \dWp(\beta, \alpha).
\end{align*}
\end{proof}

\begin{remark}
Lemma~\ref{lem:stability_of_rad_p} establishes the continuity of $\rad_p$ for $p \in [1,\infty)$, as these functions are 1-Lipschitz.
Similarly to $\diam_\infty$, we note that $\rad_\infty$ is not necessarily continuous because the metric topology given by $\dWqq{\infty}$ is not necessarily equal to the weak topology.
\end{remark}

The above two lemmas imply that $\rad_p$ is a $1$-controlled invariant.

We end this section of basic properties by showing that $\cP_X$ is homeomorphic to a simplex when $X$ is finite.
Hence for $X$ finite the $p=\infty$ metric thickenings $\vrpp{\infty}{X}{r}$ and $\cechpp{\infty}{X}{r}$ are homeomorphic to the simplicial complexes $\vr{X}{r}$ and $\cech{X}{r}$, respectively (see also~\cite[Corollary~6.4]{AAF}).

\begin{lem}
\label{lem:fin-prob-simplex}
If $X$ is a finite metric space with $n$ points, then $\cP_X$ is homeomorphic to the standard $(n-1)$-simplex.
\end{lem}

\begin{proof}
Let $X = \{x_1, \dots, x_n\}$.
The space $\cP_X$ of probability measures is in bijection with the standard $n-1$ simplex $\Delta_{n-1} = \{ (y_1, \dots, y_n) \in \R^n \mid \sum_{i=1}^n y_i = 1,\, y_i \geq 0 \text{ for all $i$} \}$ via the function $f\colon \cP_X \to \Delta_{n-1}$ defined by $f \left( \sum_{i=1}^n a_i \delta_{x_i} \right) = (a_1, \dots, a_n)$.
Suppose we have a sequence $\{\alpha_k\}$ in $\cP_X$ given by $\alpha_k = \sum_{i=1}^n a_{k,i} \delta_{x_i}$.
By the definition of weak convergence, $\{ \alpha_k \}$ converges to $\alpha = \sum_{i=1}^n a_i \delta_{x_i}$ in $\cP_X$ if and only if $\int_X \phi(x) \alpha_k(dx)$ converges to $\int_X \phi(x) \alpha(dx)$ for all bounded and continuous functions $\phi\colon X \to \R$.
These integrals are equal to $\sum_{i=1}^n a_{k,i} \phi (x_i)$ and $\sum_{i=1}^n a_i \phi (x_i)$ respectively, so $\{ \alpha_k \}$ converges to $\alpha$ if and only if $\lim_{k \to \infty} a_{k,i} = a_i$ for each $i$.
Therefore, $\{ \alpha_k \}$ converges to $\alpha$ in $\cP_X$ if and only if $\{ f(\alpha_k) \}$ converges to $f(\alpha)$ in $\Delta_{n-1}$.
\end{proof}

In general, the $p=\infty$ metric thickenings $\vrpp{\infty}{X}{r}$ and $\cechpp{\infty}{X}{r}$ of a finite metric space $X$ are in bijection, as sets,  with the geometric realizations of the usual Vietoris--Rips and \v{C}ech simplicial complexes on $X$, via the natural bijection $\sum_{i=1}^n a_i \delta_{x_i} \mapsto \sum_{i=1}^n a_i x_i$ where $\sum_{i=1}^n a_i x_i$ denotes the formal sum in the geometric realization.
Therefore, Lemma~\ref{lem:fin-prob-simplex} implies the following; see also~\cite[Proposition~6.6]{AAF}

\begin{lem}\label{lem:finite_infty-thickenings_homeomorphic_to_complexes}
For a finite metric space $X$ and any $r>0$, we have homeomorphisms $\vrpp{\infty}{X}{r} \cong \vr{X}{r}$ and $\cechpp{\infty}{X}{r} \cong \cech{X}{r}$.
\end{lem}

\section{Stability}\label{sec:stability}

In this section we establish the stability of all the filtrations we have introduced in this paper, under reasonable assumptions on the metric spaces.
The proof requires new techniques in order to construct maps that adequately compare two filtrations of metric thickenings.
In the proof of stability for simplicial complexes (\cite{ChazalDeSilvaOudot2014}), simplicial maps are used to construct maps between filtrations of simplicial complexes, and the shift in parameter can be bounded using the Gromov--Hausdorff distance.
A naive attempt to apply this technique to filtrations of metric thickenings fails because there is no analogue of simplicial maps for metric thickenings.
Indeed, if $X$ and $Y$ are metric spaces that are close in the Gromov-Hausdorff distance, then we get a map $f\colon X \to Y$ that need not be continuous, but which is of bounded distortion.
This induces a \emph{continuous} map $\vr{X}{r}\to \vr{Y}{r'}$ between simplicial complexes, as long as $r'$ is chosen to be large enough, and from there one can obtain interleavings.
However, the analogous map $\vrp{X}{r}\to \vrp{Y}{r'}$ between metric thickenings cannot be continuous if $f\colon X\to Y$ is not continuous, since there are natural isometric embeddings of $X$ and $Y$ into $\vrp{X}{r}$ and $\vrp{Y}{r'}$ (see Proposition~\ref{prop:r-metric_thickenings}).
In essence, the fact that metric thickenings have a more well-behaved topology means that it is more difficult to construct interleavings between them.

To overcome this difficulty, we instead construct continuous functions between metric thickenings by distributing mass according to finite partitions of unity, subordinate to open coverings by $\delta$-balls.
This ensures that these maps distort distances in a controlled way, allowing for a controlled change in the invariants defining the filtration and thus providing an interleaving.
Creating these maps and checking their properties will require many of the ideas from the previous sections.

In Section~\ref{ssec:statement-of-stability} we state our main results, Theorem~\ref{thm:general_stability} and its immediate consequence, Theorem~\ref{thm:main}.
We give two applications of stability.
First, Section~\ref{sec:tameness} applies the stability theorem in order to show that various persistence modules of interest are Q-tame, and therefore have persistence diagrams.
Next, in Section~\ref{sec:comparability} we apply stability to show the close relationship between $\infty$-metric thickenings and simplicial complexes, generalizing and answering in the affirmative~\cite[Conjecture~6.12]{AAF} which states that these filtrations have identical persistence diagrams.
We give the proof of the stability theorem, Theorem~\ref{thm:general_stability}, in Section~\ref{ssec:proof-stability}.

\subsection{Statement of the stability theorem}
\label{ssec:statement-of-stability}

Let $X$ and $Y$ be totally bounded metric spaces, let $p\in [1,\infty]$, and let $k\ge 0$.
We will see, for example, that we have the following stability bound:
\begin{equation}
\label{eq:main}
\di\big(H_k\circ\vrpf{X},H_k\circ\vrpf{Y}\big) \leq 2\,\dGH(X,Y).
\end{equation}
This implies that if two metric spaces are close in the Gromov--Hausdorff distance, then their resulting filtrations and persistence modules are also close.

\begin{example}
\label{ex:tight}
In this example we consider the metric space $Z_{n+1}$ from Example~\ref{ex:Delta-n}, which has $n+1$ points all at interpoint distance 1 apart.
Let $\zeta_{p,k}$ be the interleaving distance between persistent homology modules
\[\zeta_{p,k}:=\di\big(H_k\circ\vrpf{Z_{n+1}},H_k\circ\vrpf{Z_{m+1}}\big) \quad \mbox{for }n\neq m.\]
Notice that from Example~\ref{ex:Delta-n}, when $n\neq m$ and $p$ is finite, we have
\[\zeta_{p,0} = \frac{1}{2} \left(\frac{1}{2}\right)^{\frac{1}{p}} 
\quad \mbox{and} \quad 
\zeta_{p,1} = \frac{1}{2}\left(\left(\frac{2}{3}\right)^{\frac{1}{p}}- \left(\frac{1}{2}\right)^{\frac{1}{p}}\right).\]
However, when $p=\infty$, $\zeta_{\infty,0} = \frac{1}{2}$ and $\zeta_{\infty,k} =0$ for all $k\geq 1$.
	
From these calculations we can make the following observations:
	
\begin{enumerate}
    \item[(i)] For $p$ infinite, the only value of $k$ for which $\zeta_{\infty,k}\neq 0$ is $k=0$.
	\item[(ii)] For $p$ finite, we have $\zeta_{p,1}> 0$.
	\item[(iii)] $\sup_k\zeta_{\infty,k} = \frac{1}{2} > \frac{1}{2} \left(\frac{1}{2}\right)^{\frac{1}{p}} = \sup_{k} \zeta_{p,k} $ for $p$ finite.
\end{enumerate}
	
From items (i) and (ii) above we can see that whereas persistence diagrams for $k\geq 1$ of $p=\infty$ thickenings do not contain discriminative information for the $Z_{n}$ spaces, in contrast, the analogous quantities for $p$ finite do absorb useful information.
	
Since it is known (cf.~\cite[Example 4.1]{dgh-props}) that $\dGH(Z_{n+1}, Z_{m+1})=\frac{1}{2}$ whenever $n\neq m$, item (iii) above suggests that the lower bound for Gromov--Hausdorff given by Equation~\eqref{eq:main} may not be tight for $p$ finite.
This phenomenon can actually be explained by the more general theorem below, which identifies a certain pseudo-metric between filtrations which mediates between the two terms appearing in Equation~\eqref{eq:main}.

\end{example}

\begin{thm}\label{thm:general_stability}
Let $\mi$ be an $L$-controlled invariant.
Then for any two totally bounded metric spaces $X$ and $Y$ and any integer $k\geq 0$, we have
\begin{align*}
\di\big(H_k\circ \filtf{\cP_X}{\mi^X},H_k\circ \filtf{\cP_Y}{\mi^Y}\big)&\leq \dHT\big((\cP_X, \mi^X), (\cP_Y, \mi^Y)\big)\leq 2L\cdot\dGH(X, Y) \\
\di\big(H_k\circ \filtf{\cPfin_X}{\mi^X},H_k\circ \filtf{\cPfin_Y}{\mi^Y}\big)&\leq \dHT\big((\cPfin_X, \mi^X), (\cPfin_Y, \mi^Y)\big)\leq 2L\cdot\dGH(X, Y).
\end{align*}
\end{thm}
Note that there are instances when the quantity in the middle vanishes yet the spaces $X$ and $Y$ are non-isometric; see Appendix~\ref{app:crushings} for results about this in terms of the notion of \emph{crushing} considered by Hausmann~\cite{Hausmann1995} and~\cite{memoli2021quantitative,AAF}.

\begin{cor}\label{cor:iqp_stability}
For any two totally bounded metric spaces $X$ and $Y$, for any $q, p \in [1, \infty]$, and for any integer $k \geq 0$, we have
\begin{align*}
\di\big(H_k\circ \filtf{\cP_X}{\iqp^X},H_k\circ \filtf{\cP_Y}{\iqp^Y}\big) &\leq \dHT\big((\cP_X,\,\iqp^X), (\cP_Y,\,\iqp^Y)\big) \leq 2\,\dGH(X, Y) \\
\di\big(H_k\circ \filtf{\cPfin_X}{\iqp^X},H_k\circ \filtf{\cPfin_Y}{\iqp^Y}\big) &\leq \dHT\big((\cPfin_X,\,\iqp^X), (\cPfin_Y,\,\iqp^Y)\big) \leq 2\,\dGH(X, Y).
\end{align*}
\end{cor}

\begin{proof}
From Lemma~\ref{lem:iqp_distortion_via_pushforward} and Lemma~\ref{lem:iqp_stability_wrt_dW} we know $\iqp$ is a $1$-controlled invariant, and then we apply Theorem~\ref{thm:general_stability}.
\end{proof}

\begin{thm}\label{thm:main}
Let $X$ and $Y$ be totally bounded metric spaces, let $p\in[1,\infty]$, and let $k \geq 0$ be an integer.
Then the $k$-th persistent homology of $\vrpf{X}$ and $\vrpf{Y}$ are $\varepsilon$-interleaved for any $\varepsilon\geq \dGH(X, Y)$, and similarly for $\cechpf{X}$ and $\cechpf{Y}$:
\begin{align*}
\di\big(H_k\circ\vrpf{X},H_k\circ\vrpf{Y}\big) &\leq \dHT\big((\cP_X, \diam_p^X), (\cP_Y, \diam_p^Y)\big) \leq 2\,\dGH(X,Y) \\
\di\big(H_k\circ\vrpffin{X},H_k\circ\vrpffin{Y}\big) &\leq \dHT\big((\cPfin_X, \diam_p^X), (\cPfin_Y, \diam_p^Y)\big) \leq 2\,\dGH(X,Y) \\
\di\big(H_k\circ\cechpf{X},H_k\circ\cechpf{Y}\big) &\leq \dHT\big((\cP_X, \rad_p^X), (\cP_Y, \rad_p^Y)\big) \leq 2\,\dGH(X,Y) \\
\di\big(H_k\circ\cechpffin{X},H_k\circ\cechpffin{Y}\big) &\leq \dHT\big((\cPfin_X, \rad_p^X), (\cPfin_Y, \rad_p^Y)\big) \leq 2\,\dGH(X,Y).
\end{align*}
\end{thm}

\begin{proof}
The $p$-Vietoris--Rips case follows from Corollary~\ref{cor:iqp_stability} by letting $q=p$, in which case $\mi_{p,p}^X=\diam_p^X$ (see Definition~\ref{defn:iqp}).
The $p$-\v{C}ech case follows from Theorem~\ref{thm:general_stability} since $\rad_p$ is a $1$-controlled invariant by Lemmas~\ref{lem:rad_p_dist} and~\ref{lem:stability_of_rad_p}.
\end{proof}

In the corollary below, we use Theorem~\ref{thm:general_stability} to show that a controlled invariant on all Radon probability measures produces a filtration at interleaving distance zero from that same invariant restricted to only finitely supported measures.

\begin{cor}
\label{cor:cP_cPfin_0-interleaved}
Let $X$ be a totally bounded metric space, let $k \geq 0$ be an integer, and let $\mi$ be an $L$-controlled invariant.
Then
\[\di(H_k\circ \filtf{\cP_X}{\mi^X}, H_k\circ \filtf{\cPfin_X}{\mi^X}) = \dHT\big((\cP_X, \mi^X), (\cPfin_X, \mi^X)\big) = 0.\]
\end{cor}

\begin{proof}
For any $\varepsilon>0$, let $U$ be a $\varepsilon$-net in $X$.
Then the stability theorem, Theorem~\ref{thm:general_stability}, shows
\begin{align*}
\dHT\big((\cP_X, \mi^X), (\cP_U), \mi^U)\big) &\leq 2L\cdot \dGH(X, U) \leq 2L \varepsilon \\
\dHT\big((\cPfin_X, \mi^X), (\cPfin_U, \mi^U)\big) &\leq 2L\cdot \dGH(X, U) \leq 2L \varepsilon
\end{align*}
Note $(\cP_U, \mi^U) = (\cPfin_U, \mi^U)$ since $U$ is finite.
Since $\dHT$ satisfies the triangle inequality (Proposition~\ref{prop:dht_pseudometric}), we have
\[\dHT\big((\cP_X, \mi^X), (\cPfin_X, \mi^X)\big)\leq 4L\varepsilon.\]
By letting $\varepsilon$ go to zero, we see $\dHT\big((\cP_X, \mi^X), (\cPfin_X, \mi^X)\big) = 0$.
It then follows from Lemma~\ref{lem:dht_bound_interleaving} that $\di(H_k\circ \filtf{\cP_X}{\mi^X}, H_k\circ \filtf{\cPfin_X}{\mi^X})=0$.
\end{proof}

\subsection{Consequence of stability: Tameness}
\label{sec:tameness}

We next determine conditions under which a persistence module $H_k\circ\filtf{\cP_X}{\mi^X}$ associated with a controlled invariant $\mi$ will be Q-tame.
In particular, we apply the results to the Vietoris--Rips and \v{C}ech metric thickenings.

\begin{cor}\label{cor:controlled_invariant_tameness}
Let $\mi$ be a controlled invariant and let $k \geq 0$ be an integer.
Suppose the persistence module $H_k\circ\filtf{\cP_V}{\mi^V}$ (resp.\ $H_k\circ\filtf{\cPfin_V}{\mi^V}$) is $Q$-tame for any finite metric space $V$.
Then for any totally bounded metric space $X$, the persistence module $H_k\circ\filtf{\cP_X}{\mi^X}$ (resp.\ $H_k\circ\filtf{\cPfin_X}{\mi^X}$) is $Q$-tame.
\end{cor}

\begin{proof}
Since $X$ is totally bounded, Theorem~\ref{thm:general_stability} implies that the persistence module $H_k\circ\filtf{\cP_X}{\mi^X}$ can be approximated arbitrarily well in the interleaving distance by the Q-tame persistence modules on finite $\varepsilon$-nets $V_\varepsilon$ as $\varepsilon$ goes to zero.
Then the result follows from Lemma~\ref{lem:qtame_approximation}.
\end{proof}

Now, we give a sufficient condition for the Q-tameness of the persistence module $H_k\circ\filtf{\cP_V}{\mi^V}$ over a finite metric space $V$.

\begin{cor}\label{cor:qtame_cts_finite}
Suppose $\mi$ is a controlled invariant such that for any finite metric space $V$, $\mi^V$ is a continuous function on $\cP_V$.
Then for any totally bounded metric space $X$, the persistence modules $H_k\circ\filtf{\cP_X}{\mi^X}$ and $H_k\circ\filtf{\cPfin_X}{\mi^X}$ are Q-tame for any integer $k\geq 0$.
\end{cor}

\begin{proof}
According to Corollary~\ref{cor:controlled_invariant_tameness}, it suffices to check Q-tameness over finite metric spaces.
If $V$ is a finite metric space, then we may identify $\cP_V$ with a simplex by Lemma~\ref{lem:fin-prob-simplex}.
Then $\filtf{\cP_V}{\mi^V}$ is a sublevel set filtration on a simplex, and since $\mi^V$ is continuous,
Theorem 2.22 of~\cite{chazal2016structure} shows $H_k\circ\filtf{\cP_V}{\mi^V}$ is Q-tame.
\end{proof}

Here we summarize the $Q$-tameness results related to relaxed Vietoris--Rips and \v{C}ech metric thickenings.

\begin{cor}\label{cor:totally_bounded_implies_q-tame}
Let $X$ be a totally bounded metric space, let $p\in [1, \infty]$, and let $k \geq 0$ be an integer.
The persistence modules $H_k\circ \vrpf{X}$, $H_k\circ \vrpffin{X}$, $H_k\circ \cechpf{X}$, and $H_k\circ \cechpffin{X}$ are Q-tame.
\end{cor}

\begin{proof}
For $H_k\circ \vrpf{X}$, when $p$ is finite, Lemma~\ref{lem:iqp_stability_wrt_dW} implies $\diam_p$ is a continuous function over any $\cP_X$ where $X$ is a bounded metric space.
Then we get the result by applying Corollary~\ref{cor:qtame_cts_finite}.
When $p= \infty$, Lemma~\ref{lem:finite_infty-thickenings_homeomorphic_to_complexes} shows the persistent homology associated with $\diam_\infty$ is Q-tame over finite metric spaces.
We then get the result by applying Corollary~\ref{cor:controlled_invariant_tameness}.

We obtain the case of $H_k\circ \vrpffin{X}$ from $H_k\circ \vrpf{X}$ by using the interleaving distance result in Corollary~\ref{cor:cP_cPfin_0-interleaved} along with the Q-tame approximation result in Lemma~\ref{lem:qtame_approximation}.

For $H_k\circ \cechpf{X}$, when $p$ is finite, Lemma~\ref{lem:stability_of_rad_p} implies $\rad_p$ is a continuous function over any $\cP_X$ where $X$ is a bounded metric space.
The rest is similar to the Vietoris--Rips case.
\end{proof}

\begin{remark}
The proof that $H_k \circ \cechpf{X}$ is Q-tame can be made more direct at one step.
For $U$ finite, to see that $H_k \circ \cechpf{U}$ is Q-tame, one could appeal to Theorem~\ref{thm_pCech_homotopy_equiv_simplicial_complexes} in Appendix~\ref{app:finite-Cech} instead of Theorem~2.22 of~\cite{chazal2016structure}.
\end{remark}

The Q-tame persistence modules given by this theorem allow us to discuss persistence diagrams, using the results of~\cite{chazal2016structure}.

\begin{cor}
\label{cor:p-bottleneck-stability}
Let $X$ be a totally bounded metric space, let $p\in [1, \infty]$, and let $k \geq 0$ be an integer.
Then $H_k\circ \vrpf{X}$ and $H_k\circ \vrpffin{X}$ have the same persistence diagram, denoted $\dgmvr_{k,p}(X)$.
Similarly, $H_k\circ \cechpf{X}$ and $H_k\circ \cechpffin{X}$ have the same persistence diagram, denoted $\dgmcech_{k,p}(X)$
\end{cor}

\begin{proof}
Persistence diagrams are well-defined for Q-tame persistence modules, so for any totally bounded metric space $X$, any $p \in [1,\infty]$, and any $k \geq 0$, by Corollary~\ref{cor:totally_bounded_implies_q-tame} we have persistence diagrams associated to $H_k\circ \vrpf{X}$ and $H_k\circ \vrpffin{X}$.
From Corollary~\ref{cor:cP_cPfin_0-interleaved} we know that the interleaving distance between $H_k\circ \vrpf{X}$ and $H_k\circ \vrpffin{X}$ is zero, and so the Isometry Theorem~\cite[Theorem 4.11]{chazal2016structure} implies that these persistence modules have the same (undecorated) persistence diagram, denoted $\dgmvr_{k,p}(X)$.
The same proof also works for \v{C}ech metric thickenings.
\end{proof}

Combining the Isometry Theorem (\cite[Theorem 4.11]{chazal2016structure}) with Theorem~\ref{thm:main} and Corollary~\ref{cor:totally_bounded_implies_q-tame}, we obtain the following.

\begin{cor}
If $X$ and $Y$ are totally bounded metric spaces, then for any $p \in [1, \infty]$ and any integer $k \geq 0$, we have
\begin{align*}
\dB\left( \dgmvr_{k,p}(X), \dgmvr_{k,p}(Y) \right) &\leq 2\,\dGH(X,Y) \\
\dB\left( \dgmcech_{k,p}(X), \dgmcech_{k,p}(Y) \right) &\leq 2\,\dGH(X,Y).
\end{align*}
\end{cor}

\subsection{Consequence of stability: Connecting $\infty$-metric thickenings and simplicial complexes}
\label{sec:comparability}

We show how the $\infty$-Vietoris--Rips and $\infty$-\v{C}ech metric thickenings recover the persistent homology of the Vietoris--Rips and \v{C}ech simplicial complexes.
Our Corollary~\ref{cor:infty-simplicial-dgms} answers~\cite[Conjecture~6.12]{AAF} in the affirmative.

We recall that $\vrp{X}{r}$ denotes the $p$-Vietoris--Rips metric thickening, that $\vrpfin{X}{r}$ denotes the $p$-Vietoris--Rips metric thickening for measures of finite support, and that $\vr{X}{r}$ denotes the Vietoris--Rips simplicial complex.
Theorem~\ref{thm:main} shows, for any $p\in[1,\infty]$, any $\delta>0$, and any finite $\delta$-net $U_\delta$ of a totally bounded metric space $X$, that
\[
d_\mathrm{I}\big(H_k\circ\vrpffin{X},H_k\circ\vrpffin{U_\delta}\big)\leq 2\delta.
\]
For $p=\infty$, by Lemma~\ref{lem:finite_infty-thickenings_homeomorphic_to_complexes}, we have $H_k\circ\vrppf{\infty}{U_\delta} \cong H_k\circ\vrf{U_\delta}$, so from the above we have
\[
d_\mathrm{I}\big(H_k\circ\vrppffin{\infty}{X} , H_k\circ\vrf{U_\delta}\big)\leq 2\delta.
\]
Now, by the triangle inequality for the interleaving distance, by the inequality above, and by the Gromov--Hausdorff stability of $X\mapsto H_k\circ\vrf{X}$~\cite{ChazalDeSilvaOudot2014,chazal2009gromov}, we have
\begin{align*}
& d_\mathrm{I}(H_k\circ\vrppffin{\infty}{X},H_k\circ\vrf{X}) \\
\leq\ & d_\mathrm{I}(H_k\circ\vrppffin{\infty}{X},H_k\circ\vrf{U_\delta}) + d_\mathrm{I}(H_k\circ\vrf{U_\delta},H_k\circ\vrf{X}) \\
\leq\ & 4\delta.
\end{align*}

Since this holds for any $\delta>0$, we find $d_\mathrm{I}(H_k\circ\vrppffin{\infty}{X},H_k\circ\vrf{X}) = 0$.
This implies that the bottleneck distance between persistence diagrams is $0$, and the (undecorated) diagrams are in fact equal (see for instance~\cite[Theorem 4.20]{chazal2016structure}).
We can apply Corollary~\ref{cor:cP_cPfin_0-interleaved} to get the same result for $\vrppf{\infty}{X}$ (measures with infinite support), and the same proof works in the \v{C}ech case.
We state this as the following theorem.

\begin{cor}
\label{cor:infty-simplicial-dgms}
For any totally bounded metric space $X$ and any integer $k \geq 0$, we have
\begin{align*}
d_\mathrm{I}(H_k\circ\vrppffin{\infty}{X},\ & H_k\circ\vrf{X}) = 0 \\
d_\mathrm{I}(H_k\circ\vrppf{\infty}{X},\ & H_k\circ\vrf{X}) = 0 \\
\dgmvr_{k,\infty}(X) &= \dgmvr_{k}(X) \\
d_\mathrm{I}(H_k\circ\cechppffin{\infty}{X},\ & H_k\circ\cechf{X}) = 0 \\
d_\mathrm{I}(H_k\circ\cechppf{\infty}{X},\ & H_k\circ\cechf{X}) = 0 \\
\dgmcech_{k,\infty}(X) &= \dgmcech_{k}(X).
\end{align*}
\end{cor}

\begin{remark}
Whereas the $\infty$-metric thickenings $\vrppf{\infty}{X}$ and the simplicial complexes $\vrf{X}$ yield persistence modules with an interleaving distance of 0, we are in the interesting landscape where the metric thickenings $\vrpf{X}$ may yield something new and different for $p<\infty$.
For example, in Section~\ref{sec:spheres} we explore the new persistence modules that arise for $p$-metric thickenings of Euclidean spheres $(\Sp^n,\ell_2)$ in the case $p=2$.
\end{remark}

\begin{example}
Let $X=\Sp^1$ be the circle.
The persistent homology diagrams of the simplicial complex filtrations $\vrf{\Sp^1}$ and $\cechf{\Sp^1}$ are known from~\cite{AA-VRS1}, and therefore Corollary~\ref{cor:infty-simplicial-dgms} gives the persistent homology diagrams of $\vrppf{\infty}{\Sp^1}$ and $\cechppf{\infty}{\Sp^1}$.
However, although $\vr{\Sp^1}{r}$ and $\cech{\Sp^1}{r}$ are known to obtain the homotopy types of all odd-dimensional spheres as $r$ increases, the homotopy types of $\vrpp{\infty}{\Sp^1}{r}$ and $\cechpp{\infty}{\Sp^1}{r}$ are still not proven.
\end{example}

\begin{question}
Is the simplicial complex $\vr{X}{r}$ always homotopy equivalent to $\vrppfin{\infty}{X}{r}$ or $\vrpp{\infty}{X}{r}$?
Compare with~\cite[Remark~3.3]{AAF}.
\end{question}

\begin{question}
\cite[Theorem~5.2]{lim2020vietoris} states that for $X$ compact, bars in the persistent homology of the simplicial complex filtration $\vrf{X}$ are of the form $(a,b]$ or $(a,\infty)$.
Is the same true for $\vrppf{\infty}{X}$?
Note that we are using the $<$ convention, i.e.\ a simplex is included in $\vr{X}{r}$ when its diameter is strictly less than $r$, and a measure is included in $\vrpp{\infty}{X}{r}$ when the diameter of its support is strictly less than $r$.
\end{question}

%%%%%%%%%%%%%%%

\subsection{The proof of stability}
\label{ssec:proof-stability}

As mentioned above, the construction of interleaving maps between filtrations of metric thickenings is more intricate than in the case of simplicial complexes.
A main idea behind the proof of our stability result, Theorem~\ref{thm:general_stability}, is to approximate the metric space $X$ by a finite subspace (a net) $U$.
The advantage is that it is easier to construct a continuous map with domain $\cP_U$ than one with domain $\cP_X$.
The crucial next step is the construction of a continuous map from $\cP_X$ to $\cP_U$ by using a partition of unity subordinate to an open covering by $\delta$-balls.
This map distorts distances by a controlled amount, allowing us to approximate measures of $\cP_X$ in $\cP_U$.

We begin with some necessary lemmas.

\begin{lem}\label{lem:Ptau_cts_via_partition}
Let $U$ be a finite $\delta$-net of a bounded metric space $X$, and let $\{\zeta_u^U\}_{u\in U}$ be a continuous partition of unity subordinate to the open covering $\bigcup\limits_{u\in U} B(u;\delta)$ of $X$.

We have the continuous map $\Phi_U \colon \cP_X \to \cP_U$ defined by
\[\alpha \mapsto \sum_{u\in U} \int_X\zeta_u^U(x)\,\alpha(dx) \cdot\delta_{u}.\]
For any $q\in [1, \infty]$, we have $\dWq(\alpha, \Phi_U(\alpha))< \delta$.
\end{lem}

\begin{proof}
The map $\Phi_U$ is well-defined as
\[\sum_{u\in U} \int_X\zeta_u^U(x)\,\alpha(dx)  =\int_X\,\sum_{u\in U} \,\zeta_u^U(x)\,\alpha(dx) = \int_X\, \alpha(dx) = 1.\]
For the continuity, since the weak topology on $\cP_X$ and $\cP_U$ can be metrized, it suffices to show for a weakly convergent sequence $\alpha_n\in \cP_X$, that the image $\Phi_U(\alpha_n)$ is also weakly convergent.
As $U$ is a finite metric space, it suffices to show that for any fixed $u_0\in U$, the sequence of real numbers $\big(\Phi_U(\alpha_n)(\{u_0\})\big)_n$ is itself convergent.
Note that
\[ \Phi_U(\alpha_n)(\{u_0\}) = \int_X\zeta_{u_0}^U(x)\,\alpha_n(dx).\]
Since $\zeta_{u_0}^U$ is a bounded continuous function on $X$, we obtain the desired convergence through the weak convergence of $\alpha_n$.

Lastly, we must show that $\dWq(\alpha, \Phi_U(\alpha))< \delta$.
For any $u\in U$, we use the notation $w_u^\alpha := \int_X \zeta_u^U(x)\alpha(dx)$, so that $\Phi_U(\alpha)=\sum_{u\in U} w_u^\alpha \cdot\,\delta_u$ and
\[ \sum_{u\in U} w_u^\alpha = \sum_{u\in U} \int_X \zeta_u^U(x)\,\alpha(dx) = \int_X \sum_u \zeta_u^U(x)\,\alpha(dx) = \int_X \alpha(dx) = 1.\]
Let $\zeta_u^U\,\alpha$ denote the measure such that $\zeta_u^U\,\alpha(B) = \int_B\zeta_u^U(x)\alpha(dx)$ for any measurable set $B\subseteq X$.
We have
\begin{equation*}
\alpha = \sum_{u\in U} \zeta_u^U\,\alpha = \sum_{\substack{u\in U,          \\ w^{\alpha}_u\neq 0}}  w_u^\alpha \cdot\,\left(\tfrac{1}{w^{\alpha}_u}\,\zeta_u^U\,\alpha \right).
\end{equation*}
From this we get,
\begin{align*}
\dWq(\alpha, \Phi_U(\alpha)) & = \dWq\left(\sum_{\substack{u\in U, \\ w^{\alpha}_u\neq 0}}  w_u^\alpha \cdot\,\left(\tfrac{1}{w^{\alpha}_u}\,\zeta_u^U\,\alpha \right), \sum_{\substack{u\in U,\\ w^{\alpha}_u\neq 0}}  w_u^\alpha \cdot\,\delta_u\right) \\
&  \leq \left( \sum_{\substack{u\in U, \\ w^{\alpha}_u\neq 0}} w_u^\alpha\cdot\left( \dWq\left(\tfrac{1}{w^{\alpha}_u}\,\zeta_u^U\,\alpha ,\, \delta_u\right) \right)^q \right)^\frac{1}{q}\\
& < \delta.
\end{align*}
The first inequality is by Lemma~\ref{lem:bound-distance-convex-comb}.
The last one comes from the fact that each $\tfrac{1}{w^{\alpha}_u}\,\zeta_u^U\,\alpha$ is a probability measure supported in $B(u;\delta)$, and thus $\dWq\left(\tfrac{1}{w^{\alpha}_u}\,\zeta_u^U\,\alpha ,\, \delta_u\right) < \delta$.
\end{proof}

For any two totally bounded metric spaces $X$ and $Y$, through partitions of unity we can build continuous maps between $\cP_X$ and $\cP_Y$, as follows.
Let $X$ and $Y$ be two totally bounded metric spaces, let $\eta > 2\,\dGH(X, Y)$, and let $\delta>0$.
We fix finite $\delta$-nets $U\subseteq X$ of $X$ and $V\subseteq Y$ of $Y$.
By the triangle inequality, we have $\dGH(U, V)<\frac{\eta}{2}+2\delta$.
So there exist maps $\varphi\colon U\to V$ and $\psi\colon V\to U$ with
\[\max\big(\mathrm{dis}(\varphi),\mathrm{dis}(\psi),\mathrm{codis}(\varphi,\psi)\big)\leq \eta +4\delta.\]
We then define the maps $\widehat{\Phi}\colon \cP_X\to \cP_Y$ and $\widehat{\Psi}\colon \cP_Y\to \cP_X$ to be
\[\widehat{\Phi}:=\iota_V\circ \varphi_\sharp \circ \Phi_U \quad\text{and}\quad \widehat{\Psi} :=\iota_U \circ \psi_\sharp\circ \Phi_V,\]
where $\iota_U \colon \cP_U\hookrightarrow \cP_X$ and $\iota_V \colon \cP_V\hookrightarrow \cP_Y$ are inclusions and $\Phi_U$ and $\Phi_V$ are the maps defined in Lemma~\ref{lem:Ptau_cts_via_partition}:
\begin{center}
\begin{tabular}{cc}
\begin{minipage}{1.5in}
\begin{displaymath}
\leftline{\xymatrix{
      \cP_X \ar@{.>}[r]^{\widehat{\Phi}} \ar@{->}[d]_{\Phi_U}& \cP_Y \\
      \cP_U \ar@{->}[r]_{\varphi_\sharp} &  \cP_V \ar@{^{(}->}[u]_{\iota_V}
    }}
\end{displaymath}
\end{minipage}
&
\begin{minipage}{1.5in}
\begin{displaymath}
\leftline{\xymatrix{
      \cP_Y \ar@{.>}[r]^{\widehat{\Psi}} \ar@{->}[d]_{\Phi_V}& \cP_X \\
      \cP_V \ar@{->}[r]_{\psi_\sharp} &  \cP_U \ar@{^{(}->}[u]_{\iota_U}
    }}
\end{displaymath}
\end{minipage}
\end{tabular}
\end{center}

Then both $\widehat{\Phi}$ and $\widehat{\Psi}$ are continuous maps by Lemma~\ref{lem:Ptau_cts_via_partition} and by the continuity of pushforwards of continuous maps.

The use of the finite $\delta$-nets $U$ and $V$ is important because they can be compared by continuous maps $\varphi$ and $\psi$, which yield continuous maps $\varphi_\sharp$ and $\psi_\sharp$.
The following lemma shows that $\widehat{\Phi}$ and $\widehat{\Psi}$ are homotopy equivalences.

\begin{lem}\label{lem:homotopy_scaffording}
With the above notation, we have homotopy equivalences $\widehat{\Psi}\circ \widehat{\Phi} \simeq \mathrm{id}_{\cP_X}$ and $\widehat{\Phi}\circ \widehat{\Psi} \simeq \mathrm{id}_{\cP_Y}$ via the linear families $H^X_t\colon \cP_X\times [0, 1]\to \cP_X$ and $H^Y_t\colon \cP_Y\times [0, 1]\to \cP_Y$ given by
\[H^X_t(\alpha) := (1-t)\,\widehat{\Psi}\circ \widehat{\Phi}(\alpha)+ t\,\alpha \quad\text{and}\quad H^Y_t(\beta) := (1-t)\,\widehat{\Phi}\circ \widehat{\Psi}(\beta)+ t\,\beta.\]
Moreover, for any $q\in [1, \infty]$ and any $t\in [0, 1]$, we have $\dWq(H_t^X(\alpha), \alpha) < \eta + 6\delta$ and $d_{\mathrm{W}, q}^Y(H_t^Y(\beta), \beta) < \eta + 6\delta.$
\end{lem}

\begin{proof}
The homotopies $H^X_t$ and $H^Y_t$ are continuous by Proposition~\ref{prop:linear_homotopies}, since $\widehat{\Phi}$ and $\widehat{\Psi}$ are continuous.
We will only present the estimate $\dWq(H_t^X(\alpha), \alpha) < \eta + 6\delta$, as the other inequality can be proved in a similar way.
We first calculate the expression for $\widehat{\Psi} \circ\widehat{\Phi} (\alpha)$, obtaining
\begin{align*}
\widehat{\Psi} \circ\widehat{\Phi} (\alpha)
& = \psi_\sharp\left(\Phi_V\left(\varphi_\sharp\left(\Phi_U(\alpha)\right)\right)\right)  \\
& = \psi_\sharp\left(\Phi_V\left(\varphi_\sharp\left(\sum_{u\in U} \int_X\zeta^U_u(x)\,\alpha(dx)\cdot\delta_{u}\right)\right)\right)  \\
& = \psi_\sharp\left(\Phi_V\left(\sum_{u\in U} \int_X\zeta^U_u(x)\,\alpha(dx)\cdot\delta_{\varphi(u)}\right)\right)  \\
& =\psi_\sharp\left(\sum_{v\in V}\int_Y \zeta^V_v(y)\, \left(\sum_{u\in U} \int_X\zeta^U_u(x)\,\alpha(dx)\cdot\delta_{\varphi(u)}\right)(dy) \cdot \delta_v\right) \\
& = \psi_\sharp\left(\sum_{v\in V}\bigg( \sum_{u\in U} \int_X\zeta^U_u(x)\,\alpha(dx)\cdot \int_Y\zeta_v^V(y)\cdot \delta_{\varphi(u)}(dy)\bigg)\cdot \delta_v\right)  \\
& = \psi_\sharp\left(\sum_{v\in V}\bigg( \sum_{u\in U} \int_X\zeta^U_u(x)\,\alpha(dx)\cdot \zeta^V_v({\varphi(u)})\bigg)\cdot \delta_v\right)  \\
& =\sum_{v\in V}\bigg( \sum_{u\in U} \int_X\zeta^U_u(x)\,\alpha(dx)\cdot \zeta^V_v({\varphi(u)})\bigg)\cdot \delta_{\psi(v)} \\
& =\sum_{v\in V}\bigg( \sum_{u\in U} w^{\alpha}_u\cdot \zeta^V_v({\varphi(u)})\bigg)\cdot \delta_{\psi(v)},
\end{align*}
where we again use the notation $w^{\alpha}_u:=\int_X \zeta_u^U(x)\alpha(dx)$. 
We have
\[H_t^X(\alpha) = (1-t)\sum_{v\in V}\sum_{u\in U} \bigg(  w^{\alpha}_u\cdot\zeta^V_v({\varphi(u)})\bigg)\cdot \delta_{\psi(v)} + t\, \alpha.\]

Now we show that $H_t^X(\alpha)$ is a convex combination of probability measures in $\cP_X$ and hence itself lies in $\cP_X$.
That is, the sum of coefficients in front of all $\delta_{\psi(v)}$ and $\alpha$ in the above formula for $H_t^X$ is $1$:
\[(1-t)\sum_{v\in V}\sum_{u\in U} \bigg(  w^{\alpha}_u\cdot\zeta^V_v({\varphi(u)})\bigg) + t = (1-t)\sum_{u\in U} w^{\alpha}_u + t = 1.\]
Next, we will bound the $q$-Wasserstein distance $\dWq(H_t^X(\alpha), \alpha)$.
For this, we first rewrite $\alpha$ as follows:
\begin{align*}
\alpha & =(1-t) \sum_{u\in U} \zeta_u^U\,\alpha  + t\,\alpha \\
& = (1-t)\sum_{v\in V}\sum_{u\in U} \,\zeta^V_v({\varphi(u)})\cdot\,(\zeta_u^U\,\alpha )  + t\, \alpha \\
& =(1-t)\sum_{v\in V}\sum_{\substack{u\in U, \\ w^{\alpha}_u\neq 0}} \,\bigg(w_u^\alpha\cdot \zeta^V_v({\varphi(u)})\bigg)\cdot\,\left(\tfrac{1}{w^{\alpha}_u}\,\zeta_u^U\,\alpha \right)  + t\, \alpha. \end{align*}
Based on these observations, we then apply the inequality from Lemma~\ref{lem:bound-distance-convex-comb} to obtain
\begin{align*}
& \dWq(H_t^X(\alpha), \alpha) \\
=\,& \dWq\bigg((1-t)\sum_{v\in V}\sum_{\substack{u\in U,  \\ w^{\alpha}_u\neq 0}} \bigg(w_u^\alpha\cdot \zeta^V_v({\varphi(u)})\bigg)\cdot \delta_{\psi(v)} + t\, \alpha,\, \\
& \quad\quad\quad \quad\quad\quad(1-t)\sum_{v\in V}\sum_{\substack{u\in U, \\ w^{\alpha}_u\neq 0}} \,\bigg(w_u^\alpha\cdot \zeta^V_v({\varphi(u)})\bigg)\cdot\,\left(\tfrac{1}{w^{\alpha}_u}\,\zeta_u^U\,\alpha \right)  + t\, \alpha\bigg) \\
\leq\, & \left( (1-t)\sum_{v\in V} \sum_{\substack{u\in U,  \\ w^{\alpha}_u\neq 0}}  \bigg(w_u^\alpha\cdot \zeta^V_v({\varphi(u)})\bigg)\cdot \bigg(\dWq\big( \delta_{\psi(v)} ,\,\tfrac{1}{w^{\alpha}_u}\,\zeta_u^U\,\alpha\big)\bigg)^p\right)^\frac{1}{p} \\
\leq\, & \left( (1-t)\sum_{v\in V} \sum_{\substack{u\in U, \\ w^{\alpha}_u\neq 0}}  \bigg(w_u^\alpha\cdot \zeta^V_v({\varphi(u)})\bigg)\cdot \bigg(\dWq\big( \delta_{\psi(v)} ,\,\delta_u\big)+\dWq\big( \delta_{u} ,\,\tfrac{1}{w^{\alpha}_u}\,\zeta_u^U\,\alpha\big)\bigg)^p\right)^\frac{1}{p}.
\end{align*}
For each non-zero term in the summand of the last expression, we have $d_Y(v,\,\varphi(u)) < \delta$ which implies $d_X(\psi(v), u)< 5\delta + \eta$ via the codistortion assumption.
Therefore, $\dWq\big( \delta_{\psi(v)} ,\,\delta_u\big)< 5\delta + \eta$ as well.
Moreover, since $\frac{1}{w^{\alpha}_u}\,\zeta_u^U\,\alpha$ is a probability measure supported in the $\delta$-ball centered at the point $u$, we have $\dWq\big( \delta_{u} ,\,\frac{1}{w^{\alpha}_u}\,\zeta_u^U\,\alpha\big)<\delta$.
Therefore, we have $\dWq(H_t^X(\alpha), \alpha)<6\delta + \eta$.
\end{proof}

\begin{remark}
The above lemma still holds if we replace $\cP_X$ by $\cPfin_X$ and $\cP_Y$ by $\cPfin_Y$ as all related maps naturally restrict to the set of finitely supported measures, $\cPfin_X$.
\end{remark}

We are now ready to prove our main result, Theorem~\ref{thm:general_stability}.

\begin{proof}[Proof of Theorem~\ref{thm:general_stability}]
For $\mi$ an $L$-controlled invariant, for totally bounded  metric spaces $X$ and $Y$, and for any integer $k\geq 0$, we must show
\begin{align*}
\di\big(H_k\circ \filtf{\cP_X}{\mi^X},H_k\circ \filtf{\cP_Y}{\mi^Y}\big)&\leq \dHT\big((\cP_X, \mi^X), (\cP_Y, \mi^Y)\big)\leq 2L\cdot\dGH(X, Y) \\
\di\big(H_k\circ \filtf{\cPfin_X}{\mi^X},H_k\circ \filtf{\cPfin_Y}{\mi^Y}\big)&\leq \dHT\big((\cPfin_X, \mi^X), (\cPfin_Y, \mi^Y)\big) \leq 2L\cdot\dGH(X, Y).
\end{align*}
We prove only the first line above, as the finitely supported case in the second line has a nearly identical proof.
The inequality involving $\di$ and $\dHT$ follows from Lemma~\ref{lem:dht_bound_interleaving}.
Hence it suffices to prove the inequality $\dHT\big((\cP_X, \mi^X), (\cP_Y, \mi^Y)\big)\leq 2L\cdot\dGH(X, Y)$ involving $\dHT$ and $\dGH$.

Following the construction in Lemma~\ref{lem:homotopy_scaffording}, let $\eta > 2\cdot \dGH(X, Y)$ and $\delta>0$.
We fix finite $\delta$-nets $U\subset X$ of $X$ and $V\subset Y$ of $Y$.
By the triangle inequality, we have $\dGH(U, V)<\frac{\eta}{2}+2\delta$, and so there exist maps $\varphi:U\to V$ and $\psi\colon V\to U$ with
\[\max\big(\mathrm{dis}(\varphi),\mathrm{dis}(\psi),\mathrm{codis}(\varphi,\psi)\big)\leq \eta +4\delta.\]
For any $\delta>0$, we will show that $(\cP_X, \mi^X)$ and $(\cP_Y, \mi^Y)$ are $(\eta+6\delta)\, L$--homotopy equivalent.
By the definition of a $\delta$-homotopy (Definition~\ref{defn:delta-homotopy}), it suffices to show that
\begin{itemize}
    \item the map $\widehat{\Phi}$ is a $(\eta+ 6\delta)\, L$--map from $(\cP_X, \mi^X)$ to $(\cP_Y, \mi^Y)$,
    \item the map $\widehat{\Psi}$ is a $(\eta+ 6\delta)\, L$--map from $(\cP_Y, \mi^Y)$ to $(\cP_X, \mi^X)$,
    \item the map $\widehat{\Psi}\circ\widehat{\Phi}\colon  \cP_X\to \cP_X$ is $(2\eta + 12\delta)\, L$--homotopic to $\mathrm{id}_{\cP_X}$ with respect to $( \mi^X,  \mi^X)$,
    \item the map $\widehat{\Phi}\circ\widehat{\Psi}\colon  \cP_X\to \cP_X$ is $(2\eta + 12\delta)\, L$--homotopic to $\mathrm{id}_{\cP_Y}$ with respect to $( \mi^Y,  \mi^Y)$.
\end{itemize}
We will only present the proof for the first and third items; the other two can be proved similarly.
We have
\[ \mi^Y(\widehat{\Phi}(\alpha)) = \mi^Y(\varphi_\sharp\circ\Phi_U(\alpha))
\leq \mi^X(\Phi_U(\alpha)) + (\eta + 4\delta)\, L
\leq \mi^X(\alpha) + (\eta + 6\delta)\, L, \]
where the first inequality is from the stability of the invariant $\mi$ under pushforward and from the bound on the distortion of $\varphi$, and where the second inequality is from the stability of the invariant with respect to Wasserstein distance and from the bound $\dWq(\alpha, \Phi_U(\alpha))< \delta$ in Lemma~\ref{lem:Ptau_cts_via_partition}.
This proves the first item.

For the third item, we use the homotopy $H_t^X$ in Lemma~\ref{lem:homotopy_scaffording}.
Then it suffices to show for any fixed $t\in [0, 1]$, the map $H_t^X$ is a $(2\eta + 12\delta)\, L$--map from $(\cP_X, \mi^X)$ to $(\cP_X, \mi^X)$, that is,
\[ \mi^X(H_t^X(\alpha)) \leq \mi^X(\alpha) + (2\eta + 12\delta)\, L.\]
This comes from the inequality \(d_{\mathrm{W}, \infty}^X(H_t^X(\alpha), \alpha)< \eta + 6\delta\) from Lemma~\ref{lem:homotopy_scaffording}, and the from the stability assumption of $\mi$ with respect to Wasserstein-distance.

Now, since $(\cP_X, \mi^X)$ and $(\cP_Y, \mi^Y)$ are $(\eta+6\delta)\, L$--homotopy equivalent for any $\eta > 2\cdot \dGH(X, Y)$ and $\delta>0$, it follows from the definition of the $\dHT$ distance (Definition~\ref{defn:dHT}) that $\dHT\big((\cP_X, \mi^X), (\cP_Y, \mi^Y)\big)\leq 2L\cdot\dGH(X, Y)$.
This completes the proof.
\end{proof}

\section{A Hausmann type theorem for $2$-Vietoris-Rips and $2$-\v{C}ech thickenings of Euclidean submanifolds}
\label{sec:hausmann}

Hausmann's theorem~\cite{Hausmann1995} states that if $X$ is a Riemannian manifold and if the scale $r>0$ is sufficiently small (depending on the curvature $X$), then the Vietoris--Rips simplicial complex $\vr{X}{r}$ is homotopy equivalent to $X$.
A short proof using an advanced version of the nerve lemma is given in~\cite{virk2021rips}.
Versions of Hausmann's theorem have been proven for $\infty$-metric thickenings in~\cite{AAF}, where $X$ is equipped with the Riemannian metric, and in~\cite{AM}, where $X$ is equipped with a Euclidean metric.
In this section we prove a Hausmann type theorem for the $2$-Vietoris--Rips and $2$-\v{C}ech metric thickenings.
Our result is closest to that in~\cite{AM} (now with $p=2$ instead of $p=\infty$): we work with any Euclidean subset $X$ of positive reach.
This includes (for example) any embedded $C^k$ submanifold of $\R^n$ for $k\ge 2$, with or without boundary~\cite{Thale}.

We begin with a definition and some lemmas that will be needed in the proof.

Let $\alpha$ be a measure in $\cP_1(\R^n)$, the collection of all Radon measures with finite first moment.
For any coordinate function $x_i$, where $1\leq i\leq n$, we have
\[
    \int_{\R^n} |x_i|\, \alpha(dx) \leq \int_{\R^n} \norm{x}\,\alpha(dx)<\infty.
\]
Therefore, the \emph{Euclidean mean} map $m\colon \cP_1(\R^n)\to \R^n$ given by the following vector-valued integral is well-defined:
\[
    m(\alpha):= \int_{\R^n} x\, \alpha(dx).
\]

\begin{lem}
Let $X$ be a metric space and let $f\colon X\to \R^n$ be a bounded continuous function.
Then the induced map $m\circ f_\sharp \colon \cP_X \to \R^n$ is continuous.
\end{lem}

\begin{proof}
For a sequence $\alpha_n$ that weakly converges to $\alpha$, we have the following vector-valued integral:
\[
m\circ f_\sharp(\alpha_n) = \int_{\mathbb{R}^n}x\, f_\sharp(\alpha_n)(dx) = \int_X f(x)\,\alpha_n(dx).\]
As $f$ is bounded and continuous, the above limit converges to $m\circ f_\sharp(\alpha)$ as $\alpha_n$ converges to $\alpha$.
Therefore, the map $m\circ f_\sharp$ is continuous.
\end{proof}

\begin{lem}\label{lem:dHBoundViaDiamp}
Let $\alpha$ be a probability measure in $\cP_1(\R^n)$.
Then for $p\in [1, \infty]$, there is some $z \in \mathrm{supp}(\alpha)$ with
\[\norm{m(\alpha) - z} \leq \diam_p(\alpha).\]
\end{lem}

\begin{proof}
As $\diam_1(\alpha)\leq \diam_p(\alpha)$ for all $p\in[1,\infty]$, it suffices to show $ \norm{m(\alpha) - z} \leq \diam_1(\alpha)$.
We consider the following formula:
\begin{align*}
\int_{\R^n}\norm{m(\alpha) - z} \alpha(dz) & = \int_{\R^n} \norm{\int_{\R^n} (x - z) \alpha(dx)} \alpha(dz)\\
& \leq \int_{\R^n}\int_{\R^n}\norm{ (x - z)} \alpha(dx)\alpha(dz)\\
& = \diam_1(\alpha).
\end{align*}
So there must be some $z\in \mathrm{supp}(\alpha)$ with $\norm{m(\alpha) - z} \leq \diam_1(\alpha)$.
\end{proof}

\begin{lem}\label{lem:Frechet_2_decomposition}
For any probability measure $\alpha\in \cP_1(\R^n)$ and for any $x\in \R^n$, we can write the associated squared $2$-Fr\'{e}chet function $F_{\alpha, 2}^2(x)$ as
\begin{equation*}
F_{\alpha, 2}^2(x) = \norm{x - m(\alpha)}^2 +F_{\alpha, 2}^2(m(\alpha)).
\end{equation*}
\end{lem}

\begin{proof}
By the linearity of the inner product, we have
\begin{align*}
F_{\alpha, 2}^2(x) & := \int_{\R^n}\norm{x - y}^2 \alpha(dy) \\
& =\int_{\R^n}\Big(\norm{x}^2 - 2\langle x, y\rangle + \norm{y}^2\Big)\alpha(dy) \\
& =\int_{\R^n}\Big(\norm{x}^2 - 2\langle x, y\rangle +\norm{m(\alpha)}^2 - \norm{m(\alpha)}^2 + \norm{y}^2\Big)\alpha(dy) \\
& = \norm{x}^2 - 2 \langle x, m(\alpha)\rangle+ \norm{m(\alpha)}^2 + \int_{\R^n}\Big(-\norm{m(\alpha)}^2 + \norm{y}^2\Big)\alpha(dy) \\
& = \norm{x - m(\alpha)}^2 + \int_{\R^n}\Big(\norm{m(\alpha)}^2 - 2\langle m(\alpha), y\rangle + \norm{y}^2\Big)\alpha(dy) \\
& = \norm{x - m(\alpha)}^2 + \int_{\R^n}\norm{m(\alpha) - y}\alpha(dy) \\
& =\norm{x - m(\alpha)}^2 +F_{\alpha, 2}^2(m(\alpha)).
\end{align*}
\end{proof}

Recall the \emph{medial axis} of $X$ is defined as the closure
\[\mathrm{med}(X) = \overline{\{y\in \R^n~|~\exists\, x_1 \neq x_2 \in X\,\text{with } \norm{y - x_1} = \norm{y - x_2} = \inf_{x\in X}\norm{y-x}\}}.\]
The \emph{reach} $\tau$ of $X$ is the closest distance $\tau := \inf_{x\in X,y\in\mathrm{med}(X)}\norm{x-y}$ between points in $X$ and $\mathrm{med}(X)$.
Let $U_\tau(X)$ be the $\tau$-neighborhood of $X$ in $\R^n$, that is, the set of points $y\in\R^n$ such that there is some $x\in X$ with $\norm{y - x}<\tau$.
The definition of reach implies that for any point $y$ in the open neighborhood $U_\tau(X)$, there is a unique closest point $x\in X$.
The associated nearest projection map $\pi\colon U_\tau(X)\to X$ is continuous; see~\cite[Theorem~4.8(8)]{federer1959curvature} or~\cite[Lemma~3.7]{AM}.

\begin{lem}
\label{lem:boundedness_of_rad_and_diam}
Let $X$ be a bounded subset of $\R^n$ with reach $\tau(X)>0$.
Let $\alpha \in \cP_X$ have its Euclidean mean $m(\alpha)$ in the neighborhood $U_\tau(X)$.
Then along the linear interpolation family
\[\alpha_t := (1-t)\alpha + t\,\delta_{\phi(\alpha)},\]
where $\phi$ is the composition $\pi\circ m\colon \cP_X\to X$, both $\diam_2$ and $\rad_2$ obtain their maximum 
at $t=0$.
\end{lem}

\begin{proof}
According to Lemma~\ref{lem:Frechet_2_decomposition}, for any $x\in X$ we have
\[F_{\alpha, 2}^2(x) = \norm{x - m(\alpha)}^2 +F_{\alpha, 2}^2(m(\alpha)).\]
As $m(\alpha)$ is inside $U_\tau(X)$, the first term $\norm{x - m(\alpha)}$ is minimized over $x\in X$ at $x=\phi(\alpha)$, and hence so is $F_{\alpha, 2}(x)$.
Therefore, we have $\rad_2(\alpha) = \inf_{x\in X} F_{\alpha, 2}(x) = F_{\alpha, 2}(\phi(\alpha))$.
So for $\rad_2(\alpha_t)$, we have
\begin{align*}
(\rad_2(\alpha_t))^2 &\leq F_{\alpha_t,2}^2(\phi(\alpha)) \\
&= \int_{\R^n} \norm{x - \phi(\alpha)}^2 \alpha_t(dx) \\
&=(1-t)\int_{\R^n} \norm{x - \phi(\alpha)}^2 \alpha(dx) + t \int_{\R^n} \norm{x - \phi(\alpha)}^2\delta_{\phi(\alpha)}(dx) \\
&= (1-t)\int_{\R^n} \norm{x - \phi(\alpha)}^2\alpha(dx) \\
&= (1-t) F_{\alpha,2}^2(\phi(\alpha)) \\
&\leq (\rad_2(\alpha))^2.
\end{align*}

We now consider $\diam_2$.
Recall that $F_{\alpha, 2}(\phi(\alpha))\leq F_{\alpha, 2}(x)$ for all $x\in X$.
From this, we have $(\diam_2(\alpha))^2 = \int_X F^2_{\alpha, 2}(x)\,\alpha(dx) \ge F_{\alpha, 2}^2(\phi(\alpha))$.
Therefore
\begin{align*}
\big(\diam_2(\alpha_t)\big)^2 & = \iint_{\R^n\times \R^n}\norm{x - y}^2\alpha_t(dx)\ \alpha_t(dy) \\
& = (1-t)^2\iint_{\R^n\times \R^n}\norm{x - y}^2 \alpha(dx)\,\alpha(dy) \\
& \quad \quad + 2t(1-t)\iint_{\R^n\times \R^n} \norm{x - y}^2 \alpha(dx)\,\delta_{\phi(\alpha)}(dy) \\
& \quad \quad + t^2\iint_{\R^n\times \R^n}\norm{x - y}^2\delta_{\phi(\alpha)}(dx)\,\delta_{\phi(\alpha)}(dy) \\
& = (1-t)^2 \big(\diam_2(\alpha)\big)^2 + 2t(1-t)\, F_{\alpha, 2}^2(\phi(\alpha)) \\
& \leq (1-t)^2\big(\diam_2(\alpha)\big)^2 +2t(1-t)\,\big(\diam_2(\alpha)\big)^2 \\
& = (1-t^2)\big(\diam_2(\alpha)\big)^2 \\
& \leq\big(\diam_2(\alpha)\big)^2.
\end{align*}
\end{proof}

\begin{thm}\label{thm:vr2_Hausmann}
Let $X$ be a bounded subset of $\R^n$ with positive reach $\tau$.
Then for all $0<r\le\tau$, the isometric embeddings $X\hookrightarrow\vrpp{2}{X}{r}$, $X\hookrightarrow\vrppfin{2}{X}{r}$, $X\hookrightarrow\cechpp{2}{X}{r}$, and $X\hookrightarrow\cechppfin{2}{X}{r}$ are homotopy equivalences.
\end{thm}

\begin{proof}
We begin with the $2$-Vietoris--Rips metric thickening.
Let $\alpha$ be a measure in $\vrpp{2}{X}{r}$.
Lemma~\ref{lem:dHBoundViaDiamp} shows there exists some $x\in X$ with $\|m(\alpha)- x\|\leq\diam_2(\alpha) < r \leq \tau$.
Hence $m(\alpha)$ is in the neighborhood $U_\tau(X)$, and $\phi:= \pi\circ m$ is well-defined on $\vrpp{2}{X}{r}$.
As both $\pi$ and $m$ are continuous, so is their composition $\phi$.
Let $i$ be the isometric embedding of $X$ into $\vrpp{2}{X}{r}$.
Clearly we have $\phi \circ i = \mathrm{id}_X$.
Then it suffices to show the homotopy equivalence $i\circ \phi\ \simeq\ \mathrm{id}_{\vrpp{2}{X}{r}}$.

For this, we use the the linear family $\big(\alpha_t\big)_{t\in[0,1]}$ where $\alpha_t := (1-t)\alpha + t \delta_{\phi(\alpha)}$.
According to Lemma~\ref{lem:boundedness_of_rad_and_diam}, we know this family is inside $\vrpp{2}{X}{r}$, and therefore $\alpha_t$ is well-defined.
By Proposition~\ref{prop:linear_homotopies}, this is a continuous family as the map $\phi$ is continuous.
Therefore, the linear family indeed provides a homotopy $i\circ \phi\ \simeq\ \mathrm{id}_{\vrpp{2}{X}{r}}$.

If $\alpha$ is a finitely supported measure, then the above construction resides inside finitely supported measures as well and therefore we get the result for $\vrppfin{2}{X}{r}$.

For the the \v{C}ech case, we again map $\cechpp{2}{X}{r}$ onto $X$ using the map $\alpha \mapsto \phi(\alpha):= \pi(m(\alpha))$.
To see this is well-defined, note that as $\alpha \in \cechpp{2}{X}{r}$, there is some $z\in X$ with $F_\alpha(z) = \norm{z - m(\alpha)}^2 + F_\alpha(m(\alpha))< r^2$.
This implies $\norm{z - m(\alpha)}< r \le \tau$ and so $m(\alpha)$ lies inside $U_\tau(X)$.
We then get the result by the same homotopy construction as before.
The finitely supported \v{C}ech case holds similarly.
\end{proof}

\section{The $2$-Vietoris-Rips and $2$-\v{C}ech thickenings of spheres with Euclidean metric}
\label{sec:spheres}

Let $\Sp^n$ be the $n$-dimensional sphere $\Sp^n=\{x\in\R^{n+1}~|~\norm{x}=1\}$, and let $(\Sp^n,\ell_2)$ denote the unit sphere equipped with the Euclidean metric.
In this section we determine the successive homotopy types of the 2-Vietoris--Rips and 2-\v{C}ech metric thickening filtrations $\vrpp{2}{(\Sp^n,\ell_2)}{\smb}$ and $\cechpp{2}{(\Sp^n, \ell_2)}{\smb}$.
We begin with a lemma.

\begin{lem}\label{lem:diam2rad2Sphere}
For any measure $\alpha\in \cP_{(\Sp^n,\ell_2)}$, we have

\[ \diam_2(\alpha) = \left( 2 - 2\cdot \norm{m(\alpha)}^2 \right)^{1/2} \leq \sqrt{2},\]
and
\[ \rad_2(\alpha) = \left( 2 - 2\cdot \norm{m(\alpha)} \right)^{1/2}\leq \sqrt{2}.\]
\end{lem}

Note that since $\|m(\alpha)\|\leq 1$, we have that indeed $\rad_2(\alpha)\leq \diam_2(\alpha)$ and by the same token $\diam_2(\alpha) \leq 2\cdot \rad_2(\alpha)$, in agreement with Proposition~\ref{prop:basic-diamp}.

\begin{proof}[Proof of Lemma~\ref{lem:diam2rad2Sphere}]
We can calculate $\diam_2(\alpha)$ as follows:
\begin{align}\label{eq:diam2-formula}
\diam_2(\alpha)
&= \left( \int_{\Sp^n} \int_{\Sp^n} \norm{x-y}^2\ \alpha(dx)\alpha(dy)\right)^{1/2} \\
&= \left( \int_{\Sp^n} \int_{\Sp^n} 2-2\langle x,y\rangle\ \alpha(dx)\alpha(dy)\right)^{1/2} \\
&= \left( 2 - 2\cdot \left\langle \int_{\Sp^n} x\ \alpha(dx), \int_{\Sp^n}y\ \alpha(dy)\right\rangle\right)^{1/2}\nonumber \\
&= \left( 2 - 2\cdot \norm{m(\alpha)}^2 \right)^{1/2}\nonumber \\
&\leq \sqrt{2}\nonumber.
\end{align}
The calculation for $\rad_2(\alpha)$ gives
\begin{align*}
\rad_2(\alpha) & = \inf_{x\in \Sp^n} \left(\int_{\Sp^n}\norm{x- y}^2\alpha(dy)\right)^{1/2} \\
&= \inf_{x\in \Sp^n} \left( 2- 2 \int_{\Sp^n}\langle x, y\rangle \alpha(dy)\right)^{1/2} \\
&= \inf_{x\in \Sp^n} \left( 2- 2 \langle x, m(\alpha)\rangle \right)^{1/2}  \\
&= \left( 2- 2 \left\langle \frac{m(\alpha)}{\norm{m(\alpha)}}, m(\alpha)\right\rangle \right)^{1/2} \\
&= \left( 2 - 2\cdot \norm{m(\alpha)} \right)^{1/2}   \\
&\leq \sqrt{2}.
\end{align*}
\end{proof}

We now see that for spheres with euclidean metric, both the 2-Vietoris-Rips and the 2-\v{C}ech metric thickenings attain \emph{only two} homotopy types:

\begin{thm}\label{thm:vp2Euclidean}
For any $r\leq \sqrt{2}$, the isometric embeddings $\Sp^n\hookrightarrow\vrpp{2}{(\Sp^n,\ell_2)}{r}$ and $\Sp^n\hookrightarrow\cechpp{2}{(\Sp^n,\ell_2)}{r}$ are homotopy equivalences.
When $r>\sqrt{2}$, the spaces $\vrpp{2}{(\Sp^n,\ell_2)}{r}$ and $\cechpp{2}{(\Sp^n,\ell_2)}{r}$ are contractible.
By restricting to finitely supported measures, we get the same result for $\vrppfin{2}{(\Sp^n,\ell_2)}{r}$ and $\cechppfin{2}{(\Sp^n,\ell_2)}{r}$.
\end{thm}

\begin{remark}
A similar phenomenon is found for the Vietoris--Rips filtration of $(\Sp^n,\ell_\infty)$ where $\vr{(\Sp^n,\ell_\infty)}{r}$ has the homotopy type of $\Sp^n$ when $r\leq \frac{2}{\sqrt{n+1}}$ and becomes contractible when $r>\frac{2}{\sqrt{n+1}}$; see~\cite[Section 7.3]{lim2020vietoris}.
In contrast, the homotopy types of the Vietoris--Rips and \v{C}ech filtrations of geodesic circle include all possible odd-dimensional spheres~\cite{AA-VRS1}.
The first new homotopy type of the $\infty$-Vietoris--Rips metric thickening of the $n$-sphere (with either the geodesic or the Euclidean metric) is known~\cite[Theorem~5.4]{AAF}, but so far only for a single (non-persistent) scale parameter, and only when using the convention $\diam_\infty(\alpha)\le r$ instead of $\diam_\infty(\alpha) < r$.
\end{remark}

We note that Theorem~\ref{thm:vr2_Hausmann} only implies that $\vrpp{2}{(\Sp^n,\ell_2)}{r}\simeq \Sp^n$ for $r\le\tau(\Sp^n)=1$, whereas Theorem~\ref{thm:vp2Euclidean} extends this range to $r\le\sqrt{2}$.

\begin{proof}[Proof of Theorem~\ref{thm:vp2Euclidean}]
From Lemma~\ref{lem:diam2rad2Sphere} we know that $\sqrt{2}$ is the maximal possible $2$-diameter on $\cP_{(\Sp^n,\ell_2)}$.
Therefore, when $r> \sqrt{2}$, the spaces $\vrpp{2}{(\Sp^n,\ell_2)}{r} = \cP_{(\Sp^n,\ell_2)}$ and $\vrppfin{2}{(\Sp^n,\ell_2)}{r} = \cPfin_{(\Sp^n,\ell_2)}$ are both convex and hence contractible.

When $0<r\leq\sqrt{2}$, by Lemma~\ref{lem:diam2rad2Sphere} we get that $m(\alpha)$ is not the origin for any $\alpha\in \cP_{(\Sp^n,\ell_2)}$ (since otherwise we would have $\diam_2(\alpha)=\sqrt{2}$, which is too large).
Therefore $m(\alpha)$ is inside the tubular neighborhood $U_1((\Sp^n,\ell_2))$.
Let $i$ be the inclusion map $(\Sp^n,\ell_2) \hookrightarrow \vrpp{2}{(\Sp^n,\ell_2)}{r}$ mapping points to delta measures, and let $\phi$ be the composition $\pi\circ m \colon \vrpp{2}{(\Sp^n,\ell_2)}{r} \to (\Sp^n,\ell_2)$ which is well-defined as $m(\alpha)$ is not the origin.
Then we naturally have $\phi\circ i = \mathrm{id}_{(\Sp^n,\ell_2)}$.
By applying Lemma~\ref{lem:boundedness_of_rad_and_diam}, the linear family $\alpha_t = (1-t)\alpha + t\cdot \delta_{\phi(\alpha)}$ lies inside $\vrpp{2}{(\Sp^n,\ell_2)}{r}$ and therefore provides the desired homotopy between $i\circ\phi$ and $\mathrm{id}_{\vrpp{2}{(\Sp^n,\ell_2)}{r}}$.
By restricting to the set of finitely supported measures, we get the result for $\vrppfin{2}{(\Sp^n,\ell_2)}{r}$.

The proof is identical for the \v{C}ech case.

\end{proof}

\begin{remark}
The above calculation gives, for any two positive integers $n\neq m$, that
\begin{align*}
\dB(\dgmvr_{n, 2}(\Sp^n,\ell_2), \dgmvr_{n, 2}(\Sp^m,\ell_2)) &= \tfrac{\sqrt{2}}{2} \\
\dB(\dgmcech_{n, 2}(\Sp^n,\ell_2), \dgmcech_{n, 2}(\Sp^m,\ell_2)) &= \tfrac{\sqrt{2}}{2}.
\end{align*}
The stability Theorem~\ref{thm:main} gives
\begin{equation}\label{eq:lb-gh}
\dGH((\Sp^n,\ell_2),(\Sp^m,\ell_2)) \geq \frac{\sqrt{2}}{4}.
\end{equation}
\end{remark}

\begin{remark}
In a more detailed analysis, Proposition~9.13 of~\cite{lim2021gromov} provides the lower bound $\frac{1}{2}$ for $\dGH((\Sp^n,\ell_2),(\Sp^m,\ell_2))$ (when $n\neq m$), which is larger than the lower bound $\frac{\sqrt{2}}{4}$ given in~\eqref{eq:lb-gh}.
As for the geodesic distance case, one can use~\cite[Corollary~9.8~(1)]{lim2021gromov} to obtain $\dGH(\Sp^n,\Sp^m)\geq \arcsin(\sqrt{2}/4)$ from~\eqref{eq:lb-gh}, where $\dGH(\Sp^n,\Sp^m)$ is the Gromov--Hausdorff distance between spheres endowed with their geodesic distances for any $0< m< n$.
A better lower bound is found in~\cite{lim2021gromov}: 
\[
\dGH(\Sp^n,\,\Sp^m)\geq \frac{1}{2}\, \arccos\left(\frac{-1}{m+1}\right) \geq \frac{\pi}{4},
\]
which arises from considerations other than stability. 
The lower bound given by the quantity in the middle is shown to coincide with the exact Gromov--Hausdorff distance between $\Sp^1$ and $\Sp^2$, between $\Sp^1$ and $\Sp^3$, and also between $\Sp^2$ and $\Sp^3$.
\end{remark}

\section{Bounding barcode length via spread}
\label{app:spread}
The \emph{spread} of a metric space is defined by Katz~\cite{katz1983filling}, and used in~\cite[Section 9]{lim2020vietoris} to bound the length of intervals in Vietoris-Rips simplicial complex persistence diagrams.
In this section we identify a measure theoretic version of the notion of spread, and we use it to bound the length of intervals in $p$-Vietoris-Rips and $p$-\v{C}ech metric thickening persistence diagrams.

\begin{defn}
The $p$-spread of a bounded metric space $X$ is defined as
\[\mathrm{spread}_p(X):=\inf_{\alpha_*\in \cPfin_X}\ \sup_{\alpha \in\cP_X}\dWp(\alpha_\ast,\alpha).\]
\end{defn}

Note that for any $1\leq p \leq q \leq \infty$, we have
\[\mathrm{spread}_1(X)\leq \mathrm{spread}_p(X)\leq \mathrm{spread}_q(X)\leq \mathrm{spread}_\infty(X)\leq \rad(X).\]

\begin{remark}
When $p\in [1, \infty)$, the space $\cPfin_X$ is dense in $\cPq{p}{X}$, the set of Radon measures with finite moment of order $p$
(see~\cite[Corollary~3.3.5]{bogachev2018weak}).
Therefore, for a bounded metric space, the $p$-spread is just the radius of the metric space $(\cP_X, \dWp)$.
\end{remark}

\begin{prop}\label{prop:p-spread_general_invariant}
Let $X$ be a bounded metric space and let $\mi$ be an $L$-controlled invariant.
For any $r>0$ and any $\xi> \mathrm{spread}_\infty(X) $, the space $\filt{\cP_X}{\mi^X}{r}$ can be contracted inside of $\filt{\cP_X}{\mi^X}{r+2L\xi}$.
In particular, any homology class of $\filt{\cP_X}{\mi^X}{r}$ will vanish in $\filt{\cP_X}{\mi^X}{r+2L\xi}$.
\end{prop}

\begin{proof}
As $\xi > \mathrm{spread}_\infty(X)$, there is some $\alpha_\xi\in \cPfin_X$ such that for all $\alpha\in \cP_X$ we have $\dWqq{\infty}(\alpha_\xi, \alpha)< \xi$.
Consider the linear homotopy
\[h_t\colon [0,1]\times \filt{\cP_X}{\mi^X}{r}\to \filt{\cP_X}{\mi^X}{r+2L\xi}\]
defined by
\[(t, \alpha)\mapsto (1-t)\,\alpha + t\, \alpha_{\xi}.\]
This gives a homotopy between the inclusion from $\filt{\cP_X}{\mi^X}{r}$ to $\filt{\cP_X}{\mi^X}{r+2L\xi}$ and a constant map, and therefore implies our conclusion so long as the homotopy is well-defined.
It then suffices to show $\mi^X(h_t(\alpha)) < r+ 2L\xi$.
This comes from the stability condition of the invariant with respect to Wasserstein distances and from Lemma~\ref{lem:bound-distance-convex-comb}:
\begin{equation*}
\mi^X(h_t(\alpha)) \leq \mi^X(\alpha) + 2L\cdot\dWqq{\infty}(h_t(\alpha), \alpha) \leq \mi^X(\alpha) + 2L\cdot\dWqq{\infty}(\alpha_\xi, \alpha) < r + 2L\xi.
\end{equation*}
\end{proof}

\begin{remark}
It is not difficult to see that for the $\iqp$ invariants we can improve the bound $\xi>\mathrm{spread}_\infty(X)$ to $\xi>\max\{\mathrm{spread}_p(X),\mathrm{spread}_q(X)\}$.
\end{remark}

We also have the following stronger contractibility conclusion for $\vrp{X}{\smb}$ and $\cechp{X}{\smb}$.

\begin{thm}
Let $X$ be a bounded metric space.
For any $p\in [1, \infty]$ and any $r>\mathrm{spread}_p(X)$, both $\vrp{X}{r}$ and $\cechp{X}{r}$ are contractible.
\end{thm}

\begin{proof}
As $r > \mathrm{spread}_p(X)$, there is some $\alpha_r\in \cPfin_X$ such that for all $\alpha\in \cP_X$ we have $\dWp(\alpha_r, \alpha)< r$.
In particular, this implies that $\diam_p(\alpha_r) = (\int_X F_{\alpha_r,p}^p(x)\,\alpha_r(dx))^{1/p}< r$ and $\rad_p(\alpha_r) = \inf_{x\in X}F_{\alpha_r,p}(x) < r$.
Now, let $\mi$ be either $\diam_p$ or $\rad_p$.
Consider the linear homotopy
\[h_t\colon [0,1]\times \filt{\cP_X}{\mi^X}{r}\to \filt{\cP_X}{\mi^X}{r} \]
defined by
\[(t, \alpha)\mapsto (1-t)\,\alpha + t\, \alpha_{r}.\]
It then suffices to show for all $t\in [0, 1]$, we have $\diam_p(h_t(\alpha)) < r$ and $\rad_p(h_t(\alpha))< r $, so that this linear homotopy from the identity map on $\filt{\cP_X}{\mi^X}{r}$ to the constant map to $\alpha_r$ is well-defined.

In the $\diam_p$ case, we have
\begin{align*}
&\diam_p^p(h_t(\alpha)) \\
=&(1-t)^2\iint_{X\times X}d^p_X(x, x') \alpha(dx)\,\alpha(dx') + 2t\,(1-t)\iint_{X\times X} d^p_X(x, x') \alpha_r(dx)\,\alpha(dx') \\
& \quad \quad \quad + t^2\iint_{X\times X}d^p_X(x, x') \alpha_r(dx)\,\alpha_r(dx') \\
=& (1-t)^2 \diam_p^p(\alpha) + 2t\,(1-t)\int_X \left(\dWp(\alpha_r, \delta_{x'})\right)^p\,\alpha(dx')  \\
& \quad \quad \quad + t^2\int_X\left(\dWp(\alpha_r, \delta_{x'})\right)^p\,\alpha_r(dx')  \\
<& (1-t)^2 r^p+ 2t\,(1-t)r^p +t^2r^p \\
=& r^p.
\end{align*}

In the $\rad_p$ case, as $\rad_p(\alpha)<r$, there exists a point $y\in X$ such that $\dWp(\delta_y,\alpha) < r$.
Thus
\begin{align*}
\rad_p^p(h_t(\alpha)) &=\left(\inf_{x\in X}\dWp(\delta_x,h_t(\alpha))\right)^p \\
&\leq \left(\dWp(\delta_y,h_t(\alpha))\right)^p \\
&= (1- t)\left(\dWp(\delta_y,\alpha)\right)^p + t \left(\dWp(\delta_y,\alpha_r)\right)^p \\
&< (1-t) r^p + t r^p \\
&= r^p.
\end{align*}
This completes the proof.
\end{proof}

The upper bound for lifetime of the persistent homology features in $\filtf{\cP_X}{\mi^X}$ for $L$-controlled invariants is also related to the metric spread defined in~\cite{katz1983filling}.

\begin{defn}
The \emph{metric spread} of a metric space $X$ is defined to be
\[\mathrm{spread}(X):=\inf_{U\subset X,\,|U|< \infty}\max\big(d_{\mathrm{H}}(U, X),\diam(U)\big),\]
where $d_\mathrm{H}$ is the Hausdorff distance.
\end{defn}

Let $\mi$ be an $L$-controlled invariant and let $\max(\mi^X)$ be the maximum of the function $\mi^X$ on $\cP_X$, for $X$ a bounded metric space.
For any $r > \max(\mi^X)$, we have $\filt{\cP_X}{\mi^X}{r}=\cP_X$, and therefore $\filt{\cP_X}{\mi^X}{r}$ is contractible.
Inspired by the definition of the spread, we have the following result.

\begin{lem}\label{lem:bound_barcode_via_metric_spread}
Let $X$ be a metric space, let $\mi$ be an $L$-controlled invariant, and let $U$ be a finite subset of $X$.
Then for any $\xi > \max\left( d_\mathrm{H}(U, X),\, \frac{\max(\mi^U) - r}{2L}\right)$, the space $\filt{\cP_X}{\mi^X}{r}$ is contractible in $\filt{\cP_X}{\mi^X}{r+2L\xi}$.
In particular, any homology class of $\filt{\cP_X}{\mi^X}{r}$ will vanish in $\filt{\cP_X}{\mi^X}{r+2L\xi}$.
\end{lem}

\begin{proof}
As $\xi > d_\mathrm{H}(U, X)$, the balls $\{B(u; \xi)\}_{u\in U}$ form a open covering of $X$.
We choose a partition of unity subordinate to the covering and build the map $\Phi_U \colon \filt{\cP_X}{\mi^X}{r} \to \cP_U$ as in Lemma~\ref{lem:Ptau_cts_via_partition}.
Since $r + 2L\,\xi> \max(\mi^U)$, we know $\cP_U = \filt{\cP_U}{\mi^U}{r+2L\xi}$.
Since the inclusion map $\iota_U \colon U\to X$ is of zero distortion, we have $\mi^X((\iota_U)_\sharp(\beta))\leq \mi^U(\beta) \leq \max(\mi^U)$ for all $\beta \in \cP_U$.
This implies that the contractible set $\cP_U$ is inside $\filt{\cP_X}{\mi^X}{r+2L\xi}$, and so we have the following diagram.
\begin{equation*}
\xymatrix{
\filt{\cP_X}{\mi^X}{r} \ar[rrd]_{\Phi_U} \ar[rr]^{\nu^X_{r, r + 2L\,\xi}}
&& \filt{\cP_X}{\mi^X}{r+2L\xi}\\
&&\cP_U\ar@{^{(}->}[u]_{\iota_U}}
\end{equation*}
As the image of $\iota_U\circ \Phi_U$ maps $\filt{\cP_X}{\mi^X}{r}$ into a contractible subset $\cP_U\subset \filt{\cP_X}{\mi^X}{r+2L\xi}$, it is homotopy equivalent to a constant map.
To obtain this conclusion, it thus suffices to show that $\iota_U\circ \Phi_U$ is homotopy equivalent to the inclusion given by the structure map $\nu^X_{r, r+ 2L\, \xi}$.
Let's consider the linear homotopy
\[h_t\colon [0,1]\times \filt{\cP_X}{\mi^X}{r}\to \filt{\cP_X}{\mi^X}{r+2L\xi}\]
defined by
\[(t, \alpha)\mapsto (1-t)\,\alpha + t\, \Phi_U(\alpha).\]
From the stability property of $\mi$, Lemma~\ref{lem:bound-distance-convex-comb}, and the estimate in Lemma~\ref{lem:Ptau_cts_via_partition}, we have
\begin{align*}
    \mi^X(h_t(\alpha)) & \leq \mi^X(\alpha) + 2L\cdot \dWqq{\infty}(\alpha, h_t(\alpha))\\
    & \leq \mi^X(\alpha) + 2L\cdot \dWqq{\infty}(\alpha, \Phi_U(\alpha)) \\
    & \leq \mi^X(\alpha) +2L\,\xi \\
    & < r + 2L\,\xi.\end{align*}
This shows that the homotopy $h_t$ from $\iota_U\circ \Phi_U$ to $\nu^X_{r, r+ 2L\, \xi}$ is well-defined.
\end{proof}

\begin{remark}
For the 1-controlled $\mi_{q, p}$ invariant on a bounded metric space $X$, the maximum of $\mi_{q, p}^X$ is bounded by $\diam(X)$.
Therefore, for any finite subset $U$ we have
\[
    \max\left(d_\mathrm{H}(U, X), \diam(U)\right) \geq \max\left( d_\mathrm{H}(U, X),\, \frac{\max(\mi_{q, p}^U) - r}{2}\right).
\]
Hence, the previous lemma implies that the lifetime of features in $\filtf{\cP_X}{\iqp^X}$ is bounded by $2\,\mathrm{spread}(X)$.
In the $\diam_p$ case, the factor $2$ can be removed, as we show next, and therefore matches the result in~\cite[Proposition~9.6]{lim2020vietoris} for Vietoris--Rips simplicial complexes.
\end{remark}

\begin{prop}
For any $r>0$, $p\in[1, \infty]$, and $\xi> \mathrm{spread}(X)$, the space $\vrp{X}{r}$ is contractible in $\vrp{X}{r+\xi}$.
In particular, any homology class on $\vrp{X}{r}$ will vanish in $\vrp{X}{r+\xi}$.
\end{prop}

\begin{proof}
As $\xi > \mathrm{spread}(X)$, there exists some $U\subset X$ such that $\diam(U)< \xi$ and $d_\mathrm{H}(U, X)<  \xi$.
As the maximum of $\diam_p^U$ on $\cP_U$ is bounded by $\diam(U)$, the contractible subset $\cP_U$ lies inside $\vrp{X}{r+\xi}$, and therefore $\iota_U\circ \Phi_U$ is homotopy equivalent to a constant map.
Let $h_t$ be the linear homotopy used in Lemma~\ref{lem:bound_barcode_via_metric_spread}.
What is left to get a homotopy equivalence between $\iota_U\circ \Phi_U$ and $\nu^X_{r, r+ \xi}$ is to show $\diam_p(h_t(\alpha)) < r + \xi$.
This comes from the following calculation:
\begin{align*}
&\diam_p^p(h_t(\alpha)) \\
=\, &(1-t)^2\iint_{X\times X}d^p(x, x') \alpha(dx)\,\alpha(dx') + 2t\,(1-t)\iint_{X\times X} d^p(x, x') \alpha(dx)\,\Phi_U(\alpha)(dx')\\
&\quad \quad \quad + t^2\iint_{X\times X}d^p(x, x') \Phi_U(\alpha)(dx)\,\Phi_U(\alpha)(dx')\\
=\, & (1-t)^2 \diam_p^p(\alpha) + 2t\,(1-t)\int_X \left(\dWp(\alpha, \delta_{x'})\right)^p\,\Phi_U(\alpha)(dx')+ t^2\diam_p^p(\Phi_U(\alpha))\\
\leq\, & (1-t)^2 \diam_p^p(\alpha) + 2t\,(1-t)\int_X \left(\dWp(\alpha, \Phi_U(\alpha)) + \dWp(\Phi_U(\alpha), \delta_{x'})\right)^p\,\Phi_U(\alpha)(dx')\\
&\quad \quad \quad + t^2\diam_p^p(\Phi_U(\alpha))\\
\leq\, & (1-t)^2 \diam_p^p(\alpha) + 2t\,(1-t)\big(\xi + \diam_p(\Phi_U(\alpha))\big)^p + t^2\diam_p^p(\Phi_U(\alpha))\\
<\, & (1-t)^2 r^p+ 2t(1-t)(r+ \xi)^p +t^2\xi^p\\
<\, & (r+ \xi)^p.
\end{align*}
\end{proof}

The following result shows the lifetime of features of $\cechp{X}{\smb}$ is also bounded by the metric spread of $X$, $\mathrm{spread}(X)$.

\begin{prop}
For any $r>0$, $p\in[1, \infty]$, and $\xi> \mathrm{spread}(X)$, the space $\cechp{X}{r}$ is contractible in $\cechp{X}{r+\xi}$.
In particular, any homology class on $\cechp{X}{r}$ will vanish in $\cechp{X}{r+\xi}$.
\end{prop}

\begin{proof}
As $\xi > \mathrm{spread}(X)$, there exists some $U\subset X$ such that $\diam(U)< \xi$ and $d_\mathrm{H}(U, X)< \xi$.
The maximum of $\rad_p^U$ on $\cP_U$ is bounded by $\diam(U)$, and the contractible subset $\cP_U$ lies inside $\cechp{X}{r+\xi}$.
This shows $\iota_U\circ \Phi_U$ is homotopy equivalent to a constant map.
Let $h_t = (1-t)\,\alpha + t\, \Phi_U(\alpha)$ be the linear homotopy used in Lemma~\ref{lem:bound_barcode_via_metric_spread}.
What is left is to show is $\rad_p(h_t(\alpha)) < r + \xi$.
By Lemma~\ref{lem:stability_of_rad_p} and Lemma~\ref{lem:bound-distance-convex-comb}, we have
\begin{align*}
\rad_p(h_t(\alpha)) & \leq \rad_p(\alpha) + \dWqq{p}(\alpha, h_t(\alpha)) \\
&\leq r + t^\frac{1}{p} \dWqq{p}(\alpha, \Phi_U(\alpha)) \\
&\leq r + \xi.
\end{align*}
\end{proof}

\section{Conclusion}
\label{sec:conclusion}

Filtrations, i.e.\ increasing sequences of spaces, play a foundational role in applied and computational topology, as they are the input to persistent homology.
To produce a filtration from a metric space $X$, one often considers a Vietoris--Rips or \v{C}ech simplicial complex, with $X$ as its vertex set, as the scale parameter increases.
Since a point in the geometric realization of a simplicial complex is a convex combination of the vertices of the simplex $[x_0,x_1,\ldots,x_k]$ in which it lies, each such point can alternatively be identified with a probability measure: a convex combination of Dirac delta masses $\delta_{x_0}$, $\delta_{x_1}$, \ldots, $\delta_{x_k}$.
We can therefore re-interpret the Vietoris--Rips and \v{C}ech simplicial complex filtrations instead as filtrations in the space of probability measures, which are referred to as the Vietoris--Rips and \v{C}ech \emph{metric thickenings}.
In~\cite{AAF} it is argued that the metric thickenings have nicer properties for some purposes: for example, the inclusion from metric space $X$ into the metric thickening is always an isometry onto its image, whereas an inclusion from metric space $X$ into the vertex set of a simplicial complex is not even continuous unless $X$ is discrete.
In this paper, we prove that these two perspectives are compatible: the $\infty$-Vietoris--Rips (resp.\ \v{C}ech) metric thickening filtration has the same persistent homology as the Vietoris--Rips (resp.\ \v{C}ech) simplicial complex filtration when $X$ is totally bounded.
Therefore, when analyzing these filtrations, one can choose to apply either simplicial techniques (simplicial homology, simplicial collapses, discrete Morse theory) or measure-theoretic techniques (optimal transport, Karcher or Fr\'{e}chet means), whichever is more convenient for the task at hand.

The measure-theoretic perspective motivates new filtrations to build on top of a metric space $X$.
Though the Vietoris--Rips simplicial complex filtration is closely related (at interleaving distance zero) to the metric thickening filtration obtained by looking at sublevelsets in the space of probability measures of the $\infty$-diameter functional, one can instead consider sublevelsets of the $p$-diameter functional for any $1\le p\le \infty$.
The same is true upon replacing Vietoris--Rips with \v{C}ech and replacing $p$-diameter with $p$-radius.
These relaxed $p$-Vietoris--Rips and $p$-\v{C}ech metric thickenings enjoy the same stability results underlying the use of persistent homology: nearby metric spaces produce nearby persistence modules.
The generalization to $p<\infty$ is a useful one: though determining the homotopy types of $\infty$-Vietoris--Rips thickenings of $n$-spheres is a hard open problem, we give a complete description of the homotopy types of $2$-Vietoris--Rips thickenings of $n$-spheres for all $n$.
We also prove a Hausmann-type theorem in the case $p=2$, and ask if the $p<\infty$ metric thickenings may be amenable to study using tools from Morse theory.

More generally, one can consider sublevelsets of any $L$-controlled function on the space of probability measures on $X$.
We prove stability in this much more general context.
This allows one to consider metric thickenings that are tuned to a particular task, perhaps incorporating other geometric notions besides just proximity, such as curvature, centrality, eccentricity, etc.
One can design an $L$-controlled functional to highlight specific features that may be useful for a particular data science task.

We hope these contributions inspire more work on metric thickenings and their relaxations.
We end with some open questions.

\begin{enumerate}

\item For $X$ totally bounded, is the $p=\infty$ metric thickening $\vrpp{\infty}{X}{r}$ homotopy equivalent to the simplicial complex $\vr{X}{r}$, and similarly is $\cechpp{\infty}{X}{r}$ homotopy equivalent to $\cech{X}{r}$?
Note we are using the $<$ convention.

\item For $X$ totally bounded, is $\vrp{X}{r}$ homotopy equivalent to $\vrpfin{X}{r}$, is $\cechp{X}{r}$ homotopy equivalent to $\cechpfin{X}{r}$, and for $\mi$ a controlled invariant is $\filt{\cP_X}{\mi^X}{r}$ homotopy equivalent to $\filt{\cPfin_X}{\mi^X}{r}$?

\item Is there an analogue to the Hausmann-type Theorem~\ref{thm:vr2_Hausmann} which holds for $p\in(2,\infty)$? The case $p=\infty$ was tackled in~\cite{AM}.
In a similar spirit, it seems interesting to explore whether analogous theorems hold when the ambient space is a more general Hadamard space instead of $\R^d$.

\item Can one prove Latschev-type theorems~\cite{Latschev2001} for $p$-metric thickenings?

\item For $p\neq 2$, what are the homotopy types of $p$-Vietoris--Rips and $p$-\v{C}ech thickenings of spheres at all scales?
Is the homotopy connectivity a non-decreasing function of the scale, and if so, how quickly does the homotopy connectivity increase?

\item What are the homotopy types of $p$-Vietoris--Rips and $p$-\v{C}ech metric thickenings of other manifolds, such as ellipses (see~\cite{AAR}), ellipsoids, tori, and projective spaces~\cite{katz1983filling,katz1983filling,katz1991rational,AdamsHeimPeterson}?

\item What versions of Morse theory~\cite{milnor1963morse} can be developed in order to analyze the homotopy types of $p$-metric thickenings of manifolds as the scale increases?
See Appendix~\ref{app:Morse} for some initial ideas in the case of $p$-\v{C}ech thickenings.
For homogeneous spaces such as spheres, versions of Morse--Bott theory~\cite{bott1954nondegenerate,bott1982lectures,bott1988morse} may be needed.

\item For $X$ finite, is $\vrp{X}{r}$ always homotopy equivalent to a subcomplex of the complete simplex on the vertex set $X$?
See Appendix~\ref{app:finite-Cech} for a proof of the \v{C}ech case.

\item For $X$ finite with $n+1$ points, the space $\cP_X$ is an $n$-simplex in $\R^{n+1}$ where coordinates are the weights of a measure at each point.
In this case, $\diam_p$ is a quadratic polynomial on $\R^{n+1}$ and $\rad_p$ is the minimum of $n+1$ linear equations.
Therefore, both $\vrp{X}{r}$ and $\cechp{X}{r}$ are semi-algebraic sets in $\R^{n+1}$.
Can one use linear programming along with the results of Appendix~\ref{app:finite-Cech} to calculate the homology groups of $\cech{X}{r}$, and the work on quadratic semi-algebraic sets~\cite{basu2008computing, basu2008computing2, burgisser2018computing} to calculate the homology groups of $\vrp{X}{r}$?
See also the recent paper~\cite{basu2022persistent}, which provides a singly exponential complexity algorithm for computing the sub-level set persistent homology induced by a polynomial on a semi-algebraic set up to some fixed homological dimension. 

\end{enumerate}

\subsection*{Acknowledgements}

HA would like to thank Florian Frick and \v{Z}iga Virk and FM would like to thank Sunhyuk Lim for helpful conversations. 

FM and QW acknowledge funding from these sources: NSF-DMS-1723003, NSF-CCF-1740761 and NSF-CCF-1839358.

\bibliographystyle{alpha}
\bibliography{VRp}

\appendix

\section{Appendices}

The appendices contain results which are related to but not central to the main thread of the paper.
In Appendix~\ref{app:Metrization_weak_topology} we explain how the $q$-Wasserstein distance metrizes the weak topology for $1\le q<\infty$.
We describe connections to min-type Morse theories in Appendix~\ref{app:Morse}, and ask what can be gained from these connections.
In Appendix~\ref{app:finite-Cech} we show that $p$-\v{C}ech thickenings of finite metric spaces are homotopy equivalent to simplicial complexes with one vertex for each point in the metric space.
We derive the persistent homology diagrams of the $p$-Vietoris--Rips and $p$-\v{C}ech metric thickenings of a family of discrete metric spaces in Appendix~\ref{app:pd_delta_1_space}, and we describe the $0$-dimensional persistent homology of the $p$-Vietoris--Rips and $p$-\v{C}ech metric thickenings of an arbitrary metric space in Appendix~\ref{app:PH0}.
We consider crushings in Appendix~\ref{app:crushings}.
In Appendix~\ref{app:ambient}, we show that the main properties we prove for the (intrinsic) $p$-\v{C}ech metric thickening also hold for the ambient $p$-\v{C}ech metric thickening.

\subsection{Metrization of the weak topology}
\label{app:Metrization_weak_topology}

For $1\le q<\infty$ the $q$-Wasserstein distance metrizes the weak topology, as we explain here for completeness.

\begin{defn}
Let $X$ be a metric space.
The L\'{e}vy-Prokhorov metric on $\cP_X$ is given by the formula
\[
d_{\mathrm{LP}}(\alpha, \beta):= \inf\{\varepsilon>0~|~\alpha(E)\leq \beta(E^\varepsilon) + \varepsilon,\, \beta(E)\leq \alpha(E^\varepsilon) + \varepsilon \text{ for all } E\in \mathfrak{B}(X)\}.
\]
Here $E^\varepsilon = \cup_{x\in E}B_\varepsilon(x)$ is the open $\varepsilon$-neighborhood of $E$ in $\cP_X$, and $\mathfrak{B}(X)$ is the Borel $\sigma$-algebra.
\end{defn}

We state two theorems that we will use\footnote{The first statement is from~\cite{bogachev2018weak}. 
We emphasize that we use $\cP_X$ to denote the set of all Radon probability measures on $X$ whereas the author instead uses $\cP_r(X)$.}

\begin{thm2}[{\cite[Theorem~3.1.4]{bogachev2018weak}}]
The weak topology on the set $\cP_X$ is generated by the L\'{e}vy-Prokhorov metric.\footnote{Theorem 3.1.4 in~\cite{bogachev2018weak} proves a stronger result result, namely that the claim is true within the set of $\tau$-additive probability measures. The restricted version we use follows from the fact that every Radon measure is $\tau$-additive, cf.~\cite[Proposition 7.2.2 (1)]{bogachev2007measure}. }
\end{thm2}

\begin{thm2}[{\cite[Theorem~2]{gibbs2002choosing}}]
On a metric space $X$, for any $\alpha, \beta$ in $\cP_X$, one has
\[(d_{\mathrm{LP}})^2\leq \dWqq{1} \leq (\diam(X)+ 1) \, d_{\mathrm{LP}}.\]
\end{thm2}

\begin{cor}\label{cor:dwq_generate_weak_topology}
On a bounded metric space $X$, for any $q\in [1, \infty)$, the $q$-Wasserstein metric generates the weak topology on $\cP_X$.
\end{cor}

\begin{proof}
From~\cite[Theorem~2]{gibbs2002choosing} and~\cite[Theorem~3.1.4]{bogachev2018weak}, we know $\dWqq{1}$ generates the weak topology.
For other values of $q$, notice that for any coupling $\mu$ between $\alpha$ and $\beta$, we have
\[ \left(\int_X d_X^q(x,x')\, \mu(dx\times dx')\right)^{\frac{1}{q}}\leq \big(\diam(X)\big)^{\frac{q-1}{q}} \left(\int_X d_X(x, x')\, \mu(dx \times dx')\right)^{\frac{1}{q}}.\]
This implies $\dWq \leq \big(\diam(X)\big)^{\frac{q-1}{q}}\big(\dWqq{1} \big)^{\frac{1}{q}}$.
Along with $\dWqq{1}\leq \dWq$, this implies that on a bounded metric space, all $q$-Wasserstein metrics with $q$ finite are equivalent and generate the weak topology on $\cP_X$.
\end{proof}

\subsection{Min-type Morse theory}
\label{app:Morse}

The paper~\cite{baryshnikov2014min} by Baryshnikov, Bubenik, and Kahle studies a Morse theory for min-type functions;
see also~\cite{gershkovich1997morse,bryzgalova1978maximum,matov1982topological}.
The following notation is from~\cite[Section~3.1]{baryshnikov2014min}.
Let $X$ be a compact metric space (called the \emph{parameter space}), let $M$ be a compact smooth manifold perhaps with boundary, and let $f\colon X\times M \to\R$ be a continuous function.
For each $x\in X$, define $f_x\colon M\to \R$ by $f_x(m)=f(x,m)$.
Let $\nabla f_x \colon M \to \R$ be the gradient of $f_x$ with respect to $m$.
We furthermore assume that the function $X \times M\to \R$ defined by $(x,m)\mapsto \nabla f_x(m)$ is continuous.
There is then an analogue of Morse theory, called ``min-type Morse theory'', for the min-type function $\tau\colon M\to \R$ defined by $\tau(m)=\min_{x\in X}f_x(m)$.

Whereas~\cite{baryshnikov2014min} uses min-type Morse theory to study configurations of hard spheres, we instead propose the use of min-type Morse theory to study $p$-\v{C}ech metric thickenings, as follows.
Let $X$ be a \emph{finite} metric space with $n+1$ points.
Define $M=\Delta_n$ to be the $n$-simplex on $n+1$ vertices; a point $m\in \Delta_n$ is given in barycentric coordinates as $m=(m_0,\ldots,m_n)$ with $m_i\ge 0$ and $\sum_i m_i=1$.
For $x\in X$, let $f_x \colon \Delta_n \to \R$ be defined by $f_x(m) = \sum_i m_i d_X^p(x,x_i)$; note this is equal to the $p$-th power of the $p$-Frech\'{e}t function, namely to $F_{\alpha,p}^p(x) = \sum_i m_i d_X^p(x,x_i)$, when $\alpha$ is the measure $\alpha = \sum_i m_i \delta_{x_i}$.
So $f\colon X\times \Delta_n \to\R$ is defined by $f(x,m)=f_x(m)$.
Note that each gradient $\nabla f_x \colon \Delta_n\to \R$ is linear and hence continuous, and so the joint function $X\times \Delta_n\to \R$ given by $(x,m)\mapsto \nabla f_x(m)$ is continuous since $X$ is discrete.
The function $\tau\colon \Delta_n\to \R$ is then defined by $\tau(m)=\min_{x\in X}f_x(m)$; note that this is equal to $\rad_p^p(\alpha) = \inf_{x\in X}F_{\alpha,p}^p(x)$ for $\alpha = \sum_i m_i \delta_{x_i}$.
By Lemma~\ref{lem:fin-prob-simplex} and its proof, we have not only a homeomorphism $\cP_X \cong \Delta_n$, but also a homeomorphism \[\cechp{X}{r}=\rad_p^{-1}\left((-\infty,r)\right)\cong\tau^{-1}\left((-\infty,r^p)\right),\] meaning that the $p$-\v{C}ech metric thickenings are homeomorphic to the sublevel sets of the min-type function $\tau$.
In this setting, we have the additional convenience that each function $f_x$ is affine.

\begin{question}
Can the machinery from~\cite{baryshnikov2014min} be used to prove new results about $p$-\v{C}ech metric thickenings, such as homotopy types?
A first step in this direction might be to use their balanced criterion (which derives from Farkas' lemma) to help identify which points are topological regular points or critical points of $\tau$.
\end{question}

\begin{question}
We have restricted to $X$ finite ($n+1$ points) so that $M=\Delta_n$ will be a manifold with boundary.
Can one build up towards letting $X$ be a manifold, such as a circle or $n$-sphere?
\end{question}

\subsection{Finite p-\v{C}ech metric thickenings are homotopy equivalent to simplicial complexes}
\label{app:finite-Cech}

The $p$-Vietoris--Rips and $p$-\v{C}ech matric thickenings considered in this paper are based on the corresponding simplicial complexes and are closely related to them.
The most direct relationship is given by Lemma~\ref{lem:finite_infty-thickenings_homeomorphic_to_complexes} in the case of finite metric spaces and $p=\infty$.
Further similarities in the case of totally bounded metric spaces and $p = \infty$ are observed in the persistence diagrams, as shown in Corollary~\ref{cor:infty-simplicial-dgms}.
For $p<\infty$, the metric thickenings are not as directly related to the corresponding simplicial complexes.
However, in this appendix, we establish that all $p$-\v{C}ech metric thickenings on finite metric spaces are homotopy equivalent to simplicial complexes (although generally not the corresponding \v{C}ech simplicial complexes).

\begin{thm}\label{thm_pCech_homotopy_equiv_simplicial_complexes}
Let $(X,d_X)$ be a finite metric space with $n+1$ points.
For any $p \in [1, \infty]$ and any ${r>0}$, $\cechp{X}{r}$ is homotopy equivalent to a simplicial complex on $n+1$ vertices, consisting of the simplices contained in the homeomorphic image of $\cechp{X}{r}$ in the standard $n$-simplex.
\end{thm}

\begin{proof}

The result holds for the case $p = \infty$ by Lemma~\ref{lem:finite_infty-thickenings_homeomorphic_to_complexes}, so we will suppose $p \in [1,\infty)$.
We begin with some background notation and observations.
Let $\sigma$ be a simplex in a Euclidean space, let $z_0 \in \sigma$, and let $\sigma_{z_0}$ be the union of the closed faces of $\sigma$ not containing $z_0$ (we always have $\sigma_{z_0} \subseteq \partial \sigma$, and $\sigma_{z_0} = \partial \sigma$ if $z_0$ is in the interior of $\sigma$).
Then there is a continuous function $P:\sigma \setminus \{ z_0 \}\to \sigma_{z_0}$ defined by projecting radially from $z_0$.
Furthermore, if $C \subseteq \sigma$ is convex and contains $z_0$, then for any $z \in \sigma \setminus C$, the line segment connecting $z$ and $P(z)$ is contained in $\sigma \setminus C$, since this lies in the line segment connecting $z_0$ and $P(z)$.
Therefore a linear homotopy shows that $P |_{\sigma \setminus C}: \sigma \setminus C \to \sigma_{z_0} \setminus C$ is a strong deformation retraction.
We will apply such retractions successively to a simplex and its faces as described below.

Let $X = \{ x_0, \dots, x_n \}$.
If $\alpha = \sum_i a_i \delta_{x_i}$, then using the notation for the Fr\'{e}chet function from Section~\ref{sec:background}, we have $F_{\alpha, p}^p(x_j) = \sum_i a_i d_X^p(x_i,x_j)$.
Thus, $\alpha \in \cechp{X}{r}$ if and only if $\sum_i a_i d_X^p(x_i,x_j) < r^p$ for some $j$.
Let $\Delta = \{ (y_0, \dots, y_n) \in \R^{n+1} \mid \sum_i y_i = 1,\, y_i \geq 0 \text{ for all $i$} \}$ be the standard $n$-simplex in $\R^{n+1}$.
By Lemma~\ref{lem:fin-prob-simplex}, $\cechp{X}{r}$ is homeomorphic to
\[
Y = \left\{ (y_0, \dots, y_n) \in \Delta \hspace{.1cm} \bigr\vert \, \sum_i y_i d_X^p(x_i,x_j) < r^p \text{ for some $j$}\right\}.
\]
Equivalently, $Y$ is a sublevel set of the function $\rho \colon \Delta \to \R$ given by 
\[
\rho(y_0, \dots, y_n) = \min_j \left\{ \sum_i y_i d_X^p(x_i,x_j) \right\}.
\]
We note that $Y$ contains the vertices of $\Delta$ by the assumption that $r>0$.
If $Y = \Delta$, then $\cechp{X}{r} \cong \Delta$, and $\Delta$ is homeomorphic to a simplex with vertex set $X$, as required.
If not, then $\Delta \setminus Y$ is convex as it is the intersection of half-spaces and $\Delta$, so we may project radially from any point $z_0 \in \Delta \setminus Y$ as above.
This shows $Y \simeq Y \cap \Delta_{z_0}$, where $\Delta_{z_0}$ is the union of the closed faces of $\Delta$ not containing $z_0$.

Since $Y \cap \Delta_{z_0}$ is contained in the boundary of $\Delta$, we will next verify that we can define retractions within the $(n-1)$-dimensional faces of $\Delta$ contained in $\Delta_{z_0}$ that contain a point not in $Y$.
More generally, we will repeat for successively lower dimensional faces, inductively obtaining a sequence of strong deformation retractions $Y_n \to Y_{n-1} \to \dots \to Y_0$, where $Y_n = Y$.
Each $Y_k$ will consist of all closed faces of $\Delta$ contained in $Y$, along with some subset of the remaining $k$-dimensional closed faces intersected with $Y$.
Thus, $Y_0$ will be a simplicial complex consisting of the closed faces of $\Delta$ contained $Y$ and will be homotopy equivalent to $\cechp{X}{r}$, as required.

\begin{figure}[h]
    \centering
    \includegraphics[width = .9\textwidth]{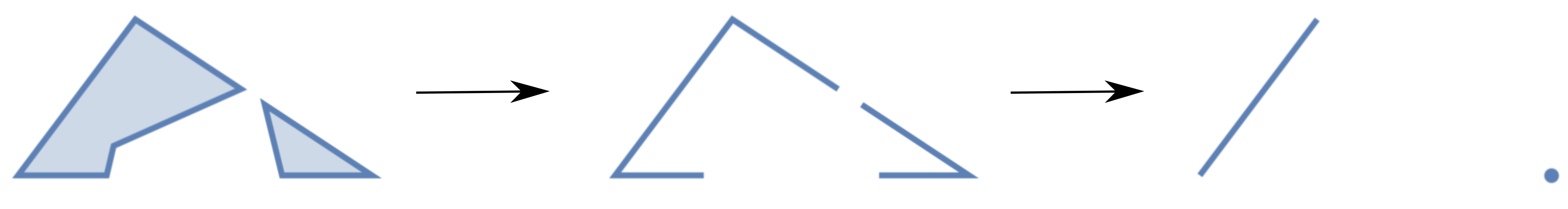}
    \caption{The sequence of deformation retractions used in the proof of Theorem~\ref{thm_pCech_homotopy_equiv_simplicial_complexes} collapses a subset of a simplex to a simplicial complex on its vertices.
    In this example, $n=2$ and the sequence of deformation retractions is $Y_2 \to Y_1 \to Y_0$.}
    \label{figure:Cech_collapse}
\end{figure}

We will use induction, so suppose $Y_k$ meets the description above.
Let $\sigma_1, \dots, \sigma_m$ be those remaining $k$-dimensional faces of $\Delta$ whose intersections with $Y$ are contained in $Y_k$ and that contain a point in their interior that is not in $Y_k$; let $\tau_1, \dots, \tau_{m'}$ be the remaining $k$-dimensional faces whose intersections with $Y$ are contained in $Y_k$ and that have only some boundary point not in $Y_k$.
On each $\sigma_l$, we may choose an interior point not in $Y_k$ and project radially from this point as above.
The homotopy is constant on the boundary, so we may define a homotopy on $Y_k$ that simultaneously retracts each $\sigma_l \cap Y_k$ to $\partial \sigma_l \cap Y_k$ and is constant on all points not in any $\sigma_l$.
Thus, $Y_k$ deformation retracts onto a subset $Y'_k$ consisting of faces of $\Delta$ that are contained in $Y$, along with $\tau_l \cap Y_k$ for all $l$.

Next, choose any point $y \in \partial \tau_1 \setminus Y'_k$.
Note that projecting radially from $y$ as above affects any simplex in $\partial \tau_1$ that contains $y$, so we must extend this to a homotopy $H$ on $Y'_k$ in a way that is consistent on the simplices containing $y$.
For any $\tau_l$ containing $y$, let $H$ be the defined on $\tau_l \cap Y'_k$ by the radial projection from $y$, as above.
For any two $\tau_l$ and $\tau_{l'}$ containing $y$, the two definitions of $H$ on $\tau_l \cap \tau_{l'} \cap Y$ are consistent, as they are both linear homotopies.
We also let $H$ be constant on any point in a face of $\Delta$ not containing $y$: this includes all faces of $\Delta$ contained in $Y$ and all $\tau_l \cap Y$ such that $y \notin \tau_l$.
The definitions are consistent on their overlap, since if $\tau_l$ contains $y$, then radial projection from $y$ is constant on the faces of $\tau_l$ that do not contain $y$.
Therefore $H$ is a well-defined homotopy, which shows that $Y'_k$ deformation retracts onto the subset that excludes the interiors of all $\tau_l$ containing $y$.
We can repeat for the remaining set of $\tau_l$ until all have been retracted.
Composing these deformation retractions, we have shown that $Y_k$ deformation retracts onto a subset $Y_{k-1}$, consisting of faces of $\Delta$ contained in $Y$ along with the intersections of $Y$ with a subset of the remaining $(k-1)$-dimensional faces of $\Delta$ (subsets of the boundaries of those $k$-dimensional simplices that were collapsed).
This completes the inductive step, so we obtain the sequence $Y_n \to Y_{n-1} \to \dots \to Y_0$ of strong deformation retractions as required.
\end{proof}

\begin{cor}\label{Coro:Cech_pSimplicalComplex}
Let $(X,d_X)$ be a finite metric space, let $p \in [1, \infty]$, and for any $r>0$, let $\mathrm{S}(X;r)$ be the simplicial complex from Theorem~\ref{thm_pCech_homotopy_equiv_simplicial_complexes} that is homotopy equivalent to $\cechp{X}{r}$.
These simplicial complexes form a filtration $\mathrm{S}(X;\smb)$, and for any integer $k \geq 0$, $H_k\circ\cechp{X}{\smb}$ and $H_k\circ \mathrm{S}(X;\smb)$ are isomorphic persistence modules.
\end{cor}

\begin{proof}
The $p = \infty$ case again holds by Lemma~\ref{lem:finite_infty-thickenings_homeomorphic_to_complexes}, so let $p \in [1,\infty)$.
We will identify $\cechp{X}{r}$ with its homeomorphic image in the standard simplex.
By Theorem~\ref{thm_pCech_homotopy_equiv_simplicial_complexes}, $\mathrm{S}(X;r)$ consists of the simplices contained in $\cechp{X}{r}$.
So if $r_1 \leq r_2$, the fact that $\cechp{X}{r_1} \subseteq \cechp{X}{r_2}$ implies $\mathrm{S}(X;r_1) \subseteq \mathrm{S}(X;r_2)$, and thus $\mathrm{S}(X;\smb)$ is a filtration of simplicial complexes.

Since in the proof of the theorem we constructed a deformation retraction $\cechp{X}{r} \to \mathrm{S}(X;r)$, the inclusion $\mathrm{S}(X;r) \hookrightarrow \cechp{X}{r}$ is a homotopy equivalence for each $r$.
Therefore the induced maps on homology give isomorphisms $H_k(\mathrm{S}(X;r)) \cong H_k(\cechp{X}{r})$ for each $r$.
Furthermore, the following diagram commutes for all $r_1<r_2$, since all maps are inclusions.

\setlength\mathsurround{0pt}%<<<<<<<<<<<<<<<<<<<<<<<<
\[
\begin{tikzcd}
\cechp{X}{r_1} \arrow[r, hook]                              & \cechp{X}{r_2}                              \\
\mathrm{S}(X;r_1) \arrow[r, hook] \arrow[u, hook] & \mathrm{S}(X;r_2) \arrow[u, hook]
\end{tikzcd}
\]

This shows that the induced maps on homology commute, giving a morphism of persistence modules.
Since each vertical map is an isomorphism, this is an isomorphism of persistence modules.
\end{proof}

These results show that $p$-\v{C}ech persistent homology can be computed, at least in principle.
The complex $\mathrm{S}(X;r)$ in Corollary~\ref{Coro:Cech_pSimplicalComplex} consists of all simplices contained in the $r^p$-sublevel set of the function $\rho$ in the proof of Theorem~\ref{thm_pCech_homotopy_equiv_simplicial_complexes}.
So finding the filtration $\mathrm{S}(X;\smb)$ would require finding the maximum of $\rho$ on each face.
Calculating persistent homology would involve finding the maximum of $\rho$ on each face of the necessary dimensions.

\subsection{Persistence diagrams of $\vrp{Z_{n+1}}{{\smb}}$ and $\cechp{Z_{n+1}}{{\smb}}$}
\label{app:pd_delta_1_space}

In this section, we will calculate the persistence diagrams of the 
$p$-Vietoris--Rips and $p$-\v{C}ech metric thickenings of the metric space $Z_{n+1}$, consisting of $n+1$ points with all interpoint distances equal to one.
Let $\Delta_n$ be the $n$-dimensional simplex on $n+1$ points, and let $\Delta_n^{(k)}$ denote its $k$-skeleton.

\begin{prop}
Let $Z_{n+1}$ be the metric space consisting of $n +1$ points with all interpoint distances equal to $1$.
For $\left(\frac{k}{k+1}\right)^{\frac{1}{p}}< r \leq \left(\frac{k+1}{k+2}\right)^{\frac{1}{p}}$ with $0\le k \le n-1$, we have
\[
\vrp{Z_{n+1}}{r} \simeq \cechp{Z_{n+1}}{r}
    \simeq \Delta_n^{(k)},
\]
and when $r>\left(\frac{n}{n+1}\right)^{\frac{1}{p}}$, both $\vrp{Z_{n+1}}{{r}}$ and $\cechp{Z_{n+1}}{{r}}$ become the $n$-simplex $\cP_{Z_{n+1}}$ which is contractible.
\end{prop}

\begin{proof}
We will use the following observation, for which we omit the proof.
Let $\alpha$ be a measure in $\cP_{Z_{m+1}}$, where $m+1$ is any positive integer.
Then both $\diam_p$ and $\rad_p$ will obtain the maximum $\left(\frac{m}{m+1}\right)^{\frac{1}{p}}$ only at the uniform measure supported on the $m+1$ points of the metric space $Z_{m+1}$, i.e.\ the barycenter of the $m$-simplex $\cP_{Z_{m+1}}$.

Let $k$ be an integer where $0\leq k\leq n-1$.
For any $r$ in $\left(\left(\frac{k}{k+1}\right)^{\frac{1}{p}}, \left(\frac{k+1}{k+2}\right)^{\frac{1}{p}}\,\right]$, we will show both $\vrp{Z_{n+1}}{r}$ and $\cechp{Z_{n+1}}{r}$ can be deformation retracted to the $k$-skeleton of the $n$-simplex $\cP_{Z_{n+1}}$, denoted $\Delta_{n}^{(k)}$.
Since $(\frac{k}{k+1})^{\frac{1}{p}}<r \leq (\frac{k+1}{k+2})^{\frac{1}{p}} $, we know both $\vrp{Z_{n+1}}{r}$ and $\cechp{Z_{n+1}}{r}$ contain the $k$-skeleton of $\Delta_{n}$, but not any higher-dimensional skeleta.
We can then use the radial projection as in Corollary~\ref{Coro:Cech_pSimplicalComplex} to work out the retraction.
From the above observation, the barycenter of any simplex in $\Delta_{n}$ with dimension higher than $k$ is not in $\vrp{Z_{n+1}}{r}$ and $\cechp{Z_{n+1}}{r}$.
We can use these barycenters as the base point for radial projection.
We start from $\Delta_{n}$: if the radial projection based at its barycenter can be restricted to $\vrp{Z_{n+1}}{r}$ or $\cechp{Z_{n+1}}{r}$, then it will retract them onto their intersections with the $(n-1)$-skeleton $(\Delta_{n})^{(n-1)}$.
From the proof of Corollary~\ref{Coro:Cech_pSimplicalComplex}, we know the restriction is well-defined for $\cechp{Z_{n+1}}{r}$.
Here, it is also well-defined for $\vrp{Z_{n+1}}{r}$, because our base point for the radial projection is an interior maximum of the quadratic function $\diam_p^p$, and $\diam_p^p$ will be concave along any line that passes through the base point.
This in turn shows the line segment connecting two points outside of $\vrp{Z_{n+1}}{r}$ is also outside of $\vrp{Z_{n+1}}{r}$.
We can continue the retraction process inductively on $\vrp{Z_{n+1}}{r}\cap \Delta_{n}^{(l)}$ and $\cechp{Z_{n+1}}{r}\cap \Delta_{n}^{(l)}$, where $k<l\leq n-1$, with radial projections based at barycenters of $l$-simplices of $\Delta_{n}$.
This results in a retraction onto $\vrp{Z_{n+1}}{r}\cap \Delta_{n}^{(l-1)}$ and $\cechp{Z_{n+1}}{r}\cap \Delta_{n}^{(l-1)}$, respectively.
Eventually, we will deformation retract onto $\vrp{Z_{n+1}}{r}\cap \Delta_{n}^{(k)}$ and $\cechp{Z_{n+1}}{r}\cap \Delta_{n}^{(k)}$, which are both equal to $\Delta_n^{(k)}$.
\end{proof}

\begin{cor}
Let $n$ be a positive integer and let $Z_{n+1}$ be the metric space consisting of $n+1$ points with interpoint distance equal to $1$.
The persistence diagrams of $\vrp{Z_{n+1}}{{\smb}}$ and $\cechp{Z_{n+1}}{{\smb}}$ are the same and of the following form:
\[
\dgmvr_{k, p}(Z_{n+1}) = \dgmcech_{k, p}(Z_{n+1}) =  \begin{cases}
\left(0, \left(\tfrac{1}{2}\right)^{\tfrac{1}{p}}\,\right)^{\otimes {n}}\oplus (0, \infty)  &\,\text{if $k=0$},\\
\left(\left(\tfrac{k}{k+1}\right)^{\tfrac{1}{p}}, \left(\tfrac{k+1}{k+2}\right)^{\tfrac{1}{p}}\,\right)^{\otimes {{n}\choose{k+1}}}  &\,\text{if $0<k\leq n-1$},\\
\emptyset &\,\text{if $k> n-1$}.
\end{cases}
\]
\end{cor}

The superscripts denote the multiplicity of a point in the persistence diagram.

\begin{proof}
For any integer $k$ with $0< k< n-1$, we know the homology of the $k$-skeleton of an $n$-simplex $\Delta_n$ is given by
\[
H_l(\Delta_n^{(k)}, \mathbb{Z}) = \begin{dcases}
\mathbb{Z} &\, \text{if $l=0$},\\
\mathbb{Z}^{{n}\choose {k+1}} &\, \text{if $l=k$},\\
\emptyset &\, \text{otherwise}.
\end{dcases}
\]
We get the result by combining with the previous result on the homotopy types of $\vrp{Z_{n+1}}{{\smb}}$ and $\cechp{Z_{n+1}}{{\smb}}$.
\end{proof}

\subsection{Zero-dimensional persistent homology of \(\vrp{X}{\smb}\) and \(\cechp{X}{\smb}\)}
\label{app:PH0}

For a finite metric space $X$, we will show that the $0$-dimensional persistent homology of $\vrp{X}{\smb}$ and $\cechp{X}{\smb}$ are the same, and that they both recover the single-linkage clustering up to a constant factor related to $p$.

\begin{lem}\label{lem:birth_of_zero_PH}
For $X$ a finite metric space, the birth time for all intervals in the $0$-dimensional barcodes of $\vrp{X}{r}$ and $\cechp{X}{r}$ is zero.
\end{lem}
\begin{proof}
Since all delta measures $\delta_x$ have $\diam_p$ and $\rad_p$ equal zero, it suffices to show for any measure $\alpha$ in $\vrp{X}{r}$ (or $\cechp{X}{r}$), there is a path in $\vrp{X}{r}$ (or $\cechp{X}{r}$) that connects $\alpha$ with some delta measure.

For $\alpha$ a measure in $\vrp{X}{r}$, we have
\[\int_X \big(\dWp(\alpha, \delta_x)\big)^p\, \alpha(dx) = \big(\diam_p(\alpha)\big)^p < r^p.\]
Therefore, there is some $x_0\in \mathrm{supp}(\alpha)$ such that $\dWp(\alpha, \delta_{x_0}) < r^p$.
We then pick the path $\alpha_t = (1-t)\, \alpha + t\, \delta_{x_0}$ for $t\in[0,1]$.
Now consider
\begin{align*}
\diam_p(\alpha_t) & = \left(\int_X \big(\dWp(\alpha_t, \delta_x)\big)^p\, \alpha_t(dx)\right)^{\frac{1}{p}} \\
& =\left(\int_X \big((1-t)\,\dWp(\alpha, \delta_x) + t\,\dWp(\delta_{x_0}, \delta_x) \big)^p\, \alpha_t(dx)\right)^{\frac{1}{p}} \\
&\leq(1-t)\,\left(\int_X \big(\dWp(\alpha, \delta_x)\big)^p\, \alpha_t(dx)\right)^{\frac{1}{p}}+ t\,\left(\int_X \big(\dWp(\delta_{x_0}, \delta_x )\big)^p\, \alpha_t(dx)\right)^{\frac{1}{p}} \\
&\leq(1-t)\,\bigg((1-t)\,\diam_p^p(\alpha) + t\, \big(\dWp(\alpha, \delta_{x_0}) \big)^p\bigg)^{\frac{1}{p}}+ t\,\bigg((1-t)\,\big(\dWp(\alpha, \delta_{x_0})\big)^p\, \bigg)^{\frac{1}{p}} \\
& < (1-t)\, r + t(1-t)^\frac{1}{p}\, r \\
& < r.
\end{align*}
This shows $\alpha_t\in\vrp{X}{r}$, and the path is continuous by Proposition~\ref{prop:linear_homotopies}.

For $\alpha$ a measure in $\cechp{X}{r}$, then there is some $x_0$ with $\dWp(\alpha, \delta_{x_0}) < r$.
Then $\dWp((1-t)\, \alpha + t\,\delta_{x_0}, \delta_{x_0}) = (1-t)\,\dWp(\alpha, \delta_{x_0}) < r$.
This shows $\alpha_t$ is a continuous path in $\cechp{X}{r}$.
\end{proof}

\begin{prop}
Let $(X, d_X)$ be a finite metric space.
Then the $0$-dimensional persistent module of $\vrp{X}{\smb}$ and $\cechp{X}{\smb}$ are both equal to the $0$-dimensional persistence module of the Vietoris--Rips simplicial complex filtration of the rescaled metric space $(X, \left(\frac{1}{2}\right)^p\, d_X)$.
\end{prop}

\begin{proof}
From Lemma~\ref{lem:birth_of_zero_PH}, we know that all the bars for the $0$-dimensional persistent module are born at zero.
Let $x$ be a point in $(X, d)$, and let $x'$ be a closest point to $x$ in the finite metric space $X$.
Then it suffices to show that$\delta_x$ and $\delta_{x'}$ will only be in the same connected component of $\vrp{(X, d)}{r}$ or $\cechp{(X, d)}{r}$ for any $r>\left(\frac{1}{2}\right)^p\, d_X(x, x')$.

Since the path $\gamma_t = (1-t)\, \delta_x + t\delta_{x'}$ has maximal $\diam_p$ and $\rad_p$ given by $\left(\frac{1}{2}\right)^p\, d_X(x, x')$, we know $\delta_x$ and $\delta_{x'}$ will be in the same connected component of $\vrp{(X, d)}{r}$ or $\cechp{(X, d)}{r}$ for any $r>\left(\frac{1}{2}\right)^p\, d_X(x, x')$.

On the other hand, let $(Y, d_Y)$ be the metric space $(\{x, x'\}, d_X|_{\{x, x'\}})$.
Then, the map $g_{x, x'}\colon (X, d)\to (Y, d_Y)$ that maps $x$ to $x$ and all other points to $x'$ is a $1$-Lipschitz map.
The map $g_{x, x'}$ induces a continuous map $G_{x, x'}\colon \vrp{X}{r} \to \vrp{Y}{r}$ for any $r> 0$ via pushforward.
Note that for any $r\leq \left(\frac{1}{2}\right)^p\, d_X(x, x')$, the image of $\delta_x$ and $\delta_{x'}$ are not in the same connected component of $\vrp{Y}{r}$.
By the continuity of $G_{x, x'}$, the delta masses $\delta_x$ and $\delta_{x'}$ cannot be in the same same connected component of $\vrp{X}{r}$ either.
A similar argument works for $\cechp{X}{r}$ as well.
This concludes our proof.
\end{proof}

\subsection{Crushings and the homotopy type distance}
\label{app:crushings}

A crushing is a particular type of deformation retraction that doesn't increase distances.
In this appendix, we consider the effects of a crushing applied to the metric space underlying a metric thickening.

Let $X$ be a metric space and $A\subseteq X$ is a subspace.
Following~\cite{Hausmann1995}, a \emph{crushing} from $X$ to $A$ is defined as a distance non-increasing strong deformation retraction from $X$ to $A$, that is, a continuous map $H\colon X\times [0, 1]\to X$ satisfying
\begin{enumerate}
    \item $H(x, 1) = x$, $H(x, 0)\in A$, $H(a, t) = a$ if $a\in A$.
    \item $d_X(H(x, t'), H(y, t'))\leq d_X(H(x, t), H(y, t))$ whenever $t'\leq t$.
\end{enumerate}
When this happens we say that \emph{$X$ can be crushed onto $A$.}

In~\cite[Proposition~2.2]{Hausmann1995}, Haussman proves that if $X$ can be crushed onto $A$, then inclusion $\vr{A}{r}\hookrightarrow\vr{X}{r}$ of Vietoris--Rips simplicial complexes is a homotopy equivalence.
A similar result is proven for Vietoris--Rips and \v{C}ech metric thickenings with $p=\infty$ in~\cite[Appendix~B]{AAF}.
We study how more general metric thickenings behave with respect to crushings.
We begin with some preliminaries.

\begin{lem}\label{lem:dLP_under_bounded_Lipschitz}
Let $X$ be a complete separable metric space and let $\phi$ be an $L$-Lipschitz function that is absolutely bounded by a constant $C>0$.
Then for any $\alpha, \beta \in \cP_X$, we have
\[\left\vert \int_X\phi(x)\,\, \alpha\,(dx) - \int_X\phi(x)\,\,\beta\,(dx) \right\vert \leq (L + 2C)\, d_{\mathrm{LP}}(\alpha, \beta).\]
\end{lem}

\begin{proof}
Let $\varepsilon =d_{\mathrm{LP}}(\alpha, \beta)$, then by the expression of the L\'{e}vy-Prokhorov metric given in~\cite[Theorem 3.1.5]{bogachev2018weak}), there is a coupling between $\alpha$ and $\beta$ by $\mu\in \cP_{X\times X}$ such that
\[\mu(\{(x, x')\in X\times X~|~d_X(x, x')> \varepsilon\}) \leq \varepsilon.\]
Let $E$ be the set $\{(x, x')\in X\times X~|~d_X(x, x')> \varepsilon\}$.
Then we have
\begin{align*}
&\left\vert \int_X\phi(x)\, \alpha\,(dx) - \int_X\phi(x)\,\beta\,(dx) \right\vert \\
=\, & \left\vert \iint_{X\times X}\phi(x) - \phi(x')\,\,\mu\,(dx\times dx') \right\vert \\
\leq\, & \iint_{E}\left\vert \phi(x) - \phi(x')\right\vert\,\,\mu\,(dx\times dx') + \iint_{E^c}\left\vert \phi(x) - \phi(x')\right\vert\,\,\mu\,(dx\times dx') \\
\leq\, & 2C\varepsilon + L \iint_{E^c}d_X(x, x')\,\,\mu\,(dx\times dx') \\
=\, & (L + 2C)\,d_{\mathrm{LP}}(\alpha, \beta).
\end{align*}
\end{proof}

\begin{prop}
Let $X$ be a complete separable metric space such that there is a crushing from $X$ onto a subset $A\subset X$.
Then there is an induced deformation retraction from $\cP_X$ onto $\cP_A$.
\end{prop}

\begin{proof}
Let $H$ be a crushing from $X$ to $A$, so $H$ is a continuous map from $X\times [0, 1]$ to $X$.
We use the notation $f_t(x)$ to denote the map $H(x, t)\colon X\to X$ for any fixed $t$ in $[0, 1]$.
Then we can define a map $\tilde{H}\colon \cP_X\times [0,1]\to \cP_X$ via
\[\tilde{H}(\alpha, t) = (f_t)_\sharp(\alpha).\]
For continuity, let $(\alpha_n, t_n)$ be a sequence that converges to $(\alpha_\infty, t_\infty)$.
Then for any bounded continuous function $\gamma(x)$ on $X$, we have
\begin{equation*}
\int_X\gamma(x)\, (f_{t_n})_\sharp(\alpha_n)(dx) = \int_X \gamma\circ f_{t_n}(x)\,\alpha_n(dx).
\end{equation*}
Without loss of generality, we may assume $\gamma$ is $1$-Lipschitz and absolutely bounded by $C>0$.
Therefore, every $\gamma\circ f_{t_n}(x)$ is $1$-Lipschitz and bounded by $C$ for any $n$.
We use the notation $I_{i, j}$ to denote $\int_X \gamma\circ f_{t_i}(x)\,\alpha_j(dx)$.
Then $|I_{i, j}|$ is uniformly bounded by $C$.
Lemma~\ref{lem:dLP_under_bounded_Lipschitz} then shows that for any $i$, we have
\[|I_{i, j} - I_{i,\infty}|\leq (1 + 2C)\, d_{\mathrm{LP}}(\alpha_j, \alpha_\infty).\]
The uniform bound on $|I_{i, \infty}|$ implies $I_{i, \infty}$ converges to $I_{\infty, \infty}$.
For any $\varepsilon > 0$, we can find an $N$ such that for any $n> N$, $|I_{n, \infty} - I_{\infty, \infty}|\leq \frac{\varepsilon}{2}$ and
\[d_{\mathrm{LP}}(\alpha_n, \alpha_\infty) \leq \frac{\varepsilon}{2(1 + 2C)}.\]
Then we have
\[ |I_{n ,n} - I_{\infty, \infty}| \leq |I_{n ,n} - I_{n, \infty}| + |I_{n, \infty}- I_{\infty, \infty}| \leq \varepsilon.\]
This shows $I_{i, i}$ converges to $I_{\infty, \infty}$, and therefore $\tilde{H}$ is continuous.
As $H$ satisfies $H(x, 1) = x$, $H(x, 0)= f_0(x)\in A$, $H(a, t) =f_t(a)= a$ if $a\in A$, we get $\tilde{H}(\alpha, 1) = (\mathrm{id}_{\cP_X})_\sharp(\alpha) = \alpha$, $\tilde{H}(\alpha, 0) = (f_0)_\sharp(\alpha) \in \cP_A$ for any $\alpha\in \cP_X$, and $\tilde{H}(\beta, t) = \beta$ for $\beta \in \cP_A$. Therefore, $\tilde{H}$ is a indeed a strong deformation retraction from $\cP_X$ to $\cP_A$.
\end{proof}

In the same spirit of~\cite[Lemma~B.1]{AAF}, we can apply the above deformation retraction to the sublevel set filtrations of a set of invariants that includes $\iqp$ and $\rad_p$.

\begin{thm}
Let $\mi$ be an invariant such that for any metric spaces $X$ and $Y$ and any $1$-Lipschitz map $f\colon X\to Y$, the induced map on $\cP_X$ does not increase the values of $\,\mi$.
More precisely, for any $\alpha\in \cP_X$, we require
$\mi^Y(f_\sharp(\alpha))\leq \mi^X(\alpha)$.
Then for any complete separable metric space $X$ and any subset $A$ such that $X$ can crushed onto $A$, we have
\[\dHT\big((\cP_X, \mi^X), (\cP_A, \mi^A)\big) = 0.\]
\end{thm}

\begin{proof}
Let $H$ be the crushing from $X$ onto $A$ and let $f$ be $H(x, 0)$.
Then as both $f$ and the inclusion $\iota\colon A\to X$ are $1$-Lipschitz maps, the above condition on the invariant $\mi$ implies
\begin{itemize}
\item $f_\sharp$ is a $0$-map from $(X,\mi^X)$ to $(A,\mi^A)$, and
\item $\iota_\sharp$ is a $0$-map from $(A,\mi^A)$ to $(X,\mi^X)$.
\end{itemize}
Note that $f_\sharp \circ \iota_\sharp$ is the identify map on $\cP_A$ and $\iota_\sharp \circ f_\sharp$ is $0$-homotopic to $\mathrm{id}_{\cP_X}$ with respect to $(\mi^X, \mi^X)$.
Therefore, we have $\dHT\big((\cP_X, \mi^X), (\cP_A, \mi^A)\big) = 0$.
\end{proof}

\begin{remark}
The crushing result could be leveraged to analyze the persistent homology of a space using the persistent homology of embedded submanifolds; see~\cite{virk2021footprints}.
\end{remark}

\subsection{Ambient filtrations from Lipschitz invariants}
\label{app:ambient}

We now show that the main properties for intrinsic $p$-\v{C}ech metric thickenings also hold for ambient $p$-\v{C}ech metric thickenings.

Let $M$ be a metric space and let $X$ be a subset of $M$.
Then any function $\mi^M$ naturally restricts to $\cP_X$ and induces a filtration.
We have the following stability result, given that $\mi^M$ is $C$-Lipschitz with respect to $d_{\mathrm{W}, \infty}^M$ for some $C>0$.

\begin{thm}\label{thm:general_ambient_stability}
Let $M$ be a metric space and let $\mi^M$ be $C$-Lipschitz function on $\cP_X$ with respect to $d_{\mathrm{W}, \infty}^M$.
Then for any two totally bounded subsets $X$ and $Y$ in $M$, we have
\begin{align*}\dHT\big((\cP_X, \mi^M), (\cP_Y, \mi^M)) &\leq C\,d_\mathrm{H}^M(X,Y) \\
\dHT\big((\cPfin_X, \mi^M), (\cPfin_Y, \mi^M)) &\leq C\,d_\mathrm{H}^M(X,Y).
\end{align*}
\end{thm}

Here $d_\mathrm{H}^M$ is the Hausdorff distance in $M$.

\begin{proof}
Overall, the proof follows a construction similar to that in Lemma~\ref{lem:homotopy_scaffording} in the setting of Hausdorff distance.
For $\eta > 2\cdot d_\mathrm{H}^M(X, Y)$ and $\delta>0$, we fix finite $\delta$-nets $U\subset X$ of $X$ and $V\subset Y$ of $Y$.
By the triangle inequality, we have $d_{\mathrm{H}}(U, V)<\frac{\eta}{2}+2\delta$.
For any point $u\in U$, there is a point $v$ in $V$ with $d_M(u, v)< \frac{\eta}{2}+2\delta$.
Through this construction, there is a map $\varphi\colon U\to V$ and $\psi\colon V\to U$ with
\begin{itemize}
	\item $d_M(u, \varphi(u)) < \frac{\eta}{2}+2\delta$ for any $u\in U$.
	\item $d_M(v, \psi(v)) < \frac{\eta}{2}+2\delta$ for any $v\in V$.
	\item $\max\big(\mathrm{dis}(\varphi),\mathrm{dis}(\psi),\mathrm{codis}(\varphi,\psi)\big)\leq \eta +4\delta.$
\end{itemize}
We use the notations $\widehat{\Phi}$, $\widehat{\Psi}$, $H_t^X$ and $H_t^Y$ as in Lemma~\ref{lem:homotopy_scaffording}.
Let $\alpha$ be a measure in $\cP_X$ and $\beta$ a measure in $\cP_Y$.
The last item implies that the following bound from Lemma~\ref{lem:homotopy_scaffording} still holds here:
we have
$d_{\mathrm{W}, \infty}^M(H^X_t(\alpha), \alpha)< \eta + 6\delta$ and $d_{\mathrm{W}, \infty}^M(H^Y_t(\beta), \beta)< \eta + 6\delta$.
Similar to the proof of Theorem~\ref{thm:general_stability}, it suffices to show
\begin{itemize}
	\item the map $\widehat{\Phi}\colon \cP_X\to \cP_Y$ is a $(\frac{\eta}{2} + 3\delta)\,C$--map from $(\cP_X, \mi^M)$ to $(\cP_Y, \mi^M)$,
	\item the map $\widehat{\Psi}\colon \cP_Y\to \cP_X$ is a $(\frac{\eta}{2} + 3\delta)\,C$--map from $(\cP_Y, \mi^M)$ to $(\cP_X, \mi^M)$,
	\item the map $\widehat{\Psi}\circ\widehat{\Phi}\colon \cP_X\to \cP_X$ is $(\eta + 6\delta)\,C$--homotopic to $\mathrm{id}_{\cP_X}$ with respect to $( \mi^M, \mi^M)$,
	\item the map $\widehat{\Phi}\circ\widehat{\Psi}\colon \cP_X\to \cP_X$ is $(\eta + 6\delta)\,C$--homotopic to $\mathrm{id}_{\cP_Y}$ with respect to $( \mi^M, \mi^M)$.
\end{itemize}
We will only show the first and the third items; the rest can be proved similarly.
For the first item, using the fact that $\mi^M$ is $C$-Lipschitz with respect to $d_{\mathrm{W}, \infty}^M$ and the estimate in Lemma~\ref{lem:Ptau_cts_via_partition}, we get
\begin{align*}
\mi^M(\widehat{\Phi})(\alpha) & = \mi^M(\varphi_\sharp(\Phi_U(\alpha)))\\
& \leq \mi^M(\Phi_U(\alpha)) + d_{\mathrm{W}, \infty}^M(\varphi_\sharp(\Phi_U(\alpha)), \Phi_U(\alpha))\,C\\
& \leq \mi^M(\Phi_U(\alpha)) + \left(\frac{\eta}{2} + 2\delta\right)\,C\\
&\leq \mi^M(\alpha) +d_{\mathrm{W}, \infty}^M(\alpha, \Phi_U(\alpha))\,C + \left(\frac{\eta}{2} + 2\delta\right)\,C\\
&\leq \mi^M(\alpha) +\left(\frac{\eta}{2} + 3\delta\right)\,C.
\end{align*}
For the third item, by the inequality $d_{\mathrm{W}, \infty}^M(H_t^X(\alpha), \alpha)< \eta + 6\delta$,
we have
\[\mi^M(H_t^X(\alpha)) \leq \mi^M(\alpha) + (\eta + 6 \delta)\,C.\]
\end{proof}

An interesting case is the ambient $p$-radius, which leads to $p$--ambient \v{C}ech filtrations.
Let $X$ be a bounded subset in a metric space $M$.
For any $\alpha\in\cP_X$, we define the \emph{$p$--ambient radius} of $\alpha$ to be
 \[\rad_p^M(\alpha):=\inf_{m\in M}F_{\alpha,p}(m) = \inf_{m\in M}\dWp(\delta_m,\alpha).\]

\begin{defn}[$p$-ambient \v{C}ech filtration]
\label{defn:p-ambient-cech}
Let $X\subseteq M$ be metric spaces.
For each $r>0$ and $p\in[1,\infty]$, let the \emph{$p$-ambient \v{C}ech metric thickening at scale $r$} be
\[\acechp{X}{M}{r}:=\{\alpha\in\cP_X~|~\rad_p^M(\alpha)<r\}.\]
Similarly, the $p$-ambient \v{C}ech metric thickening at scale $r$ with finite support is defined as
\[\acechpfin{X}{M}{r}:=\{\alpha\in\cPfin_X~|~\rad_p^M(\alpha)<r\}.\]
\end{defn}

Therefore, as a special instance of Theorem~\ref{thm:general_ambient_stability}, we have
\begin{thm}\label{thm:stability_ambient_cech}
Let $X$ and $Y$ be two totally bounded spaces sitting inside a metric space $M$.
We have
\small{\begin{align*}
\di\big(H_k\circ\acechpf{X}{M},H_k\circ\acechpf{Y}{M}\big) &\leq \dHT\big((\cP_X, \rad_p^M), (\cP_Y, \rad_p^M)) \leq d_\mathrm{H}^M(X,Y) \\
\di\big(H_k\circ\acechpffin{X}{M},H_k\circ\acechpffin{Y}{M}\big) &\leq \dHT\big((\cPfin_X, \rad_p^M), (\cPfin_Y, \rad_p^M)) \leq d_\mathrm{H}^M(X,Y).
\end{align*}}
\end{thm}

\begin{proof}
According to Lemma~\ref{lem:stability_of_rad_p}, $\rad_p^M$ is $1$-Lipschitz with respect to $d_{\mathrm{W}, \infty}^M$.
We apply Theorem~\ref{thm:general_ambient_stability} to get the inequality on the right.
The inequality on the left follows from Lemma~\ref{lem:dht_bound_interleaving}.
\end{proof}

\end{document}